\documentclass[12pt,a4paper,nosumlimits]{amsart}
\usepackage{amssymb,amsxtra,eucal}
\usepackage[cmtip,arrow]{xy}
\usepackage{pb-diagram,pb-xy}
\usepackage{hyperref}

\calclayout

\newcommand{\+}{\nobreakdash-}
\renewcommand{\:}{\colon}
\renewcommand{\;}{,\medspace}
\renewcommand{\.}{\text{$\mskip .5\thinmuskip$}}

\renewcommand{\ge}{\geqslant}

\DeclareMathOperator{\Hom}{Hom}
\DeclareMathOperator{\End}{End}
\DeclareMathOperator{\Ext}{Ext}
\DeclareMathOperator{\Tor}{Tor}
\DeclareMathOperator{\Cohom}{Cohom}

\DeclareMathOperator{\GL}{GL}
\DeclareMathOperator{\SL}{SL}
\DeclareMathOperator{\Ad}{Ad}

\DeclareMathOperator{\gl}{\mathfrak{gl}}
\DeclareMathOperator{\ad}{\mathfrak{ad}}

\newcommand{\Vir}{\mathbb{V}\mathrm{ir}}

\newcommand{\ot}{\otimes}
\newcommand{\oc}{\mathbin{\text{\smaller$\square$}}}

\newcommand{\ocn}{\odot}
\newcommand{\Ocn}{\circledcirc}

\newcommand{\rarrow}{\longrightarrow}
\newcommand{\larrow}{\longleftarrow}
\newcommand{\birarrow}{\rightrightarrows}
\newcommand{\eps}{\varepsilon}
\newcommand{\kap}{\varkappa}

\newcommand{\lrarrow}{\.\relbar\joinrel\relbar\joinrel\rightarrow\.}

\newcommand{\comp}{\sphat\,}
\newcommand{\dual}{\spcheck}
\newcommand{\til}{\sptilde}
\newcommand{\wot}{\mathbin{\overset{\leftarrow}
                        {\vphantom{o}\smash\ot}}}
\newcommand{\eot}{\overset{\leftarrow}{\ }}
\newcommand{\tim}{{\rightthreetimes}}

\newcommand{\id}{{\mathrm{id}}}
\newcommand{\ev}{{\mathrm{ev}}}
\newcommand{\rop}{{\mathrm{op}}}

\let\SS\S

\renewcommand{\S}{{\boldsymbol{\mathcal S}}}
\newcommand{\bL}{{\boldsymbol{\mathcal L}}}
\newcommand{\bM}{{\boldsymbol{\mathcal M}}}
\newcommand{\bN}{{\boldsymbol{\mathcal N}}}
\newcommand{\bcJ}{{\boldsymbol{\mathcal J}}}
\newcommand{\bP}{{\boldsymbol{\mathfrak P}}}
\newcommand{\bQ}{{\boldsymbol{\mathfrak Q}}}
\newcommand{\bF}{{\boldsymbol{\mathfrak F}}}

\newcommand{\C}{{\mathcal C}}
\newcommand{\D}{{\mathcal D}}
\newcommand{\cE}{{\mathcal E}}
\newcommand{\cJ}{{\mathcal J}}
\newcommand{\K}{{\mathcal K}}
\renewcommand{\L}{{\mathcal L}}
\newcommand{\M}{{\mathcal M}}
\newcommand{\N}{{\mathcal N}}

\renewcommand{\P}{{\mathfrak P}}
\newcommand{\Q}{{\mathfrak Q}}
\newcommand{\R}{{\mathfrak R}}
\newcommand{\T}{{\mathfrak T}}
\newcommand{\E}{{\mathfrak E}}
\newcommand{\F}{{\mathfrak F}}
\newcommand{\I}{{\mathfrak I}}
\newcommand{\J}{{\mathfrak J}}
\newcommand{\gK}{{\mathfrak K}}
\newcommand{\gL}{{\mathfrak L}}
\newcommand{\gG}{{\mathfrak G}}
\newcommand{\gH}{{\mathfrak H}}
\newcommand{\gA}{{\mathfrak A}}
\newcommand{\gS}{{\mathfrak S}}

\newcommand{\m}{{\mathfrak m}}
\newcommand{\g}{{\mathfrak g}}
\newcommand{\h}{{\mathfrak h}}

\newcommand{\be}{{\mathbf e}}
\newcommand{\bm}{{\mathbf m}}
\newcommand{\bn}{{\mathbf n}}
\newcommand{\bp}{{\mathbf p}}

\newcommand{\vect}{{\operatorname{\mathsf{--vect}}}}
\newcommand{\modl}{{\operatorname{\mathsf{--mod}}}}
\newcommand{\modr}{{\operatorname{\mathsf{mod--}}}}
\newcommand{\bimod}{{\operatorname{\mathsf{--mod--}}}}
\newcommand{\discr}{{\operatorname{\mathsf{--discr}}}}
\newcommand{\comodl}{{\operatorname{\mathsf{--comod}}}}
\newcommand{\comodr}{{\operatorname{\mathsf{comod--}}}}
\newcommand{\bicomod}{{\operatorname{\mathsf{--comod--}}}}
\newcommand{\contra}{{\operatorname{\mathsf{--contra}}}}
\newcommand{\simodl}{{\operatorname{\mathsf{--simod}}}}
\newcommand{\sicntr}{{\operatorname{\mathsf{--sicntr}}}}

\newcommand{\sop}{{\mathsf{op}}}
\newcommand{\co}{{\mathsf{co}}}
\newcommand{\ctr}{{\mathsf{ctr}}}
\newcommand{\si}{{\mathsf{si}}}
\newcommand{\inj}{{\mathsf{inj}}}
\newcommand{\proj}{{\mathsf{proj}}}

\newcommand{\dinj}{{\operatorname{\mathsf{--inj}}}}
\newcommand{\dproj}{{\operatorname{\mathsf{--proj}}}}

\newcommand{\Ab}{{\mathsf{Ab}}}
\newcommand{\Sets}{{\mathsf{Sets}}}
\DeclareMathOperator{\Add}{\mathsf{Add}}

\newcommand{\sA}{{\mathsf A}}
\newcommand{\sB}{{\mathsf B}}
\newcommand{\sD}{{\mathsf D}}
\newcommand{\sO}{{\mathsf O}}
\newcommand{\sE}{{\mathsf E}}
\newcommand{\sM}{{\mathsf M}}
\newcommand{\sN}{{\mathsf N}}
\newcommand{\sK}{{\mathsf K}}
\newcommand{\sV}{{\mathsf V}}

\newcommand{\Z}{{\mathbb Z}}

\theoremstyle{plain}
\newtheorem*{thm}{Theorem}
\newtheorem*{thm1}{Theorem 1}
\newtheorem*{thm2}{Theorem 2}
\newtheorem*{prop}{Proposition}
\newtheorem*{prop1}{Proposition 1}
\newtheorem*{prop2}{Proposition 2}
\newtheorem*{prop3}{Proposition 3}
\newtheorem*{lem}{Lemma}
\newtheorem*{lem1}{Lemma 1}
\newtheorem*{lem2}{Lemma 2}
\newtheorem*{lem3}{Lemma 3}
\newtheorem*{cor}{Corollary}
\theoremstyle{definition}
\newtheorem*{rem}{Remark}
\newtheorem*{ex}{Example}
\newtheorem*{ex1}{Example 1}
\newtheorem*{ex2}{Example 2}

\newcommand{\Section}[1]{\bigskip\addtocontents{toc}{\smallskip}
\section{#1}\addtocontents{toc}{\protect\nopagebreak\smallskip}\medskip}
\newcommand{\Subsection}[1]{\par\medskip\noindent #1. }

\begin{document}

\title{Contramodules}

\author{Leonid Positselski}

\address{Institute of Mathematics of the Czech Academy of Sciences,
\v Zitn\'a~25, 115~67 Prague~1, Czech Republic; and
\newline\indent Laboratory of Algebra and Number Theory, Institute for
Information Transmission Problems, Moscow 127051, Russia}

\email{positselski@math.cas.cz}

\begin{abstract}
 Contramodules are module-like algebraic structures endowed with
infinite summation (or, occasionally, integration) operations
satisfying natural axioms.
 Introduced originally by Eilenberg and Moore in 1965 in the case of
coalgebras over commutative rings, contramodules experience a small
renaissance now after being all but forgotten for three decades between
1970--2000.
 Here we present a review of various definitions and results related to
contramodules (drawing mostly from our monographs, papers, and
preprints~\cite{Psemi,Pkoszul,Pweak,Pcosh,PR,PS1,Pper,Pcoun})---%
including contramodules over corings, topological associative rings, topological Lie algebras and topological groups, semicontramodules
over semialgebras, and a ``contra version'' of
the Bernstein--Gelfand--Gelfand category~$\sO$.
 Several underived manifestations of the comodule-contramodule
correspondence phenomenon are discussed.
\end{abstract}

\maketitle

\tableofcontents

\setcounter{section}{-1}
\section{Introduction}\medskip

\Subsection{0.0}
 Comodules over coalgebras or corings are familiar to many
algebraists.
 Being asked about the natural ways to assign an abelian category
to a coalgebra over a field, most people would probably mention
the left comodules and the right comodules.
 This is indeed a good answer in the case of module categories
\emph{over a ring}, where considering the left modules or the right
modules exhausts the basic possibilities.
 But the ``left or right comodules'' answer is strikingly
\emph{incomplete}, for in fact there are \emph{four} such
abelian categories.
 In addition to the left and right \emph{comodules}, there are also
the left and right \emph{contramodules}, which are no less basic, and
very much analogous, or rather \emph{dual-analogous} to (though
different from) the comodules.

 Contramodules were introduced, \emph{on par with} comodules, in
the classical 1965 AMS Memoir of Eilenberg and Moore~\cite{EM}, but
little attention was paid.
 The 2003 monograph~\cite{BW}, which was supposed to contain state of
the art on corings and comodules at the time, never mentioned
contramodules.
 As it was noticed in the presentation~\cite{Bsli}, near the end of
2000's decade there still existed only \emph{three} papers featuring
contramodules that a MathSciNet search would bring: in addition to
Eilenberg and Moore's original memoir, there were a 1965
paper~\cite{Vaz} by V\'azquez Garc\'\i a (in Spanish) and
a rather remarkable 1970 paper of Barr~\cite{Bar}.
 The next mention of contramodules in any kind of mathematical
literature that the present author is aware of comes in his own 
letters~\cite{Plet}, written (in transliterated Russian) in 2000
and~2002.

 The 2000--2002 letters were eventually noticed by two groups of
authors~\cite{GK2,Brz} and one of them got interested specifically
in contramodules, so the number of relevant MathSciNet search hits
grew a little by now (see, e.~g., \cite{BBW} and~\cite{Wis}).
 In the meantime, the present author's ideas about contramodules
and the co-contra correspondence materialized in a sequence of books,
papers, and preprints~\cite{Psemi,Pkoszul,Pweak,Pcosh,Pmgm,%
PR,Psm,Pper,PS1,PS2,Pproperf,Pcoun,PS3,BPS,Pdc}; there are also
presentations~\cite{Psli,Psli2,Psli3}.
 Still we feel that it may be difficult for a researcher or a student
to navigate this corps of work without additional guidance.
 The present paper is intended to provide such guidance, including both
an accessible exposition of the basics and an overview of some of
the more advanced topics. 

\Subsection{0.1}
 A \emph{coring} may be informally defined as a ``coalgebra over
a noncommutative ring'' (or more precisely, a coalgebra object in
the tensor category of bimodules over a ring).
 Eilenberg and Moore's original definition of contramodules~\cite{EM}
was formulated in the generality of coalgebras over commutative rings
(i.~e., coalgebra objects in the tensor category of modules), but
the generalization to corings is straightforward.
 So a comodule over a coring can be described as ``a comodule along
the coalgebra variables in the coring and a module along the ring
variables''; a contramodule over a coring is ``a contramodule along
the coalgebra variables and a module along the ring variables''.

 Another option is to consider ``algebras over coalgebras'' (or more
precisely, algebra objects in the tensor categories of bicomodules);
these are what we call \emph{semialgebras}.
 The corresponding module objects are called the \emph{semimodules}
and the \emph{semicontramodules}.
 Once again, a semimodule is ``a module along the algebra variables
in the semialgebra and a comodule along the coalgebra variables'';
a semicontramodule is ``a module along the algebra variables and
a contramodule along the coalgebra variables''.

 In the maximal natural generality achieved in
the monograph~\cite{Psemi}, one considers three-story towers of
``algebras over coalgebras over rings'', or \emph{semialgebras over
corings}.
 These still have four module categories attached to them, namely,
the left and right semimodules and the left and right semicontramodules.
 That is the generality level in which the principal results of
the main body of the book~\cite{Psemi} are obtained.

 There are many more ``comodule-like'' abelian categories in algebra
than just comodules over corings or semimodules over semialgebras,
though.
 Generally, just about every class of ``discrete'', ``smooth'', or
``torsion'' modules can be viewed as that of comodules ``along
a part of the variables'' in one sense or another.
 Every such module category is typically accompanied by a much less
familiar, but no less interesting, abelian category of contramodules.
 Hence one comes to the definitions of contramodules over topological
rings and topological Lie algebras.

\Subsection{0.2}
 Generally, contramodules are modules with \emph{infinite summation
operations}, understood algebraically as operations of infinite
arity subjected to natural axioms.
 Contramodules feel like being in some sense ``complete'', but carry
no underlying topologies on them.
 Indeed, simple counterexamples~\cite{Sim,Yek1,Psemi,Pcta} show that
contramodule infinite summation operations \emph{cannot} be interpreted
as any kind of limit of finite partial sums (for all the finite partial
sums of a particular series can vanish in a contramodule while
the infinite sum does not).

 Comodule categories typically have exact functors of filtered
inductive limits and enough injective objects, but nonexact
functors of infinite product and no projectives.
 Contramodule categories have exact functors of infinite product,
and typically enough projective objects, but nonexact functors of
infinite direct sum and no injectives.
 The historical obscurity/neglect of contramodules seems to be
the reason why many people believe that projective objects are much
less common than injective ones in ``naturally appearing'' abelian
categories.

\Subsection{0.3}
 On the other hand, there is a remarkably simple case of
contramodules over the adic completion of a Noetherian ring,
where the forgetful functor from contramodules to modules is
fully faithful, so the contramodule infinite summation operation
can be \emph{recovered} from the conventional module structure.
 Moreover, there are simple descriptions of the essential image
of the fully faithful forgetful functor and the recovery procedure.
 In this setting, there is a different stream of literature,
going back to the 1959 paper by Harrison~\cite{Har}, where
contramodules were known and studied under different names
(and neither the connection with the Eilenberg--Moore definition,
\emph{nor} the existence of the infinite summation operations were
apparently ever noticed).
 The key modern term in this connection is
the \emph{MGM} (Matlis--Greenlees--May)
\emph{duality}~\cite{Mat2,DG,PSY,Pmgm,Pcta,PMat,BP2}.

 So (what we would call) \emph{projective contramodules} over
the ring of $l$\+adic integers $\Z_l$ were studied in~\cite{Har} in
connection with the classification of (what Harrison called)
\emph{co\+torsion abelian groups}.
 A definitive result in this direction was obtained by Enochs
in~\cite{En}, where (what are since known as) \emph{flat cotorsion
modules} over a Noetherian commutative ring were classified in
terms of (what we call) projective contramodules over complete
Noetherian local rings (see also~\cite[Theorem~1.3.8]{Pcosh}).
 The argument in~\cite{En} was based on Matlis' classification
of injective modules~\cite{Mat1}.
 An equivalence between the categories of (what we would call)
injective discrete modules and projective contramodules
over $\Z_l$ was also noticed in~\cite{Har}.

 As to the arbitrary (not necessarily projective) contramodules
over $\Z_l$, these were studied under the name of
\emph{Ext-$l$-complete} abelian groups by Bousfield and
Kan~\cite{BK} and as \emph{weakly $l$\+complete}
abelian groups by Jannsen~\cite{Jan}.
 Finally, contramodules over the adic completions of Noetherian (and
certain other) rings became known as \emph{cohomologically complete}
modules in the papers of Yekutieli et al.~\cite{PSY,PSY2,Yek2}.
 These names are derived from reflection over the basic fact
that contramodules over $\Z_l$ and other adic completions are
actually \emph{always} adically complete, but
\emph{not} necessarily adically separated (as
the above-mentioned counterexamples show).
 Partially extending Enochs' result, over a Noetherian commutative ring
of Krull dimension~$1$ all cotorsion modules can be described
in terms of divisible modules and arbitrary contramodules over
the completions of the ring at its maximal ideals~\cite{Pcta}.

\Subsection{0.4}
 In the author's own research, contramodules first appeared as
a necessary ingredient for developing the \emph{semi-infinite
cohomology theory} of associative algebraic structures~\cite{Plet},
and were subsequently studied in connection with the phenomenon of 
\emph{comodule-contramodule correspondence}~\cite{Psli}.
 The latter means \emph{covariant} equivalences between appropriate
categories of comodules and contramodules.
 The simplest example is the natural equivalence between the additive
categories of injective left comodules and projective left contramodules
over a coalgebra~$\C$ over a field~$k$.
 Attempting to extend this equivalence to \emph{complexes} of left
$\C$\+comodules and left $\C$\+contramodules using complexes of
injective comodules and projective contramodules as resolutions,
one discovers that unbounded acyclic complexes of contramodules are
sometimes assigned to irreducible comodules and vice versa.

 The same problem occurs in the more complicated situation of
the correspondence between complexes of left semimodules and left
semicontramodules over a semialgebra $\S$ over~$\C$
\cite{FF0,FF,RCW,Psemi}.
 Hence the \emph{derived} co-contra correspondence is, generally
speaking, an equivalence between \emph{exotic}, rather than
conventional, derived categories.
 The \emph{coderived category} of $\C$\+comodules is equivalent to
the homotopy category of complexes of injective comodules, and
similarly, the \emph{contraderived category} of $\C$\+contramodules
is equivalent to the homotopy category of projective
contramodules~\cite{Pkoszul}.
 So the coderived category of left $\C$\+comodules and the contraderived
category of left $\C$\+contramodules are naturally equivalent to
each other, $\sD^\co(\C\comodl)\simeq\sD^\ctr(\C\contra)$
\cite[Sections~0.2.6\+-7]{Psemi}.

 This phenomenon of equivalence between ``derived categories of
the second kind'' is reproduced in a situation not involving comodules
or contramodules in the papers~\cite{Jor,Kra,IK,Sto,Pfp}, where
the homotopy categories of unbounded complexes of projective or
injective modules over a ring are studied and an equivalence between
them is sometimes obtained.
 An extension of this theory to quasi-coherent sheaves on nonaffine
schemes was developed in the papers~\cite{Neem,Murf,Psing}; and
an even more advanced version involving \emph{contraherent cosheaves}
was suggested in~\cite[Section~5.7]{Pcosh}.

\Subsection{0.5}
 In the relative situation of semimodules and semicontramodules over
a semialgebra $\S$ over a coalgebra $\C$, the \emph{derived
semimodule-semicontramodule correspondence} is an equivalence between
the \emph{semi}(co)\emph{derived category} of left $\S$\+semimodules and
the \emph{semi}(contra)\emph{derived category} of left
$\S$\+semicontramodules,
$$
 \sD^\si(\S\simodl)\simeq\sD^\si(\S\sicntr).
$$
 The former is a ``mixture of the coderived category along the variables
from $\C$ and the conventional derived category along the variables from
$\S$ relative to $\C$'' , while the latter is a ``mixture of
the contraderived category in the direction of $\C$ and the derived
category in the direction of $\S$ relative to~$\C$''
\cite[Corollary and Remark~D.3.1]{Psemi}.
 A version of the derived semico-semicontra correspondence reproduced
in a situation not involving contramodules can be found
in~\cite[Section~5]{Pfp}.

 On the other hand, the coderived category of left comodules and
the contraderived category of left contramodules over a \emph{coring}
$\C$ over a ring $A$ are equivalent when the ring $A$ has finite
homological dimension (so the coderived and contraderived categories
of $A$\+modules are indistinguishable from their derived category).
 In other words, the coderived category of comodules and
the contraderived category of contramodules are equivalent in
the relative situation provided that the homological dimension
``along the ring variables'' is finite (when it is not, one needs
a dualizing complex along the ring variables to be chosen).

 Similarly, the conventional derived categories of comodules and
contramodules may be equivalent in a relative situation mixing
the ring and coalgebra variables when the homological dimension
``along the coalgebra variables'' is finite.
 This includes, e.~g., the case of quasi-compact semi-separated schemes,
which are glued from the affine pieces by ``a gluing procedure of
finite homological dimension'' (not exceeding the number of the pieces).
 The related version of derived co-contra correspondence for
quasi-coherent sheaves and contraherent cosheaves was developed under
the name of the ``na\"\i ve co-contra correspondence''
in~\cite[Chapter~4]{Pcosh}.

 Furthermore, an affine Noetherian formal scheme is cut out from its
ambient Noetherian scheme by a ``cutting out procedure of finite
homological dimension''  \cite[Corollaries~4.28 and~5.27]{PSY}.
 This can be roughly explained by noticing that the formal completion
of a scheme $X$ along its closed subscheme $Z$ consists in
``subtracting from $X$ the open complement $U$ to $Z$ in $X$'',
and $U$ is a quasi-compact scheme whenever, say, $X$ is affine and
$Z$ is defined by a finitely generated ideal.
 So the Matlis--Greenlees--May duality is, in fact, an \emph{equivalence
between the conventional derived categories of torsion modules and
contramodules} over certain formal schemes~\cite{Pmgm}.

\Subsection{0.6}
 A common feature of all or almost all kinds of contramodules is
that they form abelian categories with enough projective objects.
 One can define ``a contramodule category'' in the most general
sense of the word as a \emph{locally presentable abelian category with
enough projective objects}, or equivalently, a locally presentable
abelian category with a projective generator~\cite{Pper}.
 Abelian categories with a fixed projective generator are described
by additive monads on the category of sets; and among such categories,
the locally presentable ones correspond to accessible
monads~\cite{Vit,PR,Pper}.
 Another name for locally presentable abelian categories with
a projective generator is \emph{the categories of models of additive
$\kappa$\+ary algebraic theories}, where $\kappa$~stands for
some regular cardinal (depending on the category) \cite{Wr,PR}.

 The categories of comodule-like structures, on the other hand, tend
to be (at least) Grothendieck abelian categories.
 So one can say, very roughly, that both the comodule and
the contramodule categories are locally presentable abelian categories;
but the difference is that the comodule categories have enough
injectives, while the contramodule categories have enough projectives.

 Abelian categories with a projective generator are known to arise as
the hearts of tilting (and even silting) t\+structures associated with
``big'' (infinitely generated) tilting or silting objects
\cite[Proposition~4.3]{PV}, \cite[Proposition~4.9]{Ang}.
 Hence the connection between contramodules and infinitely generated
tilting/silting theory.
 Dually, the hearts of the cotilting (and cosilting) t\+structures
are abelian categories with an injective cogenerator.
 The \emph{$n$\+tilting-cotilting correspondence}, as developed
in~\cite{PS1}, and even more so the \emph{$\infty$\+tilting-cotilting
correspondence} of the paper~\cite{PS2}, are a generalization and
an abstractly-categorical interpretation of the co-contra
correspondence phenomenon.

\Subsection{0.7}
 Before we finish this introduction, let as say a few words about
\emph{applications} of contramodules.
 There are different kinds of applications.
 As we mentioned above, in the work of the present author contramodules
were first used in order to \emph{formulate} a certain theory, namely,
the semi-infinite homological algebra of associative algebraic
structures~\cite{Plet,Psemi}.
 Countramodules also found their place in the formulation of
the derived nonhomogeneous Koszul duality~\cite{Pkoszul,Prel}.
 Moreover, even the classical topic of MGM duality is best formulated
using contramodules~\cite{Pmgm,Pdc}.

 Applications in which contramodules are used in order to \emph{prove}
theorems (in whose formulations they are \emph{not} mentioned)
are a different matter.
 In the work of the present author, such applications started to appear
relatively recently.
 Mostly, these are applications to commutative
algebra~\cite{BP1,PSl1,PSl2}, among which the most important one, in
our view, is the proof of the \emph{Very Flat Conjecture}
(which was formulated in the February~2014 version of the long
preprint~\cite{Pcosh} and proved in the August~2017
preprint~\cite{PSl1}).
 There is also an application to noncommutative ring
theory~\cite{Pcoun} and an application to direct limits of classes of
modules over noncommutative rings~\cite{PT}.

\Subsection{0.8}
 We refrain from elaborating any further upon the various derived
categories and the derived co-contra correspondence in this paper,
restricting ourselves mostly to the short discussion above
in this introduction.
 Indeed, it appears that the coderived and contraderived categories
have attracted already some attention in the recent years, and
a number of people have mastered the beginnings of the related
techniques in one form or another.
 Besides, there is the presentation~\cite{Psli} discussing
the philosophy of the derived co-contra correspondence.

 Instead, we concentrate on the even more basic, and at the same time
perhaps presently more counterintuitive, concepts of the abelian
categories of contramodules.
  In addition, the selection of the more advanced material for inclusion
into this paper is oriented towards representation theory (rather
than commutative algebra or algebraic geometry).
 Moreover, the numerous categories of contramodules that are
\emph{defined} as full subcategories in module categories (typically,
the right Ext-perpendicular subcategories to some modules)
\cite{Pmgm,Pcta,PMat,Pper,BP2,Pdc} are left almost entirely outside of
the scope of this overview.
 This excludes the major applications to commutative
algebra~\cite{PSl1,PSl2}, which are therefore only briefly mentioned
above in this introduction.
 A detailed treatment of this material will be presented elsewhere.

 The simplest examples of the categories of contramodules over
coalgebras over fields, the $l$\+adic integers, the Virasoro algebra,
and locally compact totally disconnected topological groups are
discussed in Section~\ref{first-examples}.
 The key definitions of the categories of contramodules over
topological rings, topological associative and Lie algebras,
corings and semialgebras, and the category $\sO^\ctr$ are
introduced in Section~\ref{co-and-contra-categories}.
 Tensor and Hom-like operations on the categories of contramodules
and comodules and relations between various classes of objects
adjusted to these operations (analogues and dual versions of
the classes of flat, projective, and injective modules) are briefly
considered in the first three subsections of
Section~\ref{tensor-and-adjusted}.
 Several \emph{underived} co-contra correspondence constructions are
discussed in the middle part of Section~\ref{tensor-and-adjusted}.
 Some additional topics, most notable of them concerning
full-and-faithfulness of contramodule forgetful functors,
occupy the final Subsections~\ref{addm}\+-\ref{fully-faithful}.

\medskip\noindent
\textbf{Acknowledgement.}
 I~learned the definition of a contramodule from a hard copy of
the Eilenberg--Moore 1965 AMS Memoir stored in the library
of the Institute for Advanced Study, where I~went in the Spring
of 1999 to look for relevant literature after the discovery of
(what are now known as) the coderived and contraderived categories
shortly before.
 I~was supported by an NSF grant at the time.
 Subsequently, I~studied contramodules for many years in Moscow,
starting from the Summers of 2000 and 2002 and then in 2006--2014.
 I~was partially supported by grants from EPDI, CRDF, INTAS,
Pierre Deligne's Balzan prize, Simons Foundation, and several RFBR
grants over the years.
 After Spring 2014, my research on contramodules continued in
Be'er Sheva, Haifa, Brno, Prague, and Padova, where I~was supported
by ISF and GA\v CR grants.
 This paper was largely written when I~was visiting the Technion in
Haifa in October 2014--March 2015, where I~was supported by
a fellowship from the Lady Davis Foundation, and then updated in
2019 and 2021 when, as a researcher at the Institute of Mathematics
of the Czech Academy of Sciences in Prague, I~was supported by research
plan RVO:~67985840.
 In 2021, I~was also supported by the GA\v CR project 20-13778S.
 I~am grateful to Dmitry Kaledin who suggested the idea of
writing an overview on contramodules to me some years ago.
 I~would like to thank Joseph Bernstein and Amnon Yekutieli
for helpful discussions.

\Section{First Examples} \label{first-examples}

\subsection{Contramodules over coalgebras over fields}
\label{over-coalgebras}
 We start with recalling the largely familiar definitions.
 A coassociative \emph{coalgebra} $\C$ with counit over a field~$k$
is a $k$\+vector space endowed with a \emph{comultiplication} map
$\mu_\C\:C\rarrow\C\ot_k\C$ and a \emph{counit} map
$\eps_\C\:\C\rarrow k$ satisfying the equations dual to the equations
on the multiplication and unit maps of an associative algebra with unit.
 Explicitly, the two compositions of the comultiplication map~$\mu$
with the two maps $\mu\ot\id_\C$ and $\id_\C\ot\mu\:\C\ot_k\C\birarrow
\C\ot_k\C\ot_k\C$ induced by the comultiplication map
$$
 \C\rarrow\C\ot_k\C\birarrow\C\ot_k\C\ot_k\C
$$
should be equal to each other, $(\mu\ot\id_\C)\circ\mu=
(\id_\C\ot\mu)\circ\mu$, and both the compositions of
the comultiplication map with the two maps $\eps\ot\id_\C$ and
$\id_\C\ot\eps\:\C\ot_k\C\birarrow\C$ induced by the counit map~$\eps$
$$
 \C\rarrow\C\ot_k\C\birarrow\C
$$
should be equal to the identity map,
$(\eps\ot\id_\C)\circ\mu=\id_\C=(\id_\C\ot\eps)\circ\mu$.

 A \emph{left comodule} $\M$ over a coalgebra $\C$ is a $k$\+vector
space endowed with a \emph{left coaction} map
$\nu_\M\:\M\rarrow\C\ot_k\M$ satisfying the coassociativity and
counitality equations.
 Explicitly, the two compositions of the coaction map~$\nu$ with
the two maps $\mu\ot\id_\M$ and $\id_\C\ot\nu\:\C\ot_k\M\birarrow
\C\ot_k\C\ot_k\M$ induced by the comultiplication and coaction maps
$$
 \M\rarrow\C\ot_k\M\birarrow\C\ot_k\C\ot_k\M
$$
should be equal to each other, $(\mu\ot\id_\M)\circ\nu=
(\id_\C\ot\nu)\circ\nu$, and the composition of the coaction map
with the map $\eps\ot\id_\M\:\C\ot_k\M\rarrow\M$ induced by
the counit map~$\eps_\C$
$$
 \M\rarrow\C\ot_k\M\rarrow\M
$$
should be equal to the identity map, $(\eps\ot\id_\M)\circ\nu=\id_\M$.
 A \emph{right comodule} $\N$ over $\C$ is a $k$\+vector space
endowed with a right coaction map $\nu=\nu_\N\:\N\rarrow\N\ot_k\C$
satisfying the similar equations,
$(\nu\ot\id_\C)\circ\nu=(\id_\N\ot\mu)\circ\nu$
$$
 \N\rarrow\N\ot_k\C\birarrow\N\ot_k\C\ot_k\C,
$$
and $(\id_\N\ot\eps)\circ\nu=\id_\N$
$$
 \N\rarrow\N\ot_k\C\rarrow\N.
$$

 In order to arrive to the definition of a contramodule over $\C$,
one only has to rewrite the most familiar definition of a module
over an associative algebra in a slightly different form before
quite formally dualizing it.
 Given an associative algebra $A$ over~$k$ with the multiplication
map $m\:A\ot_kA\rarrow A$ and the unit map $e\:k\rarrow A$, one
would usually define a left $A$\+module $M$ as a $k$\+vector space
endowed with a left action map $n\:A\ot_kM\rarrow M$ satisfying
the associativity and unitality equations $n\circ(m\ot\id_M)=
n\circ(\id_A\ot n)$
$$
 A\ot_kA\ot_kM\birarrow A\ot_kM\rarrow M
$$
and $n\circ(e\ot_k\id_M)=\id_M$
$$
 M\rarrow A\ot_kM\rarrow M.
$$
 However, having a map~$n$ is the same thing as having a map
$$
 p\:M\rarrow\Hom_k(A,M),
$$
which then has to satisfy the associativity and unitality equations
written in the form $\Hom(m,\id_M)\circ p=\Hom(\id_A,p)\circ p$
$$
 M\rarrow\Hom_k(A,M)\birarrow\Hom_k(A\ot_kA\;M)\simeq
 \Hom_k(A,\Hom_k(A,M))
$$
and $\Hom(e,\id_M)\circ p = \id_M$
$$
 M\rarrow\Hom_k(A,M)\rarrow M.
$$

 In this approach, the difference between the left and right modules
lies in the way one identifies the Hom from the tensor product
$\Hom_k(A\ot_kA\;M)$ with the double Hom space $\Hom_k(A,\Hom_k(A,M))$:
presuming the identification
\begin{equation} \label{left-tensor-identification}
 \Hom_k(U\ot_kV\;W)\simeq\Hom_k(V,\Hom_k(U,W))
\end{equation}
leads to the definition of a \emph{left} $A$\+module, while
identifying $\Hom_k(A\ot_kA\;N)$ with $\Hom_k(A,\Hom_k(A,N))$
by the rule
\begin{equation} \label{right-tensor-identification}
 \Hom_k(U\ot_kV\;W)\simeq\Hom_k(U,\Hom_k(V,W))
\end{equation}
and writing the same equations produces the definition of
a \emph{right} $A$\+module~$N$.

 Now we can formulate our main definition.
 A \emph{left contramodule} $\P$ over a coalgebra $\C$ is
a $k$\+vector space endowed with a \emph{left contraaction} map
$$
 \pi_\P\:\Hom_k(\C,\P)\rarrow\P
$$
satisfying the following \emph{contraassociativity} and
\emph{contraunitality} equations.
 Firstly, the two compositions of the two maps
$\Hom(\mu,\P)\:\Hom_k(\C\ot_k\C\;\P)\rarrow\Hom_k(\C,\P)$
and $\Hom(\C,\pi)\:\Hom_k(\C,\Hom_k(\C,\P))\rarrow\Hom_k(\C,\P)$
induced by the comultiplication map $\mu=\mu_\C$ and
the contraaction map $\pi=\pi_\P$ with the contraaction map~$\pi$
$$
 \Hom_k(\C,\Hom_k(\C,\P))\simeq\Hom_k(\C\ot_k\C\;\P)
 \birarrow\Hom_k(\C,\P)\rarrow\P
$$
should be equal to each other, $\pi\circ\Hom(\mu,\P)=
\pi\circ\Hom(\C,\pi)$, presuming the identification of
$\Hom_k(\C\ot_k\C\;\P)\simeq\Hom_k(\C,\Hom_k(\C,\P))$ by
the left rule~\eqref{left-tensor-identification}.
 Secondly, the composition of the map $\Hom(\eps,\P)\:
\P\rarrow\Hom_k(\C,\P)$ induced by the counit map $\eps=\eps_\C$
with the contraaction map
$$
 \P\rarrow\Hom_k(\C,\P)\rarrow\P
$$
should be equal to the identity map, $\pi\circ\Hom(\eps,\P)=\id_\P$.

 This definition can be found in~\cite[Section~0.2.4]{Psemi}; see
also~\cite[Section~2.2]{Pkoszul} (the classical source
is~\cite[Section~III.5]{EM}).
 Using the identification by the right
rule~\eqref{right-tensor-identification} instead
of~\eqref{left-tensor-identification} produces the definition
of a \emph{right contramodule} over~$\C$.
 The way to understand why \eqref{left-tensor-identification}~is
the ``left'' rule and \eqref{right-tensor-identification}~is
the ``right'' one lies in replacing a basic field~$k$ with
a noncommutative ring; see
Section~\ref{over-corings} below.

\subsection{Basic properties of comodules and contramodules}
\label{basic-properties}
 The simplest way to produce examples of contramodules is by
applying the Hom functor to comodules in the first argument.
 Specifically, let $\N$ be a right comodule over a coalgebra
$\C$ over~$k$ and $V$ be a $k$\+vector space.
 Then the vector space $\P=\Hom_k(\N,V)$ has a natural structure
of left contramodule over~$\C$.
 The left contraaction map~$\pi_\P$ is constructed by applying
the functor $\Hom_k({-},V)$ to the right coaction map~$\nu_\N$
$$
 \Hom_k(\C,\Hom_k(\N,V))\simeq \Hom_k(\N\ot_k\C\;V)
 \rarrow\Hom_k(\N,V).
$$

 Let us denote by $k\vect$ the category of $k$\+vector spaces,
by $\C\comodl$ the category of left $\C$\+comodules, by
$\comodr\C$ the category of right $\C$\+comodules, and by
$\C\contra$ the category of left $\C$\+contramodules.
 The $k$\+vector space of morphisms between left $\C$\+comodules
$\L$ and $\M$ will be denoted by $\Hom_\C(\L,\M)$, and 
the vector space of morphisms between left $\C$\+contramodules
$\P$ and $\Q$ by $\Hom^\C(\P,\Q)$.

 The category $\C\comodl$ is abelian and the forgetful functor
$\C\comodl\rarrow k\vect$ is exact.
 To prove as much, one has to use the observation that the tensor
product functor $\C\ot_k{-}$ is exact, or more specifically,
\emph{left exact}.
 The forgetful functor also preserves inductive limits, so filtered
inductive limits are exact functors in $\C\comodl$.
 The infinite products in $\C\comodl$ are not preserved by
the forgetful functor (unless $\C$ is finite-dimensional) and are
\emph{not} exact in $\C\comodl$ in general.

 In other words, the abelian category of $\C$\+comodules satisfies
Grothendieck's axioms Ab5 and Ab3*, but not in general Ab4*
\cite[N$^{\mathrm o}$~1.5]{GrToh}.
 It also admits a set of generators (for which one can take
the finite-dimensional comodules), so it has enough injective
objects~\cite[N$^{\mathrm o}$~1.10]{GrToh}.
 These can be explicitly described as follows.

 A \emph{cofree} left $\C$\+comodule is a $\C$\+comodule of the form
$\C\ot_k V$, where $V$ is a $k$\+vector space, with the left
$\C$\+coaction induced by the comultiplication in~$\C$.
 For any left $\C$\+comodule $\L$, there is a natural isomorphism
of $k$\+vector spaces
$$
 \Hom_\C(\L\;\C\ot_kV)\simeq\Hom_k(\L,V),
$$
so cofree $\C$\+comodules are injective.
 The coaction map $\nu\:\M\rarrow\C\ot_k\M$ embeds any left
$\C$\+comodule into a cofree one, so there are enough cofree
$\C$\+comodules.
 It follows that a $\C$\+comodule is injective if and only if it is
a direct summand of a cofree one~\cite[Sections~0.2.1, 1.1.2,
and~5.1.5]{Psemi}.

 The category $\C\contra$ is abelian and the forgetful functor
$\C\contra\rarrow k\vect$ is exact (here one has to observe that
the functor $\Hom_k(\C,{-})$ is exact, or more specifically,
\emph{right exact}).
 The forgetful functor also preserves infinite products, so
infinite products are exact functors in $\C\contra$.
 The infinite direct sums are \emph{not} preserved by
the forgetful functor (unless $\C$ is finite-dimensional) and
are not exact in $\C\contra$ in general.
 (However, \emph{un}like the infinite products of $\C$\+comodules,
the infinite direct sums of $\C$\+contramodules remain exact
when the homological dimension of the category $\C\contra$
does not exceed~$1$ \cite[Remark~1.2.1]{Pweak}.)

 In other words, the abelian category of $\C$\+contramodules satisfies
Grothendieck's axioms Ab3 and Ab4*, but not in general Ab4 or Ab5*.
 It also has enough projective objects, which can be explicitly
described as follows.

 A \emph{free} left $\C$\+contramodule is a $\C$\+contramodule of
the form $\Hom_k(\C,V)$, where $V$ is a $k$\+vector space, with
the left $\C$\+contraaction constructed as explained in
the beginning of this section.
 For any left $\C$\+contramodule $\Q$, there is a natural isomorphism
of $k$\+vector spaces
$$
 \Hom^\C(\Hom_k(\C,V),\Q)\simeq\Hom_k(V,\Q),
$$
so free $\C$\+contramodules are projective.
 The contraaction map $\pi\:\Hom_k(\C,\P)\rarrow\P$ presents any
$\C$\+contramodule as the quotient contramodule of a free one, so
there are enough free contramodules.
 It follows that a $\C$\+contramodule is projective if and only if
it is a direct summand of a free one~\cite[Sections~0.2.4,
3.1.2, and~5.1.5]{Psemi}.

 Notice that the class of injective $\C$\+comodules is not only
closed under infinite products in $\C\comodl$ (which holds in any
abelian category), but also under infinite direct sums.
 Similarly, the class of projective $\C$\+contramodules is not
only closed under infinite direct sums in $\C\contra$ (as in any
abelian category), but also under infinite products.
 These observations are important for the theory of coderived and
contraderived categories~\cite[Section~4.4, cf.\
Sections~3.7--3.8]{Pkoszul}.

 The correspondence assigning the free $\C$\+contramodule
$\Hom_k(\C,V)$ to the cofree $\C$\+comodule $\C\ot_k V$ is 
an equivalence between the additive categories of cofree
left $\C$\+comodules and free left $\C$\+contramodules.
 Hence the additive categories of injective left $\C$\+comodules
and projective left $\C$\+contramodules are equivalent,
too~\cite[Sections~0.2.6 and~5.1.3]{Psemi} (see
also~\cite{BBW} and~\cite[Sections~5.1--5.2]{Pkoszul}).

\subsection{Contramodules over the formal power series}
\label{over-power-series}
 The linear duality functor identifies the category opposite to
the category of conventional infinite-dimensional (otherwise known
as discrete, or ind-finite-dimensional) vector spaces with the category
of \emph{linearly compact}, or pro-finite-dimensional, vector spaces.
 In particular, a coassociative coalgebra with counit is the same thing
(up to inverting the arrows) as a linearly compact or
pro-finite-dimensional topological associative algebra with unit.
 Notice that any coassociative coalgebra is the union of its
finite-dimensional subcoalgebras~\cite[Section~2.2]{Swe}, so any 
topological associative algebra with a pro-finite-dimensional
underlying topological vector space is a projective limit
of finite-dimensional associative algebras.

 In particular, one can identify coalgebras by the names of their
dual linearly compact topological algebras.
 In this section we consider the simplest example of
an infinite-dimensional coassociative coalgebra---the coalgebra $\C$
for which the dual topological algebra $\C^*$ is isomorphic to
the algebra $k[[z]]$ of formal Taylor power series in one variable
over a field~$k$.
 Explicitly, $\C$ is the $k$\+vector space with a countable basis
consisting of the formal symbols $1^*$, $z^*$, $z^2{}^*$, \dots,
$z^n{}^*$, \dots, $n\in\Z_{\ge0}$, with the comultiplication map
given by the rule
$$
 \mu(z^n{}^*)=\sum_{i+j=n}z^i{}^*\ot z^j{}^*
$$
and the counit map $\eps(1^*)=1$, \ $\eps(z^n{}^*)=0$ for $n>0$.

 Then a (left or right) $\C$\+comodule $\M$ is the same thing
as a $k$\+vector space endowed with a \emph{locally nilpotent}
linear operator $z\:\M\rarrow\M$.
 In other words, for any vector $m\in\M$ there must exist
an integer $n\ge1$ such that $z^n(m)=0$ in~$\M$.
 Indeed, given a linear operator~$z$ on $\M$ one would define
the coaction map $\nu\:\M\rarrow\C\ot_k\M$ by the formula
$$
 \nu(m)=\sum_{n=0}^\infty z^n{}^*\ot z^n(m),
$$
and the local nilpotence condition is needed for the sum to be
well-defined (i.~e., finite) for every vector $m\in\M$.

 A $\C$\+contramodule structure on a $k$\+vector space $\P$ is,
by the definition, the datum of a $k$\+linear map
$\pi\:\Hom_k(\C,\P)\rarrow\P$ satisfying the contraassociativity
and contraunitality axioms.
 Having such a map is the same thing as the following \emph{infinite
summation operation} being defined in~$\P$.
 For every sequence of vectors $p_0$, $p_1$, $p_2$~\dots~$\in\P$
there should be given a vector denoted figuratively by
$$
 \sum_{n=0}^\infty z^np_n\in\P.
$$
 This infinitary operation in $\P$ should satisfy the equations of
linearity
$$
 \sum_{n=0}^\infty z^n(ap_n+bq_n) =
a\sum_{n=0}^\infty z^np_n + b\sum_{n=0}^\infty z^nq_n,
$$
contraassociativity
$$
 \sum_{i=0}^\infty z^i\left(\sum_{j=0}^\infty z^jp_{ij}\right)=
 \sum_{n=0}^\infty z^n\left(\sum_{i+j=n} p_{ij}\right),
$$
and unitality
$$
 \sum_{n=0}^\infty z^np_n=p_0
 \quad \text{when $p_1=p_2=p_3=\dotsb=0$ in $\P$}
$$
for any $p_n$, $q_n$, $p_{ij}\in\P$ and $a$, $b\in k$.
 Notice that in the main (middle) equation the first three summation
signs denote the contramodule infinite summation operation, while
the fourth one is the conventional finite sum of elements of
a vector space~\cite[Section~A.1.1]{Psemi}.

 As we will see below in Section~\ref{recovering},
the $\C$\+contramodule structure on a vector space $\P$ is in fact
determined by a single linear operator $z\:\P\rarrow\P$,
$$
 z(p)=1\cdot 0 + z\cdot p + z^2\cdot 0 + z^3\cdot 0 +\dotsb
$$
 However, \emph{un}like for the comodules, for contramodules over
more complicated coalgebras the similar statement is, of course,
no longer true.

\subsection{Contramodules over the $l$-adic integers}
\label{over-l-adics}
 A left module $\M$ over a topological ring $\R$ is called
\emph{discrete} if the action map $\R\times\M\rarrow\M$ is continuous
in the discrete topology of $\M$ and the given topology of~$\R$.
 In other words, this means that the annihilator of every element
of $\M$ must be an open left ideal in~$\R$.
 Discrete left $\R$\+modules form an abelian category, which we denote
by $\R\discr$.

 The discussion of topological algebras dual to coalgebras in
the previous section ignored one point which we now have to clarify.
 Given a coassociative coalgebra $\C$ over~$k$, one can define
the multiplication on the dual vector space $\C^*$ in two
approximately equally natural ways which differ by the passage to
the opposite algebra, i.~e., switching the left and right arguments
of the product map.
 Let us make the choice of defining the multiplication on $\C^*$
in such a way that left $\C$\+comodules acquire natural structures
of left $\C^*$\+modules and right $\C$\+comodules become right
$\C^*$\+modules.
 Explicitly, this means applying the formula
$$
 \langle fg,c\rangle = \langle f,c_{(2)}\rangle\langle g,c_{(1)}\rangle
$$
where $\langle\,,\,\rangle$ denotes the natural pairing $\C^*\times
\C\rarrow k$ and $c\longmapsto c_{(1)}\ot c_{(2)}$ is Sweedler's
symbolic notation for the comultiplication map~$\mu$
\cite[Section~1.2]{Swe}.

 Then the category of left $\C$\+comodules can be described as
the full subcategory in the category of left $\C^*$\+modules
$\C^*\modl$ consisting precisely of those $\C^*$\+modules that are
discrete with respect to the topology of~$\C^*$.
 Similarly, a right $\C$\+comodule is the same thing as a discrete
right $\C^*$\+module~\cite[Section~2.1]{Swe}.

 Now the explicit description of contramodules over the coalgebra $\C$
with $\C^*=k[[t]]$ given in the previous section raises the question
about defining contramodules over topological rings other than
pro-finite-dimensional algebras over fields.
 The most close analogues of the rings $k[[t]]$ being the rings
of $l$\+adic integers $\Z_l$, they are the natural starting point of
the desired generalization (whose full development we postpone until
Sections~\ref{over-topol-rings}\+-\ref{over-top-algebras}).

 So let $l$~be a prime number.
 Let us start with mentioning that a discrete module over
the topological ring of $l$\+adic integers $\Z_l$ is the same thing
as an $l$\+primary abelian group, i.~e., an abelian group where
the order of every element is a power of~$l$.
 The category $\Z_l\discr$ is abelian with exact functors of filtered
inductive limits, which are also preserved by the embedding functor
$\Z_l\discr\rarrow\Ab$ into the category of abelian groups.
 The infinite products in $\Z_l\discr$ are not preserved by
the forgetful functor and \emph{not} exact.
 In other words, the category $\Z_l\discr$ satisfies Ab5 and Ab3*,
but not Ab4*.
 It has enough injective objects, but no nonzero projectives.
 The injective discrete $\Z_l$\+modules are precisely the direct
sums of copies of the group $\mathbb Q_l/\Z_l$.

 A \emph{$\Z_l$\+contramodule} $\P$ is an abelian group endowed with
the following infinite summation operation.
 For any sequence of elements $p_0$, $p_1$, $p_2$~\dots~$\in\P$
an element denoted symbolically by
$$
 \sum_{n=0}^\infty l^n p_n\in\P
$$
should be defined.
 This infinitary operation should satisfy the equations of additivity
$$
 \sum_{n=0}^\infty l^n(p_n+q_n) =
 \sum_{n=0}^\infty l^np_n + \sum_{n=0}^\infty l^nq_n,
$$
contraassociativity
$$
 \sum_{i=0}^\infty l^i\left(\sum_{j=0}^\infty l^jp_{ij}\right)=
 \sum_{n=0}^\infty l^n\left(\sum_{i+j=n} p_{ij}\right),
$$
and unitality + compatibility with the abelian group structure
$$
 \sum_{n=0}^\infty l^np_n=p_0+p_1+\dotsb+p_1
 \text{ ($l$ summands $p_1$) \ when $p_2=p_2=p_3=\dotsb=0$}
$$
for any elements $p_n$, $q_n$, and $p_{ij}\in\P$.

 For any $l$\+primary abelian group $\M$ and abelian group $V$,
the abelian group $\Hom_\Z(\M,V)$ has a natural $\Z_l$\+contramodule
structure provided by the rule
$$
 \left(\sum_{n=0}^\infty l^np_n\right)(m)=\sum_{n=0}^\infty p_n(l^nm)
$$
for any $p_n\in\Hom_\Z(\M,V)$ and $m\in\M$.
 The category $\Z_l\contra$ of $\Z_l$\+contramodules is abelian and
the forgetful functor $\Z_l\contra\rarrow\Ab$ is exact.
 As we will see in Section~\ref{recovering}, the forgetful functor
is fully faithful.
 It preserves infinite products, but \emph{not} infinite direct sums.
 Both the infinite direct sums and infinite products are exact
functors in $\Z_l\contra$.
 In other words, the category $\Z_l\contra$ satisfies Ab4 and Ab4*
(but \emph{not} Ab5 or Ab5*).
 It has enough projective objects, but no injectives.

 The \emph{free\/ $\Z_l$\+contramodule} generated by a set $X$ is
the set $\Z_l[[X]]$ of all infinite formal linear combinations
$\sum_{x\in X} a_xx$ of elements of $X$ with the coefficients
$a_x\in\Z_l$ such that for every $n\ge1$ all but a finite number
of $a_x$ are divisible by $l^n$ in~$\Z_l$.
 Notice that any formal linear combination satisfying this condition
is, in fact, supported in an at most countable subset in~$X$.
 As we will see below in
Sections~\ref{over-topol-rings}\+-\ref{over-adic-completions},
for any $\Z_l$\+contramodule $\P$ the group of all
$\Z_l$\+contramodule morphisms $\Z_l[[X]]\rarrow\P$ is isomorphic
to the group $\P^X$ of arbitrary maps of sets $X\rarrow\P$.
 The classes of free and projective $\Z_l$\+contramodules coincide.

 The additive categories of injective discrete $\Z_l$\+modules and
projective $\Z_l$\+con\-tramodules are equivalent; the equivalence is
provided by the functors $\M\longmapsto\Hom_\Z(\mathbb Q_l/\Z_l,\M)$
and $\P\longmapsto\mathbb Q_l/\Z_l\ot_\Z\P$
\cite[Proposition~2.1]{Har}.
 In particular, one has
$$\textstyle
 \Hom_\Z(\mathbb Q_l/\Z_l\;\bigoplus_X \mathbb Q_l/\Z_l)\simeq\Z_l[[X]]
 \quad\text{and}\quad
 \mathbb Q_l/\Z_l\ot_Z\Z_l[[X]]\simeq\bigoplus_X\mathbb Q_l/\Z_l.
$$

\subsection{Counterexamples}  \label{counterexamples}
 For any topological ring~$\R$, one can compute infinite products in
the abelian category $\R\discr$ in the following way.
 Let $\M_\alpha$ be a family of discrete left $\R$\+modules;
denote by $M$ their product in the abelian category of
arbitrary $\R$\+modules.
 Then the product of the family of objects $\M_\alpha$ in the category
$\R\discr$ can be obtained as an $\R$\+submodule $\M\subset M$
consisting precisely of all the elements $m\in M$ whose annihilators
in $\R$ are open left ideals.

 In particular, this provides a rule for computing infinite products
in the abelian categories $\C\comodl$ of comodules over coalgebras
over fields.
 Another way to formulate such a rule is as follows.
 In any abelian category, infinite products are left exact functors;
in other words, they preserve kernels of morphisms.
 Since any $\C$\+comodule can be presented as the kernel of a morphism
of cofree $\C$\+comodules, it suffices to know what the products of
families of cofree $\C$\+comodules are.
 The latter are easily seen to be given by the formula
$\prod_\alpha \C\ot_k V_\alpha = \C\ot_k\prod_\alpha V_\alpha$.

 Similarly, in order to compute the infinite direct sum of a family
of objects in $\C\contra$, one can present these as the cokernels of
morphisms of free $\C$\+contramod\-ules.
 Since any $\C$\+contramodule can be obtained as such a cokernel
and the infinite direct sums preserve cokernels, it remains to use
the formula $\bigoplus_\alpha \Hom_k(\C,V_\alpha) = \Hom_k(\C\;
\bigoplus_\alpha V_\alpha)$ for the direct sum of a family of
free $\C$\+contramodules.

 Let us return to the example of the coalgebra $\C$ dual to
the topological algebra of formal power series $k[[z]]$
considered in Section~\ref{over-power-series}.
 Viewed as a discrete module over the algebra $\C^*=k[[z]]$,
the coalgebra $\C$ can be identified with the quotient module
$k((z))/k[[z]]$ of the $k[[z]]$\+module of Laurent series
$k((z))$ by its submodule~$k[[z]]$.
 Consider the family of discrete $k[[z]]$\+modules
$z^{-n}k[[z]]/k[[z]]$, \ $n=1$, $2$,~\dots{}
 They can be included into short exact sequences of discrete
$k[[z]]$\+modules
$$
 0\lrarrow z^{-n}k[[z]]/k[[z]]\lrarrow k((z))/k[[z]]\lrarrow
 k((z))/z^{-n}k[[z]]\lrarrow0.
$$
 Passing to the infinite product of these short exact sequences
in the category $\C\comodl$ over all $n\ge1$, one discovers
that the map $k((z))/k[[z]]\ot_k\prod_nk=
\prod_n k((z))/k[[z]]\rarrow\prod_n k((z))/z^{-n}k[[z]]=
k((z))/k[[z]]\ot_k\prod_nkz^{-n}$ is not surjective, as, e.~g.,
the vector $(z^{-n-1})_n\in\prod_n k((z))/z^{-n}k[[z]]$ does not
belong to its image.
 One also computes the infinite product $\prod_n z^{-n}k[[z]]/k[[z]]$
in the category $\C\comodl$ as isomorphic to the inductive limit
$\varinjlim_m \big(\prod_{n=1}^m z^{-n}k[[z]]/k[[z]]\times
\prod_{n=m+1}^\infty z^{-m}k[[z]]/k[[z]]\big)$.

 Now consider the family of $\C$\+contramodules
$k[[z]]/z^nk[[z]]$, \ $n=1$, $2$~\dots{}
 They can be viewed as parts of the short exact sequences of
$\C$\+contramodules
$$
 0\lrarrow z^nk[[z]]\lrarrow k[[z]]\lrarrow k[[z]/z^nk[[z]]
 \lrarrow0.
$$
 Passing to the infinite direct sum of these short exact
sequences in the category $\C\contra$ over all $n\ge1$, one
finds out that the map $\Hom_k(\C\;\bigoplus_n kz^n)=
\bigoplus_n z^nk[[z]]\allowbreak\rarrow\bigoplus_n k[[z]]=
\Hom_k(\C\;\bigoplus_n k)$ is injective.
 Its cokernel $\P=\bigoplus_n k[[z]]/z^nk[[z]]\in\C\contra$
is the $\C$\+contramodule that we are interested in.

 Let us start with introducing a more careful notation.
 Let $\E$ denote the free $\C$\+contramodule generated by
a $k$\+vector space $E$ with a countable basis
$e_1$, $e_2$, $e_3$,~\dots{}
 Explicitly, $\E$ is the set of all formal linear combinations
$\sum_{n=1}^\infty a_n(z)e_n$, where the sequence of formal power
series $a_n(z)\in k[[z]]$ converges to zero in the topology of~$k[[z]]$.
 The $k[[z]]$\+contramodule infinite summation operations on $\E$
are defined in the obvious way.
 Let $\F$ be the similar free $\C$\+contramodule generated by
a $k$\+vector space $F$ with a basis $f_1$, $f_2$, $f_3$,~\dots{}
 Define a morphism of $\C$\+contramodules $\F\rarrow\E$ by
the rule $\sum_{n=1}^\infty b_n(z)f_n\longmapsto
\sum_{n=1}^\infty z^nb_n(z)e_n$.
 Clearly, this morphism is injective; denote its cokernel by
$\P=\E/\F$.

 Set $p_n\in\P$ to be the images of the elements $e_n\in\E$ under
the surjective morphism $\E\rarrow\P$.
 Then the infinite sum $p=\sum_{n=1}^\infty z^np_n$ is a nonzero
vector in $\P$, since the element $\sum_{n=1}^\infty z^ne_n$ does
not belong to $\F\subset\E$ (there being no element
$\sum_{n=1}^\infty f_n$ in~$\F$).
 On the other hand, every finite partial sum $zp_1+z^2p_2+\dotsb+
z^np_n=zp_1+\dotsb+z^np_n+z^{n+1}\cdot0+\dotsb$ vanishes in $\P$,
the finite sum $ze_1+\dotsb+z^ne_n$ being the image of
the vector $f_1+\dotsb+f_n\in\F$ in~$\E$.
 It follows that our vector $p=z^n(p_n+zp_{n+1}+\dotsb)$ belongs to
$z^n\P$ for every $n\ge1$, so the $z$\+adic topology on $\P$ is
not separated.
 This counterexample can be found in~\cite[Section~A.1.1]{Psemi};
it also occured, under slightly different guises,
in~\cite[Example~2.5]{Sim} and~\cite[Example~3.20]{Yek1}.

 Among other things, $\P$~provides an example of a $\C$\+contramodule
that does not have the form $\Hom_k(\N,V)$ for any $\C$\+comodule~$\N$.
 An example of a finite-dimensional contramodule not of this form
(over a more complicated coalgebra~$\C$) can be found
in~\cite[Section~A.1.2]{Psemi}.
 Concerning the above coalgebra $\C$ with $\C^*=k[[z]]$, let us
point out that the natural map $\Q\rarrow\varprojlim_n\Q/z^n\Q$,
though not necessarily injective, is always \emph{surjective} for
a $\C$\+contramodule~$\Q$ \cite[Lemma~A.2.3]{Psemi}.
 Indeed, let $q_n\in\Q$ be a sequence of vectors such that
$q_{n+1}-q_n\in z^n\Q$ for every $n=1$, $2$,~\dots{}
 Suppose $q_{n+1}-q_n=z^np_n$; then the infinite sum
$q=q_1+\sum_{n=1}^\infty z^np_n$ provides an element $q\in\Q$ for which
$q-q_n\in z^n\Q$ for every $n\ge1$.

 Similarly, consider the family of $l$\+primary abelian groups
$l^{-n}\Z/\Z$, \ $n=1$, $2$,~\dots{}
 They can be included into short exact sequences of $l$\+primary
abelian groups
$$
 0\lrarrow l^{-n}\Z/\Z\lrarrow\mathbb Q_l/\Z_l\lrarrow
 \mathbb Q_l/l^{-n}\Z_l\lrarrow0.
$$
 Passing to the infinite product of these short exact sequences
in the category $\Z_l\discr$ over all~$n\ge1$, one discovers that
the map $\prod_n\mathbb Q_l/\Z_l\rarrow\prod_n\mathbb Q_l/
l^{-n}\Z_l$ is not surjective, as, e.~g., the element
$(l^{-n-1})_n\in\prod_n\mathbb Q_l/l^{-n}\Z_l$ does not belong to
its image.
 One also computes the infinite product $\prod_n l^{-n}\Z/\Z$ in
the category $\Z_l\discr$ as isomorphic to the inductive limit
$\varinjlim_m\big(\prod_{n=1}^m l^{-n}\Z/\Z\times
\prod_{n=m+1}^\infty l^{-m}\Z/\Z\big)$.

 Consider the family of $\Z_l$\+contramodules $\Z/l^n\Z$, \ 
$n=1$, $2$,~\dots{}
 They can be viewed as parts of the short exact sequences of
$\Z_l$\+contramodules
$$
 0\lrarrow l^n\Z_l\lrarrow\Z_l\lrarrow\Z/l^n\Z\lrarrow0.
$$
Passing to the infinite direct sum of these short exact sequences
in the category $\Z_l\contra$ over all $n\ge1$, one finds out
that the map $\bigoplus_n l^n\Z_l\rarrow\bigoplus_n\Z_l$ is
injective.
 Its cokernel $\P=\bigoplus_n\Z/l^n\Z\in\C\contra$ can be described
as follows.

 Let $\E$ denote the free $\Z_l$\+contramodule generated by
a sequence of symbols $e_1$, $e_2$, $e_3$,~\dots{}
 Explicitly, $\E$ is the set of all formal linear combinations
$\sum_{n=1}^\infty a_ne_n$, where the sequence of $l$\+adic integers
$a_n\in\Z_l$ converges to zero in the topology of~$\Z_l$.
 Let $\F$ be the similar $\Z_l$\+contramodule generated by
a sequence of symbols $f_1$, $f_2$, $f_3$,~\dots{}
 Define a morphism of $\C$\+contramodules $\F\rarrow\E$ by the rule
$\sum_{n=1}^\infty b_nf_n\longmapsto\sum_{n=1}^\infty l^nb_ne_n$.
 Clearly, this morphism is injective; its cokernel $\E/\F$ is our
$\Z_l$\+contramodule~$\P$.

 Set $p_n=e_n\bmod\F\in\P$.
 Then the infinite sum $p=\sum_{n=1}^\infty l^np_n$ is a nonzero
element in $\P$, since the element $\sum_{n=1}^\infty l^ne_n$
does not belong to $\F\subset\E$.
 On the other hand, every summand $l^np_n$ vanishes in $\P$,
the element $l^ne_n$ being the image of the element $f_n\in\F$
in~$\E$.
 It follows that the element~$p$ belongs to $l^n\P$ for every $n\ge1$,
so the $l$\+adic topology on $\P$ is not separated.
 Notice that the natural map $\Q\rarrow\varprojlim_n \Q/l^n\Q$,
though not necessarily injective, is always surjective for
a $\Z_l$\+contramodule~$\P$ \cite[Lemma~D.1.1]{Pcosh}.
 The proof is similar to the above argument for $k[[z]]$\+contramodules.

\subsection{Recovering the contramodule structure} \label{recovering}
 We have seen in the previous section that a $k[[z]]$\+contramodule
can contain infinitely $z$\+divisible \emph{vectors}, i.~e., vectors
$p\in\P$ for which there exists a sequence of vectors $p_n\in\P$
such that $p=z^np_n$ for every $n\ge1$.
 Let us now show that \emph{no} $k[[z]]$\+contramodule can contain
infinitely $z$\+divisible \emph{$k[z]$\+submodules}.
 In other words, one can never choose the sequence of vectors
$p_n\in\P$ in a compatible way, i.~e., any sequence of
vectors $p_n\in\P$ such that $p_n=zp_{n+1}$ for all $n\ge 0$ is
the sequence of zero vectors.

 Indeed, consider the expression $q=\sum_{n=0}^\infty z^np_n\in\P$.
 By assumption, we have
$$
 \sum_{n=0}^\infty z^n p_n=\sum_{n=0}^\infty z^n\cdot zp_{n+1}
 = \sum_{n=0}^\infty z^{n+1}p_{n+1}=\sum_{n=1}^\infty z^np_n,
$$
that is $q=q-p_0$ and $p_0=0$.
 Here the last two equations conceal the use of the contraassociativity
axiom from Section~\ref{over-power-series}, which is being applied
to the double sequence of vectors $p_{ij}=p_{i+1}$ when $j=1$ and
$p_{ij}=0$ otherwise.
 The assertion we have proven is essentially a particular case of
\emph{Nakayama's lemma for contramodules}
(see Section~\ref{over-topol-rings} below).

 Now we are in the position to show that the forgetful functor
$k[[z]]\contra\rarrow k[z]\modl$ (where we denote by $k[[z]]\contra$
the category $\C\contra$ of contramodules over the coalgebra $\C$ 
with $\C^*=k[[z]]$) is \emph{fully faithful}, i.~e.,
the $\C$\+contramodule structure on a $k$\+vector space $\P$ can be
uniquely recovered from the single linear operator $z\:\P\rarrow\P$.
 Indeed, suppose that we want to ``compute'' the value of
the infinite sum $\sum_{n=0}^\infty z^np_n$ in~$\P$.
 Consider the infinite system of linear equations
\begin{equation} \label{recovering-k[[z]]-contramodule}
 q_n=p_n+zq_{n+1}, \quad\text{$n=0$, $1$, $2$,~\dots}
\end{equation}
in the indeterminates $q_n\in\P$.
 We have just shown that the related system of homogeneous linear
equations $q_n=zq_{n+1}$ has no nonzero solutions in~$\P$.
 Hence a solution of the system~\eqref{recovering-k[[z]]-contramodule}
is unique if it exists.
 Given a $k[[z]]$\+contramodule structure in $\P$, one produces
such a solution by setting
$$
 q_n=\sum_{i=0}^\infty z^ip_{n+i}.
$$
 The value of $\sum_{n=0}^\infty z^np_n$ can be recovered as
the vector $q_0\in\P$.

 We have essentially shown that a $k[z]$\+module $P$ admits
an (always unique) $k[[z]]$\+contramodule structure if and only if
the system of nonhomogeneous linear
equations~\eqref{recovering-k[[z]]-contramodule} has a unique solution
in~$q_n$ for every sequence of vectors $p_n\in P$.
 The latter condition is equivalent to the vanishing of the two Ext
spaces $\Ext^*_{k[z]}(k[z,z^{-1}]\;P)$ (see~\cite[Remark~A.1.1]{Psemi}
and~\cite[Lemmas~B.5.1 and~B.7.1]{Pweak}).

 Similarly, we have seen that a $\Z_l$\+contramodule $\P$ can contain
infinitely $l$\+divisible \emph{elements}, i.~e., there can be nonzero
elements $p\in\P$ for which there exists a sequence of
elements $p_n\in\P$ such that $p=l^np_n$ for every $n\ge1$.
 Let us show that no $\Z_l$\+contramodule can contain infinitely
$l$\+divisible \emph{subgroups}.
 In other words, one can never choose the sequence $p_n\in\P$ in
a compatible way, i.~e., any sequence of elements $p_n\in\P$ 
such that $p_n=lp_{n+1}$ for all $n\ge0$ is the zero sequence.

 Indeed, consider the expression $\sum_{n=0}^\infty l^np_n\in\P$.
 By assumption, we have
$$
 \sum_{n=0}^\infty l^np_n=\sum_{n=0}^\infty l^n\cdot lp_{n+1}=
 \sum_{n=0}^\infty l^{n+1}p_{n+1}=\sum_{n=1}^\infty l^np_n,
$$
that is $q=q-p_0$ and $p_0=0$.
 Here the first equation signifies the use of the ``compatibility
with the abelian group structure'' axiom from
Section~\ref{over-l-adics}, while the last two equations presume
an application of the contraassociativity axiom.

 Let us show that the forgetful functor $\Z_l\contra\rarrow\Ab$ is
fully faithful, i.~e., a $\Z_l$\+contramodule structure on an abelian
group $\P$ is uniquely determined by the abelian group structure.
 Suppose that we want to ``compute'' the value of the infinite
sum $\sum_{n=0}^\infty l^np_n$ in~$\P$.
 Consider the infinite system of linear equations
\begin{equation} \label{recovering-Z_l-contramodule}
 q_n=p_n+lq_{n+1}, \quad\text{$n=0$, $1$, $2$~\dots}
\end{equation}
in the indeterminates $q_n\in\P$.
 We have just shown that the related system of homogeneous linear
equations $q_n=lq_{n+1}$ has no nonzero solutions in~$\P$.
 Hence a solution of the system~\eqref{recovering-Z_l-contramodule}
is unique if it exists.
 Assuming a $\Z_l$\+contramodule structure in~$\P$, one produces
such a solution by setting
$$
 q_n=\sum_{i=0}^\infty l^ip_{n+i}.
$$
 The value of $\sum_{n=0}^\infty l^np_n$ can be recovered as
the vector $q_0\in\P$.

 We have essentially shown that an abelian group $P$ admits
an (always unique) $\Z_l$\+contramodule structure if and only if
the system of nonhomogeneous linear
equations~\eqref{recovering-Z_l-contramodule} has a unique solution
in~$q_n$ for every sequence of elements $p_n\in P$.
 The latter condition is equivalent to the vanishing of the two
Ext groups $\Ext^*_\Z(\Z[l^{-1}],P)$
(see~\cite[Remark~A.3]{Psemi} and~\cite[Theorem~B.1.1 and
Lemma~B.7.1]{Pweak}; cf.~\cite[Sections~VI.3\+-4]{BK}
and~\cite[Definition~4.6 and Remark~4.7]{Jan}).

\subsection{Contramodules over the Virasoro algebra}
\label{over-virasoro}
 The definition of contramodules over the formal power series
algebra in terms of infinite summation operations, as stated
in Section~\ref{over-power-series}, opens the door to
generalizations of the notion of a contramodule to various
topological algebraic structures, including not only
associative rings, but also topological Lie algebras.
 In this section we demonstrate the possibility of such
a definition in the simple example of the Virasoro Lie algebra.

 The \emph{punctured formal disk}, otherwise known as 
the \emph{formal circle} over a field~$k$ is defined as
a ``space'' such that the ring of functions on it is the ring
of formal Laurent power series~$k((z))$.
 The Lie algebra of vector fields on the formal circle
$k[[z]]d/dz$ is the set of all expressions of the form
$f(z)d/dz$ with $f(z)\in k((z))$, endowed with the obvious
$k$\+vector space structure and the Lie bracket
$[f\,d/dz\;g\,d/dz] = (f\,dg/dz-g\,df/dz)\,d/dz$.
 The vector fields $L_i=z^{i+1}\,d/dz$ form a topological basis
in the vector space $k((z))d/dz$, in which the Lie bracket
takes the form $[L_i,L_j]=(j-i)L_{i+j}$.

 The \emph{Virasoro algebra} $\Vir$ is a central extension of
the Lie algebra $k((z))d/dz$ with a one-dimensional kernel
spanned by an element denoted by~$C$.
 The $k$\+vector space $\Vir=k((z))d/dz\oplus kC$ has a topological
basis formed by the vectors $L_i$, \,$i\in\Z$, and~$C$, in
which the Lie bracket is given by the rules
$$
 [L_i,C]=0, \qquad [L_i,L_j]=(j-i)L_{i+j}+
 \delta_{i+j,\.0}\frac{i^3-i}{12}\,C,
$$
where $\delta$ is the Kronecker symbol, for all $i$, $j\in\Z$
\cite{FF,RCW,KaRa}.

 A \emph{discrete module} $\M$ over the Virasoro algebra is
a module over the Lie algebra $\Vir$ for which the action map
$\Vir\times\M\rarrow\M$ is continuous in the $z$\+adic topology
of $\Vir$ and the discrete topology of~$\M$.
 In other words, $\M$ is a vector space endowed with linear
operators $L_i$ and $C\:\M\rarrow\M$ satisfying the above
commutation relations and the discreteness condition, according
to which for any vector $x\in\M$ there should exist
an integer~$n$ such that $L_ix=0$ for all $i>n$.

\begin{rem}
 The terminology related to what we call ``discrete modules over
the Virasoro algebra'' is not consistent in the literature.
 On the one hand, an analogous class of modules over locally
compact totally disconnected topological groups (such as $p$\+adic
Lie groups) is known under the name of ``smooth modules''
(see the next Section~\ref{over-top-groups} and
Example~\ref{over-semialgebras}, or the paper~\cite{Psm} and
the references therein).

 Still, in the particular case of profinite groups, such modules
are called ``discrete'' in the Galois theory~\cite{Ser} and
class field theory~\cite[Chapter~V]{CF} context, as well as in
the abstract theory of profinite groups~\cite{RZ}.
 On the other hand, an analogous class of modules over topological
associative rings is usually called ``discrete
modules''~\cite[Section~VI.4]{Ste}, \cite[end of Section~1.4]{Beil},
\cite[Section~19.1]{FG}, or even ``torsion
modules''~\cite[Section~VI.5]{Ste}.
 
 A class of modules over affine Kac--Moody Lie algebras very similar to
the above-defined class of modules over the Virasoro is called
``smooth modules'' in such references as~\cite{Yak0,Yak1,FK}.
 (This terminology in application to the Kac--Moody algebras may go back
to~\cite[Section~1.9]{KL}; notice, however, that the terminology
in~\cite{KL} is actually different, in that what are called
``smooth modules'' in~\cite{KL} are called ``strictly smooth modules''
in~\cite{Yak1}.)
 The same class of modules over the Kac--Moody algebras is called
``discrete modules'' in~\cite[Sections~5.1 and~19.1]{FG}.
 The latter reference includes a more general context of topological
Lie algebras on par with the particular case of the affine
Kac--Moody algebras.

 In this paper, we will discuss topological Lie algebras generally
in Sections~\ref{over-Lie} and~\ref{tate-harish-chandra} below
(see also Example~\ref{fully-faithful}.2 at the very end of
the paper).
 For consistency with the exposition in the author's
monograph~\cite[Section~D.2.5]{Psemi}, which is one of our main
reference sources, we prefer the ``discrete modules'' terminology
in connection with topological Lie algebras.
 The terminology in~\cite{Psemi} was largely inspired by
the one in~\cite{Beil}.
\end{rem}

 A \emph{contramodule} $\P$ over the Virasoro algebra $\Vir$
is a $k$\+vector space endowed with a linear operator
$C\:\P\rarrow\P$ and an infinite summation operation assigning
to every sequence of vectors $p_{-n}$, $p_{-n+1}$,
$p_{-n+2}$~\dots~$\in\P$, \ $n\in\Z$, a vector denoted symbolically
by $\sum_{i=-n}^\infty L_ip_i\in\P$.
 This infinitary operation, or rather, sequence of infinitary
operations indexed by the integers~$n$, should satisfy
the equations of agreement
$$
 \sum_{i=-n}^\infty L_ip_i=\sum_{i=-m}^\infty L_ip_i
 \quad\text{when $-n<-m$ and $p_{-n}=\dotsb=p_{-m-1}=0$},
$$
linearity
$$
 \sum_{i=-n}^\infty L_i(ap_i+bq_i)=
 a\sum_{i=-n}^\infty L_ip_i+ b\sum_{i=-n}^\infty L_iq_i,
$$
and the contra-Jacobi identity
$$
 \sum_{i=-n}^\infty L_i(Cp_i)\.=\.C\sum_{i=-n}^\infty L_ip_i
$$
and
\begin{multline*}
 \sum_{i=-n}^\infty L_i\left(\sum_{j=-m}^\infty L_jp_{ij}\right)
 - \sum_{j=-m}^\infty L_j\left(\sum_{i=-n}^\infty L_ip_{ij}\right)
 \\ 
 =\,\sum_{h=-n-m}^\infty L_h\left(\sum_{i+j=h}^{i\ge-n\;j\ge-m}
 (j-i)p_{ij}\right)
 +C\sum_{i+j=0}^{i\ge-n\;j\ge-m}\frac{i^3-i}{12}p_{ij}.
\end{multline*}
for any $p_i$, $q_i$, $p_{ij}\in\P$ and $a$, $b\in k$.
 This definition (for the Lie algebra $k((z))d/dz$ without
the central extension) can be found
in~\cite[Section~D.2.7]{Psemi}.
 
 For any discrete module $\M$ over the Virasoro algebra and any
$k$\+vector space $V$, the vector space $\P=\Hom_k(\M,V)$ has
a natural structure of $\Vir$\+contramodule.
 The central element $C$ acts in $\P$ by the usual
formula $(Cp)(x)=-p(Cx)$ for $p\in\P$ and $x\in\M$, while
the infinite summation operations are provided by the rule
$$
 \left(\sum_{i=-n}^\infty L_ip_i\right)(x)=
 - \sum_{i=-n}^\infty p_i(L_ix),
$$
for $p_i\in\P$ and $x\in\M$, where the second summation sign stands
for the conventional sum of an eventually vanishing sequence of
vectors in~$V$.

 The category $\Vir\discr$ of discrete modules over the Virasoro
algebra is abelian and the forgetful functor $\Vir\discr\rarrow
k\vect$ is exact.
 Both the infinite direct sums and infinite products exist in
$\Vir\discr$.
 The forgetful functor preserves infinite direct sums (but not
infinite products), so filtered inductive limits are exact
in $\Vir\discr$.
 In other words, the abelian category $\Vir\discr$ satisfies
the axioms Ab5 and Ab3*.
 It also admits a set of generators, so it has enough injectives.

 The category $\Vir\contra$ of contramodules over the Virasoro
algebra is abelian and the forgetful functor $\Vir\contra\rarrow
k\vect$ is exact.
 Both the infinite direct sums and the infinite products exist
in $\Vir\contra$.
 The forgetful functor preserves infinite products, which are
therefore exact functors in $\Vir\contra$; so this category
satisfies Ab3 and Ab4*.
 There are also enough projective objects in $\Vir\contra$.
 We will explain their construction in Section~\ref{over-Lie} below.

\subsection{Contramodules over topological groups}
\label{over-top-groups}
 The aim of this section is to demonstrate the definition of
contramodules over a locally compact, totally disconnected
topological group.
 A typical example of such a group is the group $GL_n(\mathbb Q_l)$
of invertible square matrices over the rational $l$\+adic numbers
(endowed with the topology induced by the topology of~$\mathbb Q_l$).
 
 In this section, all the \emph{topological spaces} are presumed
to be Hausdorff, locally compact, and totally disconnected.
 Open-closed subsets in such a topological space $X$ form a topology
base~\cite[Corollaire~II.4.4]{Bour}.
 A \emph{topological group} is a topological space with a group
structure given by continuous multiplication and inverse element maps.
 Open subgroups in such a topological group $G$ form a base of
neighborhoods of zero~\cite[Corollaire~III.4.6.1]{Bour}.
 When $G$ is compact, the same can be said about its open \emph{normal}
subgroups; so $G$ is profinite.

 A \emph{discrete module} $\M$ over a topological group $G$ is
an abelian group endowed with an action of $G$ provided by
a continuous map $G\times\M\rarrow\M$ in the given topology of
$G$ and the discrete topology of~$\M$.
 In other words, an action of $G$ in $\M$ is discrete if and only
if the stabilizer of any element of $\M$ is an open subgroup in~$G$.
 A discrete action can be also viewed as a map $\M\rarrow\M\{G\}$,
where for any topological space $X$ and abelian group $A$ we denote
by $A\{X\}$ the group of all locally constant $A$\+valued functions
$X\rarrow A$ on~$X$.
 Denoting by $G\modl$ the category of nontopological $G$\+modules,
i.~e., abelian groups $M$ endowed with an arbitrary action of $G$
viewed as an abstract group, and by $G\discr$ the category of
discrete $G$\+modules, there is a natural fully faithful functor
$G\discr\rarrow G\modl$.

 Let us introduce a bit more notation.
 Given a topological space $X$ and an abelian group $A$, we denote
by $A(X)$ the group of all locally constant compactly supported
$A$\+valued functions on~$X$.
 For any topological spaces $X$ and $Y$, there is a natural isomorphism
$A(X\times Y)\simeq A(X)(Y)$.
 Furthermore, denote by $A[[X]]$ the abelian group of finitely additive
compactly supported $A$\+valued measures defined on the open-closed
subsets of~$X$.
 For any continuous map of topological spaces $X\rarrow Y$,
the push-forward map $A[[X]]\rarrow A[[Y]]$ is
defined~\cite[Section~E.1.1]{Psemi}. 

 For any topological spaces $X$, $Y$ and an abelian group $A$,
there is a natural map $A[[X\times Y]]\rarrow A[[X]][[Y]]$
assigning to an $A$\+valued measure $\nu$ on $X\times Y$
the measure taking an open-closed subset $V\subset Y$ to the measure
taking an open-closed subset $U\subset X$ to the element
$\nu(U\times V)\in A$.
 This map is an isomorphism when \emph{both} the spaces $X$ and $Y$
are discrete or \emph{both} of them are compact, but not otherwise.

 A \emph{contramodule} over a topological group $G$ is an abelian
group~$\P$ endowed with a \emph{$G$\+contraaction} map
$\pi\:\P[[G]]\rarrow\P$, which can be viewed as
an \emph{integration operation} and denoted symbolically by
$$
 \pi(\mu)=\int_G g^{-1}(d\mu_g),
$$
where $d\mu_g\in\P$ denotes the value of a measure $\mu\in\P[[G]]$
on a small piece of the group $G$ containing an element $g\in G$,
while $g^{-1}(d\mu_g)\in\P$ is a small element in $\P$ obtained
by applying to $d\mu_g$ the presumed generalized action of
$g^{-1}\in G$ in~$\P$.
 
 The map $\pi$ must satisfy the following \emph{contraassociativity}
and \emph{contraunitality} equations.
 Firstly, the composition $\P[[G\times G]]\rarrow\P[[G]][[G]]
\rarrow\P$ of the above-described map $\P[[G\times G]]\rarrow
\P[[G]][[G]]$ with the iterated contraaction map
$\P[[G]][[G]]\rarrow\P[[G]]\rarrow\P$ should be equal to
the composition $\P[[G\times G]]\rarrow\P[[G]]\rarrow\P$ of
the push-forward map $\P[[G\times G]]\rarrow\P[[G]]$ with
respect to the multiplication map $G\times G\rarrow G$ with
the contraaction map $\P[[G]]\rarrow\P$,
$$
 \P[[G\times G]]\birarrow\P[[G]]\rarrow\P.
$$
 Secondly, the point measure supported at the unit element $e\in G$
and taking a prescribed value $p\in\P$ on the neighborhoods of~$e$
should be taken to the element~$p$ by the contraaction map,
$$
 \P\rarrow\P[[G]]\rarrow\P.
$$

 Given a point $g\in G$ and an element $p\in\P$, denote by
$g^{-1}(p)\in\P$ the element one obtains by applying the contraaction
map to the point measure supported at~$g$ and taking the value~$p$
on its neighborhoods.
 This rule defines a natural action of $G$ (as an abstract,
nontopological group) on any $G$\+contramodule $\P$, providing
a forgetful functor $G\contra\rarrow G\modl$
\cite[Section~E.1.3]{Psemi}, \cite[Sections~2.6 and~3.1]{Psm}.

 For any discrete $G$\+module $\M$ and an abelian group $V$,
the abelian group $\Hom_\Z(\M,V)$ has a natural $G$\+contramodule
structure.
 The contraaction map $\Hom_\Z(\M,V)[[G]]\rarrow\Hom_\Z(\M,V)$
assigns to a measure~$\mu$ the additive map taking an element
$m\in\M$ to the value of the integral
$$
 \pi(\mu)(m) = \int_G d\mu_g(gm)\.\in\. V.
$$
 The $\M$\+valued function $g\longmapsto gm$ being locally constant
on $G$ and the $\Hom_\Z(\M,V)$-valued measure~$\mu$ being
compactly supported in~$G$, the integral is well-defined
(cf.~\cite[Section~E.1.4]{Psemi}).

 The category of discrete $G$\+modules is abelian and
the forgetful functor $G\discr\rarrow\Ab$ is exact.
 Filtered inductive limits are exact functors in $G\discr$;
they are also preserved by the forgetful functor.
 In other words, the category $G\discr$ satisfies the axioms
Ab5 and Ab3*.
 It also admits a set of generators, so it has enough
injective objects.
 The category of $G$\+contramodules is abelian and the forgetful
functor $G\contra\rarrow\Ab$ is exact.
 Infinite products are exact functors in $G\contra$; they
are also preserved by the forgetful functor.
 So the category $G\contra$ satisfies the axioms Ab3 and Ab4*.
 It has enough projective objects.

 The embedding functor $G\discr\rarrow G\modl$ and the forgetful
functor $G\contra\rarrow G\modl$ have the similar properies,
as the forgetful functor $G\modl\rarrow\Ab$ preserves the inductive
and projective limits of any diagrams.
 We will see below in Section~\ref{over-semialgebras} how
discrete $G$\+modules and $G$\+contramodules can be interpreted
as \emph{semimodules} and \emph{semicontramodules} over
a certain semialgebra $\S$, opening the way to explicit
constructions of injective and projective objects in
$G\discr$ and $G\contra$.

\Section{Comodule and Contramodule Categories}
\label{co-and-contra-categories}

\subsection{Contramodules over topological rings}
\label{over-topol-rings}
 As we discussed in
Sections~\ref{over-power-series}\+-\ref{over-l-adics},
one would like to extend the definition of a contramodule
from the topological algebras dual to coalgebras over fields
to topological rings of more general nature.
 Before proceeding to present the desired definition, let us
start with reintroducing the conventional modules over a ring.

 Given a (nontopological) associative ring $R$ with unit,
one can define left $R$\+modules in the following fancy way.
 For any set $X$, denote by $R[X]$ the set of formal linear
combinations of elements of $X$ with coefficients in $R$
(i.~e., the underlying set of the free $R$\+module
generated by~$X$).
 The assignment $X\longmapsto R[X]$ is a covariant functor
from the category of sets to itself.
 The key observation is that this functor has a natural
structure of a \emph{monad}~\cite[Chapter~VI]{McL} on
the category of sets.

 In other words, for any set $X$ there is a natural map of
``opening the parentheses'' $\phi_X\:R[R[X]]\rarrow R[X]$,
assigning a formal linear combination of elements of $X$ to
a formal linear combination of formal linear combinations.
 There is also a natural map $\epsilon_X\:X\rarrow R[X]$ defined
in terms of the zero and unit elements of the ring~$R$.
 The associativity and unitality axioms of
a monad~\cite[Section~VI.1]{McL} are satisfied by these two
natural transformations.

 Given the endofunctor $R[{-}]\:\Sets\rarrow\Sets$ endowed with
the natural transformations~$\phi$ and~$\epsilon$, one can define
a left $R$\+module as an algebra/module over this monad on
the category of sets.
 In other words, a left $R$\+module $M$ is a set endowed with
a map of sets $m\:R[M]\rarrow M$ satisfying the associativity and
unitality axioms from~\cite[Section~VI.2]{McL}.
 Specifically, the two maps $\phi_M$ and $R[m]\:R[R[M]]\rarrow R[M]$
should have equal compositions with the map~$m$,
$$
 R[R[M]]\birarrow R[M]\rarrow M,
$$
while the composition of the map $\epsilon_M\:M\rarrow R[M]$ with
the map~$m$ should be equal to the identity map~$\id_M$,
$$ 
 M\rarrow R[M]\rarrow M.
$$

 Now let $\R$ be an associative topological ring with unit.
 We will have to assume that $\R$ is complete and separated, and
open right ideals form a base of neighborhoods of zero in~$\R$.
 In other words, the natural map $\R\rarrow\varprojlim\R/\J$,
where $\J$ runs over all the open right ideals, must be
a topological isomorphism.

 Notice that these are precisely the assumptions under which
the discrete right $\R$\+modules are a good category to be
assigned to~$\R$ (see the beginning of Section~\ref{over-l-adics};
cf.~\cite[Section~1.4]{Beil}).
 Even though the notion of a discrete $\R$\+module is well-defined for
any topological ring $\R$, one observes that the annihilator of
an element in a right $\R$\+module is a right ideal in~$\R$.
 So, if we are interested in discrete right $\R$\+modules, then
the collection of all the open right ideals in $\R$ is the only
aspect of the topology of $\R$ that is relevant for the definition
of such modules.
 Hence there is no loss of generality involved in assuming that $\R$
has a base of neighborhoods of zero consisting of open right ideals
when working with discrete right $\R$\+modules (and of course,
similarly for the discrete left modules and open left ideals).
 Otherwise, a discrete right $\R$\+module is the same thing as
a discrete right module over the completion of $\R$ in the new topology
with a base of neighborhoods of zero consisting of the open right ideals
in the original one.

 For any set $X$, denote by $\R[[X]]$ the set of all infinite
formal linear combinations $\sum_x r_xx$ of elements of $X$ with
the coefficients in $\R$ for which the family of coefficients
$r_x\in\R$ converges to zero in the topology of~$\R$.
 The latter condition means that for any neighborhood of zero
$U\subset\R$ the set of all $x\in X$ for which $r_x\notin U$
must be finite.
 We will endow the functor $\R[[{-}]]\:\Sets\rarrow\Sets$ with
the structure of a monad on the category of sets by defining
an ``opening of infinite parentheses'' map $\phi_X\:\R[[\R[[X]]]]
\rarrow\R[[X]]$ and a unit map $\epsilon_X\:X\rarrow\R[[X]]$.
 
 In order to define the map~$\phi_X$, one essentially has to show
that the infinite sums of products in $\R$ that one obtains after
opening the parentheses converge in the topology of~$\R$.
 That is where our assumptions about the topological ring $\R$
have to be used.
 Indeed, one has $\R[[X]]=\varprojlim_\J\R/\J[X]$, where $\J$ runs
over all the open right ideals in $\R$ (and the notation $A[X]$
for a set $X$ and an abelian group $A$ stands for the group of all
finite formal linear combinations of the elements of $X$ with
coefficients in~$A$).

 Defining the ``opening of parentheses'' map $\R/\J[\R[[X]]]\rarrow
\R/\J[X]$ does not involve any actual infinite summation,
since $\J\subset\R$ is an open right ideal.
 It remains to consider the composition $\R[[\R[[X]]]]\rarrow
\R/\J[\R[[X]]]\rarrow\R/\J[X]$ and pass to the projective limit
over~$\J$.
 The unit map~$\epsilon_X$ is easy to define; one can say that it is
the composition $X\rarrow\R[X]\rarrow\R[[X]]$.
 Checking the monad equations for the natural transformations~$\phi$
and~$\epsilon$ is straightforward.

 A \emph{left\/ $\R$\+contramodule} $\P$ is an algebra/module over
this monad on the category of sets.
 In other words, it is a set endowed with an \emph{$\R$\+contraaction}
map $\pi\:\R[[\P]]\rarrow\P$ satisfying the (contra)associativity and
unitality equations together with the natural transformations~$\phi$
and~$\epsilon$.
 Specifically, the two maps $\phi_\P$ and $\R[[\pi]]\:\R[[\R[[\P]]]]
\rarrow\R[[\P]]$ should have equal compositions with
the contraaction map~$\pi$,
$$
 \R[[\R[[\P]]]]\birarrow\R[[\P]]\rarrow\P,
$$
while the composition of the map $\epsilon_\P\:\P\rarrow\R[[\P]]$ with
the contraaction map should be equal to the identity map~$\id_\P$,
$$
 \P\rarrow\R[[\P]]\rarrow\P.
$$
 This definition can be found in~\cite[Remark~A.3]{Psemi}
and~\cite[Section~1.2]{Pweak}.

 Notice that a systematic study of a class of monads on the category
of sets, called the \emph{algebraic} monads and viewed as ``generalized
rings'', was undertaken by Durov in~\cite{Dur}.
 The definition above was in part inspired by Durov's work.
 However, the monad $X\longmapsto\R[[X]]$ is \emph{not} algebraic, as
the functor $\R[[{-}]]$ does not preserve filtered inductive limits
of sets.

 For any set $X$, the ``opening of parentheses'' map
$\pi=\phi_X\:\R[[\R[[X]]]]\rarrow\R[[X]]$ provides the set
$\R[[X]]$ with a natural left $\R$\+contramodule structure.
 The $\R$\+contra\-modules $\R[[X]]$ are called the \emph{free}
$\R$\+contramodules.
 For the reasons common to all monads~\cite[Section~VI.5]{McL},
for any set $X$ and any left $\R$\+contramodule $\Q$ there is
a natural bijection/isomorphism of abelian groups
$\Hom^\R(\R[[X]],\Q)\simeq\Hom_\Sets(X,\Q)$, where we denote by
$\Hom^\R(\P,\Q)$ the group of morphisms from an object $\P$ to
an object $\Q$ in the category of $\R$\+contramodules.

 Equivalently, an $\R$\+contramodule can be defined as a set endowed
with the following \emph{infinite summation operations}.
 For any set of indices $\{\alpha\}$, any family of elements
$p_\alpha\in\P_\alpha$, and any family of coefficients $r_\alpha\in\R$
converging to zero in the topology of $\R$, the element denoted
symbolically by $\sum_\alpha r_\alpha p_\alpha\in\P$ must be defined.
 This series of infinitary operations in the set $\P$ should
satisfy the equations of contraassociativity
$$
 \sum_\alpha r_\alpha\sum_\beta r_{\alpha\beta}p_{\alpha\beta}
 = \sum_{\alpha,\beta}(r_\alpha r_{\alpha\beta})p_{\alpha\beta}
 \quad \text{ if $r_\alpha\to0$ and $\forall\alpha$
 $r_{\alpha\beta}\to0$ in $\R$,}
$$
unitality
$$
 \sum_\alpha r_\alpha p_\alpha = p_{\alpha_0}
 \quad\text{ if the set $\{\alpha\}$ consists of
 one element $\alpha_0$ and $r_{\alpha_0}=1$},
$$
and distributivity
$$
 \sum_{\alpha,\beta} r_{\alpha\beta} p_\alpha =
 \sum_\alpha \left(\sum_\beta r_{\alpha\beta}\right) p_\alpha
 \quad \text{ if $r_{\alpha\beta}\to 0$ in $\R$}.
$$
 Here the summation over double indices $\alpha,\beta$ presumes
a set of pairs $\{(\alpha,\beta)\}$ mapping into another set
$\{\alpha\}$ by a map denoted symbolically by $(\alpha,\beta)
\longmapsto\alpha$ (i.~e., the range of possible~$\beta$'s
may depend on a chosen~$\alpha$).
 The summation sign in the parentheses in the third equation
denotes the convergent sum in $\R$, while all the other
summation signs stand for the infinite summation operation in~$\P$.
 Our conditions on the topology of $\R$ guarantee that
the family $r_\alpha r_{\alpha\beta}$ converges to zero whenever both
the family $r_\alpha$ does and the families $r_{\alpha\beta}$
do for every fixed~$\alpha$.

 Restricting the summation operations to finite sets of indices
$\{\alpha\}$, one discovers that every left $\R$\+contramodule has
an underlying left $\R$\+module structure.
 Equivalently, one composes the contraaction map $\R[[\P]]\rarrow\P$
with the natural embedding $\R[\P]\rarrow\R[[\P]]$ in order to endow
the underlying set of an $\R$\+contramodule $\P$ with the structure
of an $\R$\+module.
 We have constructed the forgetful functor $\R\contra\allowbreak
\rarrow\R\modl$ from the category of left $\R$\+contramodules
$\R\contra$ to the category $\R\modl$ of left modules over
the ring $\R$ viewed as an abstract (nontopological) ring.

 For any discrete right $\R$\+module $\N$ and any abelian group $V$,
the group of all additive maps $\Hom_\Ab(\N,V)$ has a natural left
$\R$\+contramodule structure with the infinite summation operations
defined by the rule
$$
 \left(\sum_\alpha r_\alpha p_\alpha\right)(x)=
 \sum_\alpha p_\alpha(x r_\alpha)
$$
for any $p_\alpha\in\P$, \,$x\in\N$, and a family of
coefficients~$r_\alpha$ converging to zero in the topology of~$\R$.
 Here the summation sign in the right-hand side denotes the sum
of a family of elements in $V$ all but a finite subfamily of
which vanish, as $xr_\alpha=0$ for all but a finite subset of
indices~$\alpha$.

 For any topological ring $\R$ the category $\R\discr$ of discrete
left $\R$\+modules is abelian and the forgetful functor $\R\discr
\rarrow\Ab$ is exact.
 Filtered inductive limits are exact functors in $\R\discr$; they
are also preserved by the forgetful functor.
 In other words, the category $\R\discr$ satisfies the axioms
Ab5 and Ab3* (but \emph{not} in general Ab4*).
 It also admits a set of generators, so it has enough injectives.

 For any complete and separated topological ring $\R$ with a base
of neighborhoods of zero formed by the open right ideals
the category of left $\R$\+contramodules is abelian and
the forgetful functor $\R\contra\rarrow\Ab$ is exact.
 To convince oneself that this is so, one uses the definition
of $\R$\+contramodules in terms of infinite summation operations
in order to define the $\R$\+contramodule structures on
the kernel and cokernel of any morphism of $\R$\+contramodules
taken in the category of abelian groups.
 It helps to start from writing down the equations of compatibility
of the contramodule infinite summation operations with
the conventional finite operations in an abelian group or
an $\R$\+module~\cite[Section~1.2]{Pweak}.

 Infinite products are exact functors in $\R\contra$; they are
also preserved by the forgetful functor.
 There are enough projective objects in $\R\contra$;
an $\R$\+contramodule is projective if and only if it is
a direct summand of a free one.
 Infinite direct sums of free $\R$\+contramodules are computed
by the rule $\bigoplus_\alpha\R[[X_\alpha]]=
\R[[\coprod_\alpha X_\alpha]]$; to compute the infinite direct
sum of a family of arbitrary $\R$\+contramodules, one can
present them as the cokernels of morphisms of free contramodules
and use the fact that infinite direct sums commute with cokernels.
 Hence the category $\R\contra$ satisfies the axioms Ab3 and
Ab4* (but \emph{not} in general Ab4).

 The infinite products of discrete $\R$\+modules and the infinite
direct sums of $\R$\+con\-tramodules are not preserved by
the respective forgetful functors in general.
 The embedding functor $\R\discr\rarrow\R\modl$ and the forgetful
functor $\R\contra\rarrow\R\modl$ have the similar properties,
as the forgetful functor $\R\modl\rarrow\Ab$ preserves
the inductive and projective limits of any diagrams.

 The following version of \emph{Nakayama's lemma for discrete modules
and contramodules} over a topological ring is one of their most
important properties.

\begin{lem}
\textup{(a)} Let\/ $\R$ be a topological ring and\/ $\m\subset\R$
be a topologically nilpotent ideal, i.~e., for any neighborhood of
zero $U\subset\R$ there exists an integer $n\ge1$ such that\/
$\m^n\subset U$.
 Then for any nonzero discrete left\/ $\R$\+module\/ $\M$
the submodule ${}_\m\M\subset\M$ of elements annihilated by\/~$\m$
is nonzero. \par
\textup{(b)} Let\/ $\R$ be a complete, separated topological ring with
a base of neighborhoods of zero formed by the open right ideals, and
let\/ $\m\subset\R$ be a topologically nilpotent closed ideal.
 Then for any nonzero left\/ $\R$\+contramodule\/ $\P$
the quotient contramodule $\P/\m\P$ of\/ $\P$ by the image\/ $\m\P$
of the contraaction map $\m[[\P]]\rarrow\P$ is nonzero.
\end{lem}

 Here the map $\m[[\P]]\rarrow\P$ is simply the restriction of
the contraaction map $\pi\:\R[[\P]]\rarrow\P$ to the subset
$\m[[\P]]\subset\R[[\P]]$ of all formal linear combinations with
(converging families of) coefficients in~$\m$.
 Notice that this version of Nakayama's lemma presumes \emph{no}
finite generatedness condition on either the discrete module \emph{or}
the contramodule; on the other hand, it requires a rather strong
global topological nilpotency condition on the ideal~$\m$.

\begin{proof}
 Part~(a): let $x\in\M$ be a nonzero element and $U\subset\R$ 
be its annihilator in~$\R$.
 The $\R$\+module $\M$ being discrete, $U$ is an open neighborhood
of zero in $\R$; hence there exists an integer $n\ge1$ such that
$\m^n\subset U$, so $\m^nx=0$.
 It remains to consider the maximal integer $i\ge0$ for which
$\m^ix\ne0$; then $\m^ix\subset{}_\m\M$.
 The proof of part~(b) is a bit more complicated;
see~\cite[Lemma~A.2.1]{Psemi} and~\cite[Lemma~1.3.1]{Pweak}.
\end{proof}

 For a generalization of the lemma to topologicaly T\+nilpotent ideals,
see~\cite[Lemmas~6.1 and~6.2]{Pproperf}.
 For another version of contramodule Nakayama lemma,
see~\cite[Lemma~D.1.2]{Pcosh}, \cite[Lemma~6.14]{PR},
and/or Lemma~1 in Section~\ref{fully-faithful} below.

\subsection{Contramodules over the adic completions of Noetherian rings}
\label{over-adic-completions}
 Let $R$ be a right Noetherian associative ring, and let $I\subset R$
be an ideal generated by central elements in~$R$.
 Denote by $\R=\varprojlim_n R/I^n$ the $I$\+adic completion of
the ring~$R$.
 In this section we explain how to describe the abelian category
$\R\contra$ of left contramodules over the complete ring $\R$ viewed
as a topological ring in the projective limit topology (or, which is
the same, the $I$\+adic topology) in terms of conventional modules
over the original ring~$R$.

\begin{thm}
 The composition of forgetful functors\/ $\R\contra\rarrow\R\modl
\rarrow R\modl$ provides a fully faithful embedding of abelian
categories\/ $\R\contra\rarrow R\modl$.
 A left $R$\+module $P$ belongs to the full subcategory\/
$\R\contra\subset R\modl$ if and only if any one of the following
equivalent conditions holds: \par
\textup{(a)} for any element $s\in I$ belonging to the center of
the ring $R$ and any $R[s^{-1}]$\+module $L$ one has\/
$\Ext_R^i(L,P)=0$ for all $i\ge0$; \par
\textup{(b)} for any element $s\in I$ belonging to the center of
the ring $R$ one has\/ $\Ext_R^i(R[s^{-1}],P)=0$ for all $i\ge0$; \par
\textup{(c)} for any element $s\in I$ one has\/
$\Ext^*_{\Z[t]}(\Z[t,t^{-1}],P)=0$, where\/ $\Z[t]$ denotes the ring
of polynomials in one variable with integral coefficients,
$\Z[t,t^{-1}]$ is the ring of Laurent polynomials, and\/ $t$~acts
in $P$ by the multiplication with~$s$; \par
\textup{(d)} for any $j=1$,~\dots,~$n$ and any $i=0$ or~$1$ one has\/
$\Ext_R^i(R[s_j^{-1}],P)=0$, where $s_1$,~\dots, $s_j$ is a fixed set
of central generators of the ideal $I\subset R$; \par
\textup{(e)} for any $j=1$,~\dots,~$n$ and any $i=0$ or~$1$ one has\/
$\Ext^i_{\Z[t]}(\Z[t,t^{-1}],P)=0$, where $t$~acts in $P$
by the multiplication with~$s_j$.
\end{thm}

 In other words, an $\R$\+contramodule structure on a given left
$\R$\+module is always \emph{unique}, and the theorem lists
equivalent conditions telling when it \emph{exists}.
 Of course, for contramodules over topological rings more complicated
then the adic completions no such description is in general possible.

 In the case of a commutative ring $R$, the theorem essentially says
that a contramodule over the $I$\+adic completion of $R$ is
the same thing as a \emph{cohomologically $I$\+adically complete
$R$\+module} of Porta--Shaul--Yekutieli~\cite{PSY,PSY2,Yek2}.
 A very brief sketch of the proof of the above theorem is presented
below; a detailed exposition can be found in~\cite[Appendix~B]{Pweak}
and~\cite[Section~C.5]{Pcosh} (see also~\cite[Remark~A.1.1]{Psemi}).

\begin{proof}[Sketch of proof]
 First let us explain why any left $R$\+module $P$ admitting a left
$\R$\+con\-tramodule structure satisfies the conditions~(a) and~(c).
 The choice of an element $s\in I$ defines a continuous homomorphism
of topological rings $\Z[[t]]\rarrow\R$, thus endowing any left
$\R$\+contramodule $\P$ with a left $\Z[[t]]$\+contramodule structure.
 One checks that this structure is inherited by the groups
$\Ext^i_{\Z[t]}(L,\P)$ for any $\Z[t]$\+module $L$ and, when
$s$ is central in $\R$, by the groups $\Ext^i_R(L,\P)$ for any
$R$\+module~$L$.
 Now when $t$ or~$s$ acts invertibly in $L$, the Ext groups in
question turn out to be $\Z[[t]]$\+contramodules with an invertible
action of~$t$, which have to vanish by the Nakayama
Lemma~\ref{over-topol-rings}(b) above (cf.~\cite[Section~B.2]{Pweak}).

 Furthermore, for any central element $s\in R$ and any left
$R$\+module $P$ one has $\Ext_R^i(R[s^{-1}],P)\simeq
\Ext_{\Z[t]}^i(\Z[t,t^{-1}],P)$ and, of course, both groups always
vanish for $i>1$ \cite[Lemma~B.7.1]{Pweak}.
 It remains to show that any left $R$\+module $P$ satisfying
the condition~(e) can be endowed with a left $\R$\+contramodule
structure in a unique way.
 This is accomplished by the following sequence of lemmas.
 
 Consider the topological ring of formal power series
$\T=R[[t_1,\dotsc,t_n]]$ in the central variables
$t_1$,~\dots, $t_n$ with coefficients in~$R$; then there is
a natural continuous ring homomorphism $\T\rarrow\R$
taking $t_j$ to~$s_j$.
 Consider also the ring of polynomials $T=R[t_1,\dotsc,t_n]$ and
the similar ring homomorphism $T\rarrow R$.

\begin{lem1}
 The ring homomorphism\/ $\T\rarrow\R$ is surjective, and its kernel
$\J$ is generated by the central elements $s_j-t_j$ as an ideal in
an abstract (nontopological) ring\/~$\T$.
 Moreover, any family of elements converging to zero in\/ $\R$ can be
lifted to a family of elements converging to zero in\/~$\T$, and
any family of elements in\/~$\J$ converging to zero in
the topology of\/ $\T$ can be presented as a linear combination of\/
$n$~families of elements in\/ $\T$, each of them converging to zero,
with the coefficients $s_j-t_j$.
\end{lem1} 

\begin{proof}
 This is where the Noetherianness condition on the ring $R$
is being used; see~\cite[Sections~B.3\+-B.4]{Pweak}
and~\cite[Lemma~C.5.2]{Pcosh}.
\end{proof}

\begin{lem2}
 The (contra)restriction of scalars functor $\R\contra\rarrow
\T\contra$ identifies the category of left $\R$\+contramodules with
the full subcategory in $\T\contra$ consisting of those left
$\T$\+contramodules in which the elements $s_j-t_j\in\T$ act by zero. 
\end{lem2}

\begin{proof}
 Follows from Lemma~1; see~\cite[Lemma~B.4.1]{Pweak}.
\end{proof}

 It is easy to interpret left $\T$\+contramodules as left
$R$\+modules endowed with infinite summation operations with
the coefficients $t_1^{m_1}\dotsm t_n^{m_n}$
(see~\cite[proof of Lemma~B.5.1]{Pweak}; cf.\
Section~\ref{over-power-series} above).
 Hence it follows from Lemma~2 that the definition of
contramodules over the $l$\+adic integers given in
Section~\ref{over-l-adics} is equivalent to the general
definition from Section~\ref{over-topol-rings} specialized
to the case of $\R=\Z_l$.

\begin{lem3}
 The forgetful functor\/ $\T\contra\rarrow T\modl$ identifies
the category of left\/ $\T$\+contramodules with the full
subcategory in the category of left $T$\+modules consisting
of all those modules $Q$ for which\/ $\Ext_T^*(T[t_j^{-1}],Q)=0$
for every $j=1$,~\dots,~$n$.
\end{lem3}

\begin{proof}
 The ``unique recovering'' argument here is just a more elaborated
version of the reasoning from Section~\ref{recovering} above.
 See~\cite[Sections~B.5\+-B.7]{Pweak} and~\cite[Lemma~C.5.3]{Pcosh}
for the details.
\end{proof}

 To finish the proof of the theorem, it remains to combine together
the results of Lemmas~2 and~3.
\end{proof}

 Denote by $\I=\varprojlim_n I/I^n\subset\R$ the extension of
the ideal~$I$ in the ring~$\R$.
 The following result explains the term ``cohomologically complete
module'' for left $R$\+modules satisfying the equivalent conditions
of Theorem.

\begin{prop1}
 For any left\/ $\R$\+contramodule\/ $\P$, the natural map\/
$\P\rarrow\varprojlim_n \P/\I^n\P\allowbreak=\varprojlim_n\P/I^n\P$
is surjective. \emergencystretch=0em
\end{prop1}

\begin{proof}
 See~\cite[Lemma~A.2.3]{Psemi} or~\cite[Lemma~D.1.1]{Pcosh}.
\end{proof}

 Notice that the natural functor $\R\discr\rarrow R\modl$ is fully
faithful for any topological ring $R$ with a base of neighborhoods
of zero consisting of open left ideals $J$ and its completion
$\R=\varprojlim_J R/J$.
 Moreover, when $\R=\varprojlim_n R/I^n$ is the completion of
the ring $R$ in the adic topology of an ideal $I\subset R$ generated
by a finite set of central elements~$s_j$, an $R$\+module $M$
belongs to the full subcategory $\R\discr\subset R\modl$ if and only
if it is \emph{$I$\+torsion}, i.~e., one has $R[s_j^{-1}]\ot_RM=0$
for all~$j$ or, which is the same, $\Tor^R_*(R[s_j^{-1}],M)=0$
for all~$j$.

 Let us point out that the class of $I$\+adically complete and
separated left $R$\+modules, i.~e., left $R$\+modules $P$ for which
the map $P\rarrow\varprojlim_nP/I^nP$ is an isomorphism, does
\emph{not} have good homological properties.
 Indeed, it is \emph{not} preserved not only by the passages to
the cokernels of injective morphisms (see
Section~\ref{counterexamples}), but \emph{also} by extensions in
$R\modl$ or $\R\contra$ \cite[Example~2.5]{Sim}.
 The full subcategory of left $\R$\+contramodules $\R\contra\subset
R\modl$, on the other hand, not only contains all
the $I$\+adically complete and separated left $R$\+modules, but
is also closed under the kernels, cokernels, extensions, and
projective limits in $R\modl$.

\begin{ex}
 Let $\R$ be a complete Noetherian commutative local ring with
the maximal ideal~$\m$.
 Let $E$ be an injective envelope of the residue field $\R/\m$ in
the abelian category of $\R$\+modules.
 Then the \emph{Matlis duality} (see~\cite[Corollary~4.3]{Mat1}
or~\cite[Theorem~18.6]{Mats}) is an anti-equivalence between the abelian
categories of Artinian and Noetherian $\R$\+modules provided by
the contravariant functor $\Hom_\R({-},E)$.

 Endow the ring $\R$ with the $\m$\+adic topology.
 Then any Artinian $\R$\+module is discrete (or equivalently,
$\m$\+torsion), while any Noetherian $\R$\+module is
an $\R$\+contramodule.
 It was essentially explained in Section~\ref{over-topol-rings} that,
for any discrete $\R$\+module $\N$ and any $\R$\+module $V$
the $\R$\+module $\Hom_\R(\N,V)$ has a natural $\R$\+contramodule
structure.
 Thus the construction of the $\R$\+contramodule structure on
the dual abelian group or module to a discrete $\R$\+module can be
viewed as an extension of the Matlis duality to modules with no
finiteness conditions imposed.
\end{ex}

 In the respective assumptions on a topological ring $\R$, let us
denote by $\Ext_\R^*(\L,\M)$ the Ext groups in the abelian category
$\R\discr$ and by $\Ext^{\R,\.*}(\P,\Q)$ the Ext groups in the abelian
category $\R\contra$.
 The next proposition shows that the embeddings of abelian categories
$\R\discr\rarrow R\modl$ and $\R\contra\rarrow R\modl$ have
good homological properties.

\begin{prop2}
\textup{(a)} Let $R$ be a left Noetherian ring, $I\subset R$ be
an ideal generated by central elements, and\/ $\R=\varprojlim_nR/I^n$
be the $I$\+adic completion of~$R$.
 Then the embedding functor\/ $\R\discr\rarrow R\modl$ induces
isomorphisms on all the Ext groups, $\Ext^i_\R(\L,\M)\simeq
\Ext^i_R(\L,\M)$ for all\/ $\L$, $\M\in\R\discr$ and all $i\ge0$. \par
\textup{(b)} Let $R$ be a right Noetherian ring, $I\subset R$
be an ideal generated by central elements, and $\R$ be
the $I$\+adic completion of~$R$.
 Then the embedding functor\/ $\R\contra\rarrow R\modl$ induces
isomorphisms on all the Ext groups, $\Ext^{\R,\.i}(\P,\Q)\simeq
\Ext_R^i(\P,\Q)$ for all\/ $\P$, $\Q\in\R\contra$ and all $i\ge0$.
\hbadness=1300
\end{prop2}

\begin{proof}
 Part~(a): it follows from the Artin--Rees lemma
(see~\cite[Theorem~8.5]{Mats} and~\cite[Theorems~1.9 and~13.3]{GW})
that the functor $\R\discr\rarrow R\modl$ preserves injectivity
of objects (cf.~\cite[Section~A.3]{Psing}), which is clearly sufficient.
 Part~(b): one shows that the functor $\R\contra\rarrow R\modl$
takes free $\R$\+contramodules to flat $R$\+modules and
all $\R$\+contramodules to relatively cotorsion $R$\+modules
(see~\cite[Sections~B.8\+-B.10]{Pweak}
and~\cite[Propositions~C.5.4\+-C.5.5]{Pcosh}).
\end{proof}

\begin{rem}
 In fact, the first assertion of the theorem (about the functor
$\R\contra\rarrow R\modl$ being fully faithful) holds without
the Noetherianity assumption on the ring~$R$.
 It suffices to assume that $I\subset R$ is an ideal generated by
a finite number of central elements, or even just a two-sided ideal
which, viewed as a right ideal, is finitely generated.
 Cf.\ the discussion of fully faithful contramodule forgetful
functors in Section~\ref{fully-faithful} below.

 On the other hand, the second assertion of the theorem (providing
a description of the essential image of this forgetful functor)
holds for a finitely centrally generated ideal $I\subset R$
satisfying a weak version of the \emph{weak proregularity}
condition~\cite[Examples~2.2\,(3) and~2.3\,(3)]{Pper},
\cite[Remark~3.8]{Pdc}.
 But generally speaking, this description of the essential image of
the functor $\R\contra\rarrow R\modl$ does \emph{not} hold for
arbitrary finitely generated ideals $I$ in commutative rings $R$
\cite[Example~5.2\,(8)]{Pper}, \cite[Examples~1.8]{Pdc}.

 Furthermore, the assertions of Proposition~2 hold, without
the Noetherianity assumption, for any ideal $I\subset R$ generated
by a weakly proregular finite sequence of central elements.
 Moreover, the natural functors between the unbounded derived
categories $\sD(\R\discr)\rarrow\sD(R\modl)$ and $\sD(\R\contra)
\rarrow\sD(R\modl)$ are fully faithful under these assumptions
\cite[Sections~1\+-2]{Pmgm}, \cite[Example~5.3\,(2)]{Pper},
\cite[Section~4]{Pdc}.
\end{rem}

\subsection{Contramodules over topological algebras over fields}
\label{over-top-algebras}
 The ``set-theoret\-ical'' definition of $\R$\+contramodules given
in Section~\ref{over-topol-rings} is intended to incorporate
``arithmetical'' examples such as that of the ring $\R=\Z_l$
of $l$\+adic integers.
 In the case of a topological algebra $\R$ over a field~$k$,
the definition can be simplified, facilitating the comparision
with the notion of a contramodule over a coalgebra $\C$ over~$k$.

 A topological vector space $V$ over a field~$k$ is said to have
a \emph{linear topology} if its open vector subspaces form a base
of neighborhoods of zero in it.
 In the sequel, we presume all our topological vector spaces to have
linear topologies and, unless otherwise mentioned, to be complete
and separated.
 In other words, the natural map $V\rarrow\varprojlim_UV/U$, where
$U$ runs over all the open vector subspaces in~$V$, should be
a topological isomorphism (see~\cite{Beil}
or~\cite[Section~D.1]{Psemi}).
 Given a topological vector space $V$ and an abstract (nontopological)
vector space $P$ over a field~$k$, we denote by $V\ot\comp P$
the projective limit $\varprojlim_U V/U\ot_k P$ taken over all
the open vector subspaces $U\subset V$, viewed as an abstract
(nontopological) vector space.

 For any associative algebra $R$ over a field~$k$, one can define
left $R$\+modules as modules over the monad $M\longmapsto R\ot_k M$
on the category of $k$\+vector spaces $M\in k\vect$.
 We would like to extend this definition of $R$\+modules to the case
of topological algebras over~$k$.
 Let $\R$ be a complete and separated topological associative algebra
over a field~$k$ where open right ideal form a base of neighborhoods
of zero.
 Then the functor $P\longmapsto\R\ot\comp P$ has a natural structure
of a monad on the category of (nontopological) $k$\+vector spaces
$P\in k\vect$.
 Indeed, let us construct the natural transformations of
multiplication and unit in this monad.

 For any open right ideal $\J\subset\R$ the multiplication
map $\R/\J\times\R\rarrow\R/\J$, being continuous in the discrete
topology of $\R/\J$ and the given topology of $\R$, defines
a structure of discrete right $\R$\+module on the quotient space
$\R/\J$, so the annihilator of every element of $\R/\J$ is an open
right ideal in~$\R$.
 Hence the multiplication map $\R/\J\ot_k\R\rarrow\R/\J$ induces
a natural linear map $\R/\J\ot_k(\R\ot\comp P)\rarrow\R/\J\ot_k P$.
 Composing this map with the projection $\R\ot\comp(\R\ot\comp P)
\rarrow\R/\J\ot_k(\R\ot\comp P)$ and passing to the projective
limit over open right ideals $\J$, we obtain the desired monad
multiplication map $\phi_P\:\R\ot\comp(\R\ot\comp P)\rarrow
\R\ot\comp P$.

 The unit map $\epsilon_P\:P\rarrow\R\ot\comp P$ of our monad
is obtained as the composition of the map $P\rarrow\R\ot_k P$
induced by the unit element in $\R$ with the completion map
$\R\ot_k P\rarrow\R\ot\comp P$.
 Verifying the associativity and unitality axioms of a monad for
the functor $\R\ot\comp{-}\:\allowbreak k\vect\rarrow k\vect$
endowed with the natural transformations~$\phi$ and~$\epsilon$
is straightforward.

 A \emph{left\/ $\R$\+contramodule} $\P$ is an algebra/module over
this monad on the category of $k$\+vector spaces.
 In other words, it is a $k$\+vector space endowed with
a \emph{contraaction} map $\pi\:\R\ot\comp\P\rarrow\P$ satisfying
the following \emph{contraassociativity} and contraunitality equations.
 Firstly, the two maps $\phi_\P$ and $\R\ot\comp\pi\:
\R\ot\comp(\R\ot\comp\P)\rarrow\R\ot\comp\P$ should have equal
compositions with the contraaction map~$\pi$,
$$
 \R\ot\comp(\R\ot\comp\P)\birarrow\R\ot\comp\P\rarrow\P.
$$
 Secondly, the composition of the map $\epsilon_\P\:\P\rarrow
\R\ot\comp\P$ with the contraaction map should be equal to
the identity map~$\id_\P$,
$$
 \P\rarrow\R\ot\comp\P\rarrow\P.
$$
 A \emph{free\/ $\R$\+contramodule} is an $\R$\+contramodule of
the form $\P=\R\ot\comp P$, where $P$ is a $k$\+vector space, with
the contraaction map $\pi=\phi_P\:\R\ot\comp(\R\ot\comp P)\rarrow
\R\ot\comp P$.
 This definition of $\R$\+contramodules can be found
in~\cite[Section~D.5.2]{Psemi}.

 For any discrete right $\R$\+module $\N$ and any (nontopological)
$k$\+vector space $E$, the vector space $\P=\Hom_k(\N,E)$ has
a natural left $\R$\+contramodule structure provided by
the contraaction map $\pi\:\R\ot\comp\Hom_k(\N,E)\rarrow\Hom_k(\N,E)$
defined symbolically by the formula
$$
 \pi(r\ot\comp f)(n) = f(nr),
$$
where $r\in\R$, \,$n\in\N$, \,$f\in\Hom_k(\N,E)$, and the expression
in the right-hand side makes sense, since the right action map
$\N\ot_k\R\ot_k\Hom_k(\N,E)\rarrow\N\ot_k\Hom_k(\N,E)$ restricted to
$n\ot\R\ot\P\subset\N\ot_k\R\ot_k\P$ factorizes through
the surjection $n\ot\R\ot\P\rarrow n\ot\R/\J\ot\P$ for
a certain open right ideal $\J\subset\R$.

 Let us show that our new definition of $\R$\+contramodules is
equivalent to the one from Section~\ref{over-topol-rings} in
the case of a topological algebra $\R$ over a field~$k$.
 The following argument can be found in~\cite[Section~1.10]{Pweak}.

 Recall that the category of $k$\+vector spaces can be defined
as the category of algebras/modules over the monad
$X\longmapsto k[X]$ on the category of sets.
 Hence for any $k$\+vector space $P$ there is a natural action map
$p\:k[P]\rarrow P$.
 Moreover, this map is the coequalizer of the pair of maps
$k[[P]]\birarrow k[P]$, one of which is the ``opening of
parentheses'' map~$\phi_P$, while the other one is the map
$k[p]$ induced by the map~$p$.
 Indeed, applying the forgetful functor $k\vect\rarrow\Sets$
makes this even a split coequalizer with an explicit splitting
defined in terms of the unit map of our
monad~\cite[Sections~VI.6\+-7]{McL}.
 Subtracting one of the maps in the pair from the other one, we
obtain an exact sequence in the category $k\vect$
\begin{equation}  \label{vector-space-monad-equalizer}
 k[k[P]]\rarrow k[P]\rarrow P\rarrow 0
\end{equation}
for any $k$\+vector space~$P$.

 Notice the natural isomorphism $\R\ot\comp k[X]\simeq\R[[X]]$
for any set~$X$.
 Furthermore, any additive functor on the category of $k$\+vector
spaces is exact.
 Thus applying the functor $\R\ot\comp{-}$ to the exact
sequence~\eqref{vector-space-monad-equalizer}, we obtain
an exact sequence
\begin{equation}  \label{contra-top-algebra-comparison}
 \R[[k[P]]]\lrarrow\R[[P]]\lrarrow\R\ot\comp P\lrarrow0
\end{equation}
for any $k$\+vector space~$P$.
 In particular, we have obtained a natural surjective map
$\R[[\P]]\rarrow\R\ot\comp\P$ for any $k$\+vector space~$\P$;
composing it with the contraaction map $\R\ot\comp\P\rarrow\P$
of a contramodule $\P$ over the topological $k$\+algebra $\R$,
we obtain a contraaction map $\R[[\P]]\rarrow\P$ defining
the structure of a contramodule over the topological ring
$\R$ on the set~$\P$.

 Conversely, starting from the contraaction map $\R[[\P]]\rarrow\P$
of a contramodule $\P$ over the topological ring $\R$, one can
first of all compose it with the natural embedding $k[\P]\rarrow
\R[[\P]]$, defining a $k$\+vector space structure $k[\P]\rarrow\P$
on the set~$\P$.
 Furthermore, restricting the contraassociativity equation
$\R[[\R[[\P]]]]\birarrow\R[[\P]]\rarrow\P$ to the subset
$\R[[k[\P]]]\subset\R[[\R[[\P]]]]$, one discovers that the two
maps $\R[[k[\P]]]\birarrow\R[[\P]]$ have equal compositions
with the contraaction map $\R[[\P]]\rarrow\P$.
 So we see from the exact
sequence~\eqref{contra-top-algebra-comparison}
that the contraaction map $\R[[\P]]\rarrow\P$ factorizes through
the surjective map $\R[[\P]]\rarrow\R\ot\comp\P$, providing $\P$
with a contraaction map $\R\ot\comp\P\rarrow\P$ of a contramodule
over the topological $k$\+algebra~$\R$.

 Of course, one still has to check that the map $\R\ot\comp\P\rarrow
\P$ satisfies the contraassociativity and contraunitality equations
if and only if the corresponding map $\R[[\P]]\rarrow\P$ does.
 Here it helps to notice that the natural map
\begin{equation} \label{contra-top-algebra-surjective}
 \R[[\R[[\P]]]]\lrarrow\R\ot\comp(\R\ot\comp\P)
\end{equation}
is surjective, so any two maps
$\R\ot\comp(\R\ot\comp\P)\rarrow\P$
are equal to each other whenever their compositions with
the map~\eqref{contra-top-algebra-surjective} are.
 We have also seen that the class of free $\R$\+contramodules as
defined in this section coincides with the one introduced in
Section~\ref{over-topol-rings} when $\R$ is a topological algebra
over a field.

 Now we can finally compare our notion of a contramodule over
a topological ring/topological algebra over a field~$k$ with
the definition of a contramodule over a (coassociative)
coalgebra $\C$ over~$k$ given in Section~\ref{over-coalgebras}.
 Let $\R=\C^*$ be the dual vector space to the coalgebra~$\C$
endowed with its pro-finite-dimensional topological algebra
structure (see Sections~\ref{over-power-series}\+-\ref{over-l-adics}).
 Then there is a natural isomorphism $\R\ot\comp P\simeq\Hom_k(\C,P)$
for any $k$\+vector space $P$, making an $\R$\+contraaction map
$\R\ot\comp\P\rarrow\P$ the same thing as a $\C$\+contraaction map
$\Hom_k(\C,\P)\rarrow\P$.

 The vector spaces $\R\ot\comp(\R\ot\comp\P)$ and
$\Hom_k(\C,\Hom_k(\C,\P))$ parametrizing the systems of
contraassociativity equations on the two kinds of contraaction
maps being also naturally isomorphic, one easily checks that
a map $\R\ot\comp\P\rarrow\P$ defines a left $\R$\+contramodule
structure on a $k$\+vector space $\P$ if and only if
the corresponding map $\Hom_k(\C,\P)\rarrow\P$ defines a left
$\C$\+contramodule structure on~$\P$.

\subsection{Contramodules over topological Lie algebras}
\label{over-Lie}
 The definition of the category of contramodules over
the Virasoro algebra given in Section~\ref{over-virasoro}
calls for a generalization to a reasonably large class of
topological Lie algebras.

 There are some na\"\i ve approaches: for example, it is easy to
define comodules and contramodules over Lie coalgebras in
the way analogous to the definitions for coassociative
coalgebras explained in Section~\ref{over-coalgebras}.
 This provides the notion of a contramodule over a linearly
compact topological Lie algebra.
 Notice that the class of all Lie coalgebras is in some sense
not as narrow as that of coassociative coalgebras over fields:
unlike in the coassociative case (see
Section~\ref{over-power-series}), a Lie coalgebra does not
have to be the union of its finite-dimensional subcoalgebras.

 Indeed, it suffices to consider the case of the Lie coalgebra $\L$
dual to the linearly compact Lie algebra $k[[z]]d/dz$ over
a field~$k$ of zero characteristic, that is the Lie subalgebra
topologically spanned by the generators $L_{-1}$, $L_0$,
$L_1$, $L_2$~\dots{} in the algebra $k((z))d/dz$.
 The Lie algebra $k[[z]]d/dz$ having no nonzero proper
closed ideals, the Lie coalgebra $\L$ has no nonzero proper
subcoalgebras at all.
 Nevertheless, the class of linearly compact Lie algebras does not
even contain the Virasoro algebra.

 So let us start with the class of locally linearly compact, or
\emph{Tate} Lie algebras.
 A topological vector space $V$ is said to be \emph{locally linearly
compact}, or a \emph{Tate vector space} if it has a linearly
compact open subspace, or equivalently, if (linearly) compact
open subspaces form a base of neighborhoods of zero in~$V$.
 In other words, a topological vector space is a Tate vector space
if it is topologically isomorphic to the direct sum of a compact
vector space and a discrete vector space (see~\cite[Sections~1.1\+-1.2
and the references therein]{Beil} and~\cite[Section~D.1.1]{Psemi}).

 A \emph{Tate Lie algebra} $\g$ is a Tate vector space endowed with
a continuous Lie algebra structure, i.~e., a Lie bracket $\g\times\g
\rarrow\g$ that is continuous as a function of two variables.
 Any Tate Lie algebra has a base of neighborhoods of zero consisting
of open Lie subalgebras (see footnotes in~\cite[Section~3.8.17]{BD}
or~\cite[Section~1.4]{Beil}, or a paragraph
in~\cite[Section~D.1.8]{Psemi}).
 For example, the ``Laurent totalization'' $\g=\bigoplus_{n<0}\g_n
\oplus\prod_{n\ge0}\g_n$ of any $\Z$\+graded Lie algebra
$\bigoplus_{n\in\Z}\g_n$ with finite-dimensional components~$\g_n$ 
is a Tate Lie algebra with compact open subalgebras $\prod_{i\ge n}\g_i
\subset\g$, \,$n\ge0$, forming a base of neighborhoods of zero.
 This includes such classical examples as the Virasoro and
Kac--Moody Lie algebras.

 A \emph{contramodule} $\P$ over a Tate Lie algebra $\g$ over
a field~$k$ is a $k$\+vector space endowed with a \emph{contraaction}
map $\g\ot\comp\P\rarrow\P$ satisfying the following (system of)
contra-Jacobi equation(s).
 Given a compact vector space $V$, denote by $V\dual$ the discrete
vector space to dual to $V$, so that $V\dual{}^*$ is isomorphic to
$V$ as a topological vector space.
 For any abstract (nontopological) vector space $P$, there is then
a natural isomorphism of (nontopological) vector spaces
$V\ot\comp P\simeq\Hom_k(V\dual,P)$.
 So in particular we have $\g\ot\comp\P\simeq\varinjlim_V
\Hom_k(V\dual,\P)$, where $V$ runs over all the compact vector
subspaces $V\subset\g$.

 Now let $U$, $V$, and $W\subset\g$ be three compact vector
subspaces for which $[U,V]\subset W$; then there is a natural
cobracket map $W\dual\rarrow V\dual\ot_k U\dual$.
 It required that the composition
$$
 \Hom_k(V\ot_kU\;\P)\lrarrow\Hom_k(W,\P)\lrarrow\g\ot\comp\P
 \lrarrow\P
$$
of the map induced by the cobracket map with the contraaction map
should be equal to the difference of the interated contraaction maps
\begin{alignat*}{3}
 \Hom_k(V\ot_kU\;\P)&\.\simeq\.\Hom_k(U,\Hom_k(V,\P))&&\lrarrow
 \Hom_k(U,\P)&&\lrarrow\P \\
\intertext{and}
 \Hom_k(V\ot_kU\;\P)&\.\simeq\.\Hom_k(V,\Hom_k(U,\P))&&\lrarrow
 \Hom_k(V,\P)&&\lrarrow\P.
\end{alignat*}
 This definition can be found in~\cite[Section~D.2.7]{Psemi}.
 Contramodules over Tate Lie algebras serve as the coefficients
for the theory of \emph{semi-infinite cohomology} of Lie algebras
(as opposed to the semi-infinite \emph{homology}~\cite{Feig},
\cite[Section~3.8]{BD}); see~\cite[Section~D.5.6]{Psemi} for
the definition and Section~\ref{tate-harish-chandra} below
for a brief overview.

 In order to extend the definition of a $\g$\+contramodule
to topological Lie algebras~$\g$ of more general nature, we need
to introduce a bit more topological linear algebra background.
 The following three \emph{topological tensor product} operations
were defined in~\cite[Section~1.1]{Beil} (see
also~\cite[Section~D.1.3]{Psemi} and~\cite[Section~12]{Pextop}).

 For any topological vector spaces $V$ and $W$, the $!$\+tensor
product $V\ot^!W$ is the completion of the tensor product
$V\ot_kW$ with respect to the topology with a base of neighborhoods
of zero formed by the subspaces $V'\ot W+V\ot W'\subset V\ot_k W$,
where $V'\subset V$ and $W'\subset W$ are open vector subspaces.
 In other words, one has $V\ot^!W=\varprojlim_{V',W'}V/V'\ot_k\nobreak
W/W'$, with the projective limit topology.

 Furthermore, the $*$\+tensor product $V\ot^*W$ is the completion
of the tensor product $V\ot_k W$ with respect to the topology in
which a vector subspace $T\subset V\ot_kW$ is open if and only if
it satisfies the following three conditions:
\begin{enumerate}
\renewcommand{\theenumi}{\roman{enumi}}
\item there exist open vector subspaces $V'\subset V$, \
$W'\subset W$ such that $V'\ot_k W'\subset T$;
\item for any vector $v\in V$ there exists an open subspace
$W''\subset W$ such that $v\ot W''\subset T$;
\item for any vector $w\in W$ there exists an open subspace
$V''\subset V$ such that $V''\ot w\subset T$.
\end{enumerate}
 For any topological vector space $U$, a bilinear map $V\times W
\rarrow U$ is continuous (as a function of two variables) if and only
if it can be (always uniquely) extended to a continuous linear map
$V\ot^*W\rarrow U$.

 Finally, the $\eot$\+tensor product $V\wot W$ is the completion
of the tensor product $V\ot_kW$ with respect to the topology in which
a vector subspace $T\subset V\ot_kW$ is open if and only if it
satisfies the following two conditions:
\begin{enumerate}
\renewcommand{\theenumi}{\roman{enumi}}
\item there exists an open vector subspace $V'\subset V$ such that
$V'\ot_kW\subset T$;
\item for any vector $v\in V$ there exists an open subspace
$W''\subset W$ such that $v\ot W''\subset T$.
\end{enumerate}
 The underlying abstract (nontopological) vector space of
the topological tensor product $V\wot W$ does not depend on
the topology on $W$ and is naturally isomorphic to the completed
tensor product $V\ot\comp W$ introduced in
Section~\ref{over-top-algebras}.
 The multiplication map $\R\times\R\rarrow\R$ of a (complete and
separated) topological associative algebra $\R$ can be (uniquely)
extended to a continuous linear map $\R\wot\R\rarrow\R$ if and
only if open right ideals form a base of neighborhoods of zero in~$\R$
\cite[Section~1.4]{Beil}.

 Each of the three tensor product operations $\ot^!$, $\ot^*$,
and~$\wot$ defines an associative tensor/monoidal structure on
the category of topological vector spaces; the former two tensor
products are also commutative.
 In particular, given a topological associative algebra $\R$ and
a $k$\+vector space $P$, there is a natural isomorphism of
(nontopological) vector spaces $\R\ot\comp(\R\ot\comp P)\simeq
(\R\wot\R)\ot\comp P$; hence the monad multiplication map
$\phi_P\:\R\ot\comp(\R\ot\comp P)\rarrow\R\ot\comp P$ from
Section~\ref{over-top-algebras} defined whenever open right
ideals form a base of neighborhoods of zero in~$\R$
\cite[Section~D.5.2]{Psemi}.

 For any topological vector space $V$, denote by
$\bigwedge^{2,*}(V)$ the completion of the nontopological exterior
square $\bigwedge^2(V)$ with respect to the topology in which
a vector subspace $T\subset\bigwedge^2(V)$ is open if and only if
there exists an open subspace $V'\subset V$ such that
$\bigwedge^2(V')\subset T$ and for any vector $v\in V$ there
exists an open subspace $V''\subset V$ such that
$v\wedge V''\subset T$.
 For any topological vector space $U$, a skew-commutative bilinear
map $V\times V\rarrow U$ is continuous if and only if it can be
(uniquely) extended to a continuous linear map
$\bigwedge^{2,*}(V)\rarrow U$.
 The topological vector space $\bigwedge^{2,*}(V)$ can be also
viewed as a closed vector subspace in $V\ot^*V$.

 A \emph{contramodule} $\P$ over a topological Lie algebra~$\g$
over a field~$k$ is a $k$\+vector space endowed with a contraaction
map $\pi\:\g\ot\comp\P\rarrow\P$ satisfying the following
\emph{contra-Jacobi equation}.
 The composition
$$\textstyle
 \bigwedge^{2,*}(\g)\ot\comp\P\lrarrow\g\ot\comp\P\lrarrow\P
$$
of the map induced by the Lie bracket $\bigwedge^{2,*}(\g)\rarrow\g$
of~$\g$ with the contraaction map should be equal to the composition
of the map induced by the natural maps of topological vector
spaces $\bigwedge^2(V)\rarrow V\ot^*V\rarrow V\wot V$
considered in the case $V=\g$ with the iterated contraaction map
\begin{multline*} \textstyle
 \bigwedge^{2,*}(\g)\ot\comp\P\lrarrow(\g\ot^*\g)\ot\comp\P
 \lrarrow(\g\wot\g)\ot\comp\P \\ \.\simeq\.\g\ot\comp(\g\ot\comp\P)
 \lrarrow\g\ot\comp\P\lrarrow\P. 
\end{multline*}
 This definition can be found in~\cite[Section~D.2.6]{Psemi}.

 A \emph{discrete module} $\M$ over a topological Lie algebra~$\g$
is a $\g$\+module for which the action map $\g\times\M\rarrow\M$
is continuous in the given topology of~$\g$ and the discrete
topology of~$\M$.
 In other words, this means that the annihilator of any element of
$\M$ is an open subalgebra in~$\g$.
 So one can say that discrete $\g$\+modules are a good category to
be assigned to $\g$ when open subalgebras form a base of neighborhoods
of zero in~$\g$ (cf.~\cite[Sections~1.4 and~2.4]{Beil}); otherwise,
a discrete $\g$\+module is the same thing as a discrete module over
the completion of $\g$ in the new topology with a base consisting of
the open subalgebras in the original one.
 For any discrete $\g$\+module $\M$ and any (nontopological)
$k$\+vector space $E$, the vector space $\P=\Hom_k(\M,E)$ has
a natural $\g$\+contramodule structure provided by
the contraaction map $\pi\:\g\ot\comp\Hom_k(\M,E)$ defined
symbolically by the formula
$$
 \pi(x\ot\comp f)(m)=-f(xm),
$$
where $x\in\g$, \,$m\in\M$, \,$f\in\Hom_k(\M,E)$, and the expression
in the right-hand side makes sense due to the definition of
the completed tensor product $\g\ot\comp P$ and the discreteness
condition on the $\g$\+module $\M$ \cite[Section~D.2.6]{Psemi}.

 The category $\g\discr$ of discrete $\g$\+modules is abelian and
the embedding/forgetful functors $\g\discr\rarrow\g\modl\rarrow k\vect$
from it to the categories of arbitrary $\g$\+modules and $k$\+vector
spaces are exact.
 Both infinite direct sums and infinite products exist in $\g\discr$;
the infinite direct sums are also preserved by the forgetful functors.
 It follows that filtered inductive limits are exact in $\g\discr$.
 In other words, the category $\g\discr$ satisfies the axioms Ab5
and Ab3*, but not in general Ab4*.
 It also admits a set of generators, so it has enough injective objects.

 Any $\g$\+contramodule $\P$ has an underlying structure of a module
over the Lie algebra $\g$ viewed as an abstract (nontopological)
Lie algebra; it is provided by the composition of maps
$\g\ot\P\rarrow\g\ot\comp\P\rarrow\P$.
 The category $\g\contra$ is abelian and the forgetful functors
$\g\contra\rarrow\g\modl\rarrow k\vect$ are exact.
 Infinite products exist in the category $\g\contra$ and are preserved
by the forgetful functors.
 The theorem below, when it is applicable, allows one to say more
(cf.\ Sections~\ref{over-virasoro} and~\ref{over-topol-rings}).

 The enveloping algebra $U(\g)$ of a topological Lie algebra $\g$
can be endowed with a natural topology in two opposite ways.
 Let us denote by $U_l\comp(\g)$ the completion of $U(\g)$ in
the topology where the left ideals in $U(\g)$ generated by open
subspaces in~$\g$ form a base of neighborhoods of zero, and by
$U_r\comp(\g)$ the completion of $U(\g)$ in the similar topology
with a base formed by the right ideals generated by open subspaces
in~$\g$.
 Using the assumption of continuity of the bracket in~$\g$, one can
easily check that the multiplication in $U(\g)$ can be extended to
continuous multiplications in $U_l\comp(\g)$ and $U_r\comp(\g)$.
 This construction was considered in~\cite[Section~3.8.17]{BD},
\cite[Section~2.4]{Beil}, and \cite[Section~D.5.1]{Psemi}.

\begin{thm}
\textup{(a)} For any topological Lie algebra\/~$\g$, the category
of discrete\/ $\g$\+modules is naturally isomorphic to the category
of discrete left\/ $U_l\comp(\g)$\+modules, $\g\discr\simeq
U_l\comp(\g)\discr$.
 The datum of a discrete\/ $\g$\+module structure on a vector
space\/ $\M$ is equivalent to the datum of a discrete left\/
$U_l\comp(\g)$\+module structure on\/~$\M$. \par
\textup{(b)} For any topological Lie algebra\/~$\g$ admitting
a countable base of neighborhoods of zero consisting of open Lie
subalgebras in\/~$\g$, the category of\/ $\g$\+contramodules is
naturally isomorphic to the category of left\/
$U_r\comp(\g)$\+contramodules, $\g\contra\simeq
U_r\comp(\g)\contra$.
 The datum of a\/ $\g$\+contramodule structure on a vector space\/
$\P$ is equivalent to the datum of a left\/
$U_r\comp(\g)$\+contramodule structure on\/~$\P$.
\end{thm}

\begin{proof}
 Part~(a): any $\g$\+module can be viewed as an $U(\g)$\+module
and vice versa; it is obvious from the definitions that
a $\g$\+module is discrete if and only if its $U(\g)$\+module
structure extends to a structure of discrete left module over
$U_l\comp(\g)$.
 The proof of part~(b) is more complicated;
see~\cite[Section~D.5.3]{Psemi}.
\end{proof}

\begin{rem}
 General topology and topological algebra are known to be treacherous
ground, and caution is advisable when working with topological vector
spaces with linear topologies, as many assertions which appear to be
natural at first glance turn out to be problematic at a closer look.
 In particular, the exposition in the paper~\cite{Beil}, while
correcting several mistakes or unfortunate definitions found in
the previous book~\cite{BD}, is still too optimistic on a few points.

 For example, the quotient space $V/U$ of a topological vector space
$V$ by a closed vector subspace $U$ is \emph{not} always \emph{complete}
in the quotient topology~\cite[Proposition~11.1]{RD},
\cite[Theorem~4.1.48]{AGM}, \cite[Section~2.11]{Pproperf},
\cite[Theorem~2.5]{Pextop}.
 (Cf.~\cite[Exercice~IV.4.10.b.$\alpha$]{Bour2}, where a related
counterexample in the setting of topological vector spaces with
nonlinear topologies compatible with the topology of the basic
field of real numbers is considered.)
 Even if the quotient space $V/U$ is complete, it does \emph{not} follow
(generally speaking) that the induced map of topological tensor products
$V\ot^!W\rarrow(V/U)\ot^!W$ or complete tensor products
$V\ot\comp P\rarrow(V/U)\ot\comp P$ is
surjective~\cite[Corollary~13.9]{Pextop}.
 In all these cases, the question is how to show that a map between
projective limits of vector spaces is surjective.
 
 The problem does not arise for topological vector spaces with
countable bases of neighborhoods of zero, as countable projective
limits are better behaved, and indeed, any closed subspace $U$
that has a countable base of neighborhoods of zero is a topological
direct summand in~$V$.
 However, the $*$\+tensor product operation leads outside of
the class of topological vector spaces with countable
topologies~\cite[Conclusion~13.11]{Pextop}.

 In particular, we formulate our system of contra-Jacobi equations
as being indexed by the complete tensor product
$\bigwedge^{2,*}(\g)\ot\comp\P$, while the somewhat simpler
alternative of having it indexed by the complete tensor product
$(\g\ot^*\g)\ot\comp\P$ would work just as well when
the characteristic of the field~$k$ is different from~$2$.
 Indeed, the natural map $(\g\ot^*\g)\ot\comp\P\rarrow
\bigwedge^{2,*}(\g)\ot\comp\P$ is surjective in this case,
the topological vector space $\bigwedge^{2,*}(\g)$ being
a direct summand in $\g\ot^*\g$.
 Then the contra-Jacobi equation could be written in the familiar
form of the difference between the two compositions
$(\g\ot^*\g)\ot\comp\P\birarrow\g\ot\comp(\g\ot\comp\P)\rarrow
\g\ot\comp\P\rarrow\P$ being equal to the composition
$(\g\ot^*\g)\ot\comp\P\rarrow\g\ot\comp\P\rarrow\P$.
 The desire to incorporate the characteristic~$2$ case leads to
the somewhat more complicated definition above.
\end{rem}

\subsection{Contramodules over corings}  \label{over-corings}
 The following scheme of categorical buildup is discussed in
the introduction to the book~\cite{Psemi}.
 Let $\sK$ be a category endowed with an (associative,
noncommutative) monoidal (tensor) category structure, $\sM$
be a left module category over it, $\sN$ be a right module
category, and $\sV$ be a category for which there is a pairing
between the module categories $\sN$ and $\sM$ over $\sK$ taking
values in~$\sV$.

 This means that, in addition to the multiplication functor
$\ot\:\sK\times\sK\rarrow\sK$, there are also action functors
$$
 \ot\:\sN\times\sK\rarrow\sN \quad\text{and}\quad
 \ot\:\sK\times\sM\rarrow\sM
$$
and a pairing functor $\ot\:\sN\times\sM\rarrow\sV$.
 Furthermore, there are associativity constraints for
the ternary multiplications
$$
 \sK\times\sK\times\sK\rarrow\sK, \quad
 \sN\times\sK\times\sK\rarrow\sK, \quad
 \sK\times\sK\times\sM\rarrow\sM, \quad
 \sN\times\sK\times\sM\rarrow\sV
$$
satisfying the pentagonal diagram equations for products of four
factors.

 Let $A$ be an associative ring object in~$\sK$.
 Then one can consider the category ${}_A\sK_A$ of $A$\+$A$\+bimodule
objects in $\sK$, the category ${}_A\sM$ of left $A$\+module objects
in $\sM$, and the category $\sN_A$ of right $A$\+module objects
in~$\sN$.
 When the categories $\sK$, $\sM$, $\sN$, and $\sV$ are abelian (or
additive categories with cokernels, or, at least, admit coequalizers),
there are the functors $\ot_A$ of tensor product over $A$, making
${}_A\sK_A$ a tensor category, ${}_A\sM$ a left module category over
it, $\sN_A$ a right module category, and providing a pairing
$\sN_A\times{}_A\sM\rarrow\sV$.
 The new tensor structures $\ot_A$ are associative whenever
the original tensor product functors $\ot$ were right exact
(preserved coequalizers).

 Inverting the arrows in all the four categories, one comes to
considering the situation of a \emph{co}ring object $C\in\sK$.
 Then there is the category ${}_C\sK_C$ of $C$\+$C$\+bicomodule
objects in $\sK$, the category ${}_C\sM$ of left $C$\+comodule
objects in $\sM$, and the category $\sN_C$ of right $C$\+comodule
objects in~$\sN$.
 When the categories $\sK$, $\sM$, $\sN$, and $\sV$ are abelian
(or, at least, admit equalizers), there are the functors $\oc_C$ of
\emph{cotensor product} over $C$, making ${}_C\sK_C$ a tensor
category, ${}_C\sM$ a left module category over it, $\sN_C$
a right module category, and providing a pairing
$\sN_C\times{}_C\sM\rarrow\sV$.
 The new tensor structures $\oc_\C$ are associative whenever
the functors $\ot$ were left exact (preserved equalizers).

 Now one may wish to iterate this construction, considering
a coring object $C$ in the category of $A$\+$A$\+bimodules
${}_A\sK_A$, the categories of $C$\+comodules in
the categories of $A$\+modules ${}_A\sM$ and $\sN_A$,
the category of $C$\+$C$\+bicomodules in ${}_A\sK_A$, etc.
 Then one encounters the typical phenomenon of progressive
relaxation/worsening of algebraic properties at every step of
a buildup.

 The functor $\ot_A$ of tensor product over a ring object $A$ is
in most cases \emph{not} left exact (being defined as a certain
coequalizer, it does not preserve equalizers).
 Hence the cotensor product over a coring object $C\in{}_A\sK_A$
will be only associative under certain (co)flatness conditions
imposed on the objects involved.
 But the associativity is necessary to even \emph{define}
tensor products over ring objects.
 So when one makes the next step and considers a ring object
$S$ in the category of $C$\+$C$\+bicomodules in ${}_A\sK_A$,
one discovers that the functors of tensor product over $S$ are
only partially defined.

\bigskip

 In this section, we consider coring objects $\C$ in the category
of bimodules over a conventional ring~$A$ (i.~e., a ring object
in the tensor category of abelian groups $\sK=\Ab$).
 So let $A$ be an associative ring (with unit).

 A \emph{coring} $\C$ over a ring $A$ is an $A$\+$A$\+bimodule endowed
with a \emph{comultiplication} map $\mu\:\C\rarrow\C\ot_A\C$ and
a \emph{counit} map $\eps\:\C\rarrow A$ satisfying the following
\emph{linearity}, \emph{coassociativity}, and \emph{counitality}
equations.
 First of all, both maps $\mu$ and~$\eps$ must be
\emph{$A$\+$A$\+bimodule morphisms}.
 Secondly, the two compositions of the comultiplication map~$\mu$ with
the two maps $\mu\ot\id$ and $\id\ot\mu\:\allowbreak\C\ot_A\C\birarrow
\C\ot_A\C\ot_A\C$ induced by the comultiplication map should be equal
to each other,
$$
 \C\rarrow\C\ot_A\C\birarrow\C\ot_A\C\ot_A\C.
$$
 Thirdly, both the compositions of the comultiplication map with
the two maps $\eps\ot\id$ and $\id\ot\eps\:\C\ot_A\C\birarrow\C$
induced by the counit map~$\eps$ should be equal to the identity
map~$\id_\C$,
$$
 \C\rarrow\C\ot_A\C\birarrow\C.
$$

 A \emph{left comodule} $\M$ over a coring $\C$ over a ring $A$
is a comodule object in the left module category of left
$A$\+modules over the coring object $\C$ in the tensor category of
$A$\+$A$\+bimodules.
 In other words, it is a left $A$\+module endowed with a \emph{left\/
$\C$\+coaction} map $\nu_\M\:\M\rarrow\C\ot_A\M$ satisfying the following
linearity, \emph{coassociativity}, and \emph{counitality} equations.
 First of all, the map $\nu=\nu_\M$ must be a \emph{left $A$\+module
morphism}.
 Secondly, the compositions of the coaction map~$\nu$ with the two maps
$\mu\ot\id$ and $\id\ot\nu\:\C\ot_A\M\birarrow\C\ot_A\C\ot_A\M$ induced
by the comultiplication and coaction maps should be equal to each other,
$$
 \M\rarrow\C\ot_A\M\birarrow\C\ot_A\C\ot_A\M.
$$
 Thirdly, the composition of the coaction map with the map
$\eps\ot\id\:\C\ot_A\M\rarrow\M$ induced by the counit map
should be equal to the identity map~$\id_\M$,
$$
 \M\rarrow\C\ot_A\M\rarrow\M.
$$
 Similarly, a \emph{right comodule} $\N$ over $\C$ is a right
$A$\+module endowed with a right $\C$\+coac\-tion map $\nu_\N\:\N\rarrow
\N\ot_A\C$, which must be a right $A$\+module morphism satisfying
the coassociativity and counitality equations
\begin{gather*}
 \N\rarrow\N\ot_A\C\birarrow\N\ot_A\C\ot_A\C, \\
 \N\rarrow\N\ot_A\C\rarrow\N.
\end{gather*}
 These definitions can be found in~\cite[Sections~17.1 and~18.1]{BW}
or~\cite[Section~1.1.1]{Psemi}.
 Corings and comodules also appear in noncommutative geometry, or
more specifically, in connection with noncommutative semi-separated
stacks~\cite{KR,KR2}.

 Before introducing \emph{$\C$\+contramodules}, let us discuss a bit
more abstract nonsense.
 The conventional tensor calculus over a ring $A$ includes, in
addition to the tensor product functor~$\ot_A$, the functor $\Hom_A$
of homomorphisms of (say, left) $A$\+modules.
 Applying the functor $\Hom_A$ to an $A$\+$A$\+bimodule $E$ and a left
$A$\+module $P$ produces a left $A$\+module $\Hom_A(E,P)$.
 In fact, the functor $\Hom_A$ endows the category $A\modl^\sop$
\emph{opposite to} the category of left $A$\+modules with
a \emph{right} module category structure over the tensor category
of $A$\+$A$\+bimodules $A\bimod A$.
 Indeed, for any $A$\+$A$\+bimodules $K$ and $L$ and a left
$A$\+module $P$ one has
$$
 \Hom_A(L,\Hom_A(K,P))\simeq\Hom_A(K\ot_A L\;P),
$$
or, denoting temporarily $P^\sop*_AK=\Hom_A(K,P)^\sop$,
$$
 (P^\sop*_AK)*_AL\simeq P^\sop*_A(K\ot_AL)
$$
(cf.\ the discussion of Hom space identification
rules~\eqref{left-tensor-identification}
and~\eqref{right-tensor-identification}
in Section~\ref{over-coalgebras}).
 In other words, one can say that the functor $\Hom_A$ makes
the category of left $A$\+modules a \emph{left Hom category} over
the tensor category of $A$\+$A$\+bimodules.
 The same functor $\Hom_A({-},{-})$ provides a pairing between the left
module category $A\modl$ and the right module category $A\modl^\sop$
over the tensor category $A\bimod A$ taking values in
the opposite category of abelian groups~$\Ab^\sop$.

 A \emph{left contramodule} $\P$ over a coring $\C$ over a ring $A$
is an object of the opposite category to the category of module
objects in the right module category $A\modl^\sop$ over the coring
object $\C$ in the tensor category $A\bimod A$.
 In other words, it is a left $A$\+module endowed with a \emph{left\/
$\C$\+contraaction} map $\pi_\P\:\Hom_A(\C,\P)\rarrow\P$ satisfying
the following linearity, \emph{contraassociativity}, and
\emph{contraunitality} equations.
 First of all, the map $\pi=\pi_\P$ must be a left $A$\+module morphism.
 Secondly, the compositions of the maps $\Hom(\mu,\P)\:
\Hom_A(\C\ot_A\C\;\P)\rarrow\Hom_A(\C,\P)$ and $\Hom(\C,\pi)\:
\allowbreak\Hom_A(\C,\Hom_A(\C,\P))\rarrow\Hom_A(\C,\P)$
induced by the comultiplication and contraaction maps with
the contraaction map should be equal to each other,
$$
 \Hom_A(\C,\Hom_A(\C,\P))\simeq\Hom_A(\C\ot_A\C\;\P)
 \birarrow\Hom_A(\C,\P)\rarrow\P.
$$
 Thirdly, the composition of the map $\P\rarrow\Hom_A(\C,\P)$ induced
by the counit map~$\eps$ with the contraaction map should be equal
to the identity map,
$$
 \P\rarrow\Hom_A(\C,\P)\rarrow\P.
$$
 This definition can be found in~\cite[Section~3.1.1]{Psemi}.
 In a slightly lesser generality of coassociative coalgebras over
commutative rings, it was first given, together with the definition
of a comodule, in the memoir~\cite[Section~III.5]{EM}.

 For any right $\C$\+comodule $\N$ endowed with a left action of
a ring $B$ by right $\C$\+comodule endomorphisms, and any left
$B$\+module $V$, the abelian group $\Hom_B(\N,V)$ has a natural
left $\C$\+contramodule structure.
 Here the left action of $A$ in $\Hom_B(\N,V)$ is induced
by the right action of $A$ in $\N$, and the left
$\C$\+contraaction morphism $\pi\:\Hom_A(\C,\Hom_B(\N,V))\rarrow
\Hom_B(\N,V)$ is obtained by applying the contravariant
functor $\Hom_B({-},V)$ to the right $\C$\+coaction morphism
$\nu\:\N\rarrow\N\ot_A\C$,
$$
 \Hom_A(\C,\Hom_B(\N,V))\simeq\Hom_B(\N\ot_A\C\;V)\rarrow
 \Hom_B(\N,V).
$$

 The left $\C$\+comodule $\C\ot_A V$, where $V$ is a left $A$\+module,
is called the $\C$\+comodule \emph{coinduced} from an $A$\+module~$V$.
 For any left $\C$\+comodule $\L$ there is a natural isomorphism of
abelian groups
$$
 \Hom_\C(\L\;\C\ot_AV)\simeq\Hom_A(\L,V),
$$
where $\Hom_\C(\L,\M)$ denotes the group of morphisms from
a $\C$\+comodule $\L$ to a $\C$\+comodule $\M$ in the category
$\C\comodl$ of left $\C$\+comodules~\cite[Section~1.1.2]{Psemi}.
 
 The left $\C$\+contramodule $\Hom_A(\C,V)$, where $V$ is a left
$A$\+module, is called the $\C$\+contramodule \emph{induced} from
an $A$\+module~$V$.
 For any left $\C$\+contramodule $\Q$ there is a natural isomorphism of
abelian groups
$$
 \Hom^\C(\Hom_A(\C,V),\Q)\simeq\Hom_A(V,\Q),
$$
where $\Hom^\C(\P,\Q)$ denotes the group of morphisms from
a $\C$\+contramodule $\P$ to a $\C$\+contramodule $\Q$ in the category
$\C\contra$ of left $\C$\+contramodules~\cite[Section~3.1.2]{Psemi}.

\begin{prop}
\textup{(a)} The following two conditions on a coring\/ $\C$ are
equivalent:
\begin{itemize}
\item the category of left\ $\C$\+comodules is abelian and
the forgetful functor\/ $\C\comodl\allowbreak\rarrow A\modl$ is exact;
\item the coring\/ $\C$ is a flat right $A$\+module.
\end{itemize}
\textup{(b)} The following two conditions on a coring\/ $\C$ are
equivalent:
\begin{itemize}
\item the category of left\/ $\C$\+contramodules is abelian and
the forgetful functor\/ $\C\contra\rarrow A\modl$ is exact;
\item the coring\/ $\C$ is a projective left $A$\+module.
\end{itemize}
\end{prop}

\begin{proof}
 One defines a $\C$\+comodule or $\C$\+contramodule structure on
the kernel and cokernel of any morphism of left $\C$\+comodules or
left $\C$\+contramodules computed in the category of abelian
groups/left $A$\+modules, assuming respectively that the functor
$\C\ot_A{-}\:A\modl\rarrow A\modl$ is exact (preserves kernels) or
the functor $\Hom_A(\C,{-})\:A\modl\rarrow A\modl$ is exact
(preserves cokernels).
 This allows to show that the second condition implies the first
one in either part~(a) or~(b).

 To prove the converse implication in part~(a), notice that
the functor $\C\ot_A{-}\:\allowbreak A\modl\rarrow A\modl$ is
the composition of the coinduction functor $A\modl\rarrow\C\comodl$
and the forgetful functor $\C\comodl\rarrow A\modl$, the former of
which is right adjoint to the latter one.
 Since any right adjoint functor between abelian categories is
left exact, one concludes that the functor $\C\ot_A{-}$ is left exact
whenever the forgetful functor is exact.
 Similarly, in part~(b) the functor $\Hom_A(\C,{-})\:A\modl\rarrow
A\modl$ is the composition of the induction functor $A\modl\rarrow
\C\contra$ and the forgetful functor $\C\contra\rarrow A\modl$,
the former of which is left adjoint to the latter one.
 Since any left adjoint functor is right exact, the functor
$\Hom_A(\C,{-})$ is right exact whenever the forgetful functor is
exact.
\end{proof}

 Generally speaking, the cokernels of arbitrary morphisms exist
in $\C\comodl$ and are preserved by the forgetful functor
$\C\comodl\rarrow A\modl$, but the kernels in $\C\comodl$ may be
problematic when $\C$ is not a flat right $A$\+module.
 Similarly, the kernels of arbitrary morphisms exist in
$\C\contra$ and are preserved by the forgeftul functor
$\C\contra\rarrow A\modl$, but the cokernels in $\C\contra$ may be
problematic when $\C$ is not a projective left $A$\+module.
 Counterexamples showing that the categories $\C\comodl$ and
$\C\contra$ are \emph{not} abelian in general can be found
in~\cite[Example~B.1.1]{Pcosh}.

 Assume that the coring $\C$ is a flat right $A$\+module; then,
according to Proposition, the category $\C\comodl$ is abelian
and the forgetful functor $\C\comodl\rarrow A\modl$ is exact.
 Both the infinite direct sums and infinite products exist in
$\C\comodl$; the infinite direct sums are exact and are preserved
by the forgetful functor.
 Filtered inductive limits are exact in the category of left
$\C$\+comodules; so it satisfies the axioms Ab5 and Ab3*, but
\emph{not} in general Ab4*.
 The category $\C\comodl$ also has a set of
generators~\cite[Sections~3.13 and~18.14]{BW}; moreover,
when $A$ is a left Noetherian ring or $\C$ is a projective right
$A$\+module, every left $\C$\+comodule is the union of its
subcomodules that are finitely generated as
$A$\+modules~\cite[Sections~18.16 and~19.12]{BW}.

 When $\C$ is a projective right $A$\+module (and also under some
weaker assumptions), the $A$\+$A$\+bimodule $\gA=\Hom_{A^\rop}(\C,A)$
can be endowed with a topological ring structure (with a base of
neighborhoods of zero formed by open left ideals) such that
the category of left $\C$\+comodules is isomorphic to the category
of discrete left $\gA$\+modules, $\C\comodl\simeq\gA\discr$
\cite[Proposition~10.5 and Remark~10.6]{PS1}.
 This provides another point of view on the categorical properties
of the category $\C\comodl$.

 The coaction map $\nu\:\M\rarrow\C\ot_A\M$ embeds every left
$\C$\+comodule $\M$ as a subcomodule into the coinduced 
$\C$\+comodule $\C\ot_A\M$.
 Infinite products of coinduced $\C$\+comodules are computed
by the rule $\prod_\alpha\C\ot_A V_\alpha=\C\ot_A\prod_\alpha V_\alpha$
\cite[Section~18.13]{BW}; to compute the product of an arbitrary
family of left $\C$\+comodules, one can present them as the kernels
of morphisms of coinduced $\C$\+comodules and use the fact
that infinite products always commute with
the kernels~\cite[Section~1.1.2]{Psemi}.
 There are enough injective objects in the category $\C\comodl$;
a left $\C$\+comodule is injective if and only if it is isomorphic
to a direct summand of a $\C$\+comodule $\C\ot_AJ$ coinduced from
an injective left $A$\+module~$J$ (see~\cite[Section~18.19]{BW}
or~\cite[Section~5.1.5]{Psemi}).

 Assume that the coring $\C$ is a projective left $A$\+module; then,
according to Proposition, the category $\C\contra$ is abelian and
the forgetful functor $\C\contra\rarrow A\modl$ is exact.
 Both the infinite direct sums and infinite products exist in
$\C\contra$; the infinite products are exact and are preserved by
the forgetful functor.
 So the category of left $\C$\+contramodules satisfies the axioms
Ab3 and Ab4*, but \emph{not} in general Ab4 or Ab5*.

 The contraaction map $\pi\:\Hom_A(\C,\P)\rarrow\P$ presents every
left $\C$\+contramodule as a quotient contramodule of the induced
$\C$\+contramodule $\Hom_A(\C,\P)$.
 Infinite direct sums of induced $\C$\+contramodules are computed
by the rule $\bigoplus_\alpha\Hom_A(\C,V_\alpha)=
\Hom_A(\C\;\bigoplus_\alpha V_\alpha)$; to compute the direct sum of
an arbitrary family of left $\C$\+con\-tramodules, one can present
them as the cokernels of morphisms of induced $\C$\+contramodules
and use the fact that infinite direct sums always commute with
the cokernels~\cite[Section~3.1.2]{Psemi}.
 There are enough projective objects in $\C\contra$; a left
$\C$\+contramodule is projective if and only if it is a direct summand
of a $\C$\+contramodule $\Hom_A(\C,F)$ induced from a projective
left $A$\+module~$F$ \cite[Section~5.1.5]{Psemi}.

 Furthermore, under the same assumption of $\C$ being a projective
left $A$\+module, the $A$\+$A$\+bimodule $\R=\Hom_A(\C,A)$
can be endowed with a topological ring structure (with a base of
neighborhoods of zero formed by open right ideals) such that
the category of left $\C$\+contramodules is isomorphic to the category
of left $\R$\+contramodules, $\C\contra\simeq\R\contra$
\cite[Propositions~10.4 and~10.5]{PS1}
(see Section~\ref{addm} below for a further discussion).
 This provides another point of view on the categorical properties
of the category $\C\contra$.

\bigskip

 The discussion in the beginning of this section suggests that one
should consider, in addition to the categories of left $\C$\+comodules,
right $\C$\+comodules, and left $\C$\+contramodules, the pairing
functors of \emph{cotensor product} and \emph{cohomomorphisms} acting
from those categories to the category of abelian groups.
 Let us define these functors of two co/contramodule
arguments now.

 The \emph{cotensor product} $\N\oc_\C\M$ of a right $\C$\+comodule $\N$
and a left $\C$\+comodule $\M$ is an abelian group defined as the kernel
of the difference of the pair of maps
$$
 \nu_\N\ot\id, \ \,\id\ot\nu_\M\:\N\ot_A\M\birarrow\N\ot_A\C\ot_A\M
$$
one of which is induced by the $\C$\+coaction in $\N$ and the other
one by the $\C$\+coaction in~$\M$.
 For any right $\C$\+comodule $\N$ and any left $A$\+module $V$
there is a natural isomorphism of abelian groups
$$
 \N\oc_\C(\C\ot_AV)\simeq\N\ot_AV;
$$
the similar formula holds for the cotensor product of a coinduced
right $\C$\+comodule and an arbitrary left $\C$\+comodule.
 In particular, one has $\N\oc_\C\C\simeq\N$ and $\C\oc_\C\M\simeq\M$
(see~\cite[Section~21]{BW} or~\cite[Sections~0.2.1 and~1.2.1]{Psemi}).

 The abelian group of \emph{cohomomorphisms} $\Cohom_\C(\M,\P)$ from
a left $\C$\+comodule $\M$ to a left $\C$\+contramodule $\P$ is
defined as the cokernel of (the difference of) the pair of maps
\begin{multline*}
 \Hom(\nu_\M,\id), \ \,\Hom(\id,\pi_\P)\:
 \Hom_A(\C\ot_A\M\;\P) \\ \simeq\Hom_A(\M,\Hom_A(\C,\P))
 \birarrow\Hom_A(\M,\P).
\end{multline*}
 For any left $\C$\+comodule $\M$, any left $\C$\+contramodule $\P$,
and any left $A$\+module $V$, there are natural isomorphisms of
abelian groups
\begin{gather*}
 \Cohom_\C(\C\ot_A V\;\P)\simeq\Hom_A(V,\P) \\
 \Cohom_\C(\M,\Hom_A(\C,V))\simeq\Hom_A(\M,V);
\end{gather*}
in particular, one has $\Cohom_\C(\C,\P)\simeq\P$
\cite[Sections~0.2.4 and~3.2.1]{Psemi}.

 Notice that the functor of cotensor product $\oc_\C$ over
a coring~$\C$, being defined as the kernel of a morphism of
cokernels, is \emph{neither} left \emph{nor} right exact in general.
 Similarly, the functor $\Cohom_\C$, being defined as the cokernel
of a morphism of kernels, is neither left \emph{nor} right exact
(even when all the categories involved are abelian and all
the forgetful functors are exact).

\subsection{Semicontramodules over semialgebras}
\label{over-semialgebras}
 The notion of a \emph{semialgebra} over a coalgebra over a field
is dual to that of a coring in the same way as the notion of
a coalgebra over a field is dual to that of an (associative)
ring~\cite{Agu,Brz0,Plet,Psemi}.
 In this section we present the related piece of theory, aiming to
define semimodules and semicontramodules over semialgebras and
interpret contramodules over topological groups as semicontramodules
over certain semialgebras, as it was promised in
Section~\ref{over-top-groups}.

 Let $\C$ be a (coassociative) coalgebra (with counit) over a field~$k$.
 In addition to the definitions of left $\C$\+comodules, right
$\C$\+comodules and left $\C$\+contramodules given in
Section~\ref{over-coalgebras} and then repeated, in the greater
generality of a coring~$\C$, in the previous Section~\ref{over-corings},
we will also need the definition of a $\C$\+$\C$\+bicomodule.

 Let $\D$ be another coalgebra over~$k$.
 A \emph{$\C$\+$\D$\+bicomodule} $\K$ is a $k$\+vector space endowed
with a left $\C$\+comodule and a right $\D$\+comodule structures
$\nu'\:\K\rarrow\C\ot_k\K$ and $\nu''\:\K\rarrow\K\ot_k\D$ which
commute with each other in the following sense.
 The composition of the left coaction map $\nu'\:\K\rarrow\C\ot_k\K$
with the map $\id\ot\nu'':\C\ot_k\K\rarrow\C\ot_k\K\ot_k\D$ induced
by the right coaction map~$\nu''$ should be equal to the composition
of the right coaction map $\nu''\:\K\rarrow\K\ot_k\D$ with the map
$\nu'\ot\id\:\K\ot_k\D\rarrow\C\ot_k\K\ot_k\D$ induced by the left
coaction map~$\nu'$.
 Equivalently, the vector space $\K$ should be endowed with
a \emph{$\C$\+$\D$\+bicoaction} map $\nu\:\K\rarrow\C\ot_k\K\ot_k\D$
satisfying the coassociativity and counitality equations
$(\mu_\C\ot\id_\K\ot\mu_\D)\circ\nu=(\id_\C\ot\nu\ot\id_\D)\circ\nu$
and $(\eps_\C\ot\id_\K\ot\eps_\D)\circ\nu=\id_\K$,
\begin{gather*}
 \K\rarrow\C\ot_k\K\ot_k\D\birarrow\C\ot_k\C\ot_k\K\ot_k\D\ot_k\D \\
 \K\rarrow\C\ot_k\K\ot_k\D\rarrow\K
\end{gather*}
(see~\cite[Sections~11.1 or~22.1]{BW} or~\cite[Sections~0.3.1
or~1.2.4]{Psemi}).

 Recall from the end of the previous section that the \emph{cotensor
product} $\N\oc_\C\M$ of a right $\C$\+comodule $\N$ and a left
$\C$\+comodule $\M$ is the $k$\+vector space constructed as the kernel
of (the difference of) the pair of maps 
$$
 \nu_\N\ot\id, \ \,\id\ot\nu_\M\:\N\ot_k\M\birarrow\N\ot_k\C\ot_k\M
$$
induced by the $\C$\+coaction maps in $\N$ and~$\M$.
 Similarly, the $k$\+vector space of \emph{cohomomorphisms} from
a left $\C$\+comodule $\M$ to a left $\C$\+contra\-module $\P$ is
construced as the cokernel of the pair of maps
$$
 \Hom_k(\C\ot_k\M\;\P)\simeq\Hom_k(\M,\Hom_k(\C,\P))
 \birarrow\Hom_k(\M,\P)
$$
one of which is induced by the $\C$\+coaction in $\M$ and the other
one by the $\C$\+contra\-action in~$\P$.
 The functor of cotensor product of comodules over a coalgebra~$\C$
over a field~$k$, being defined as the kernel of a morphism of exact
functors, is left exact; while the functor of cohomomorphisms of
comodules and contramodules over~$\C$, defined as the cokernel
of a morphism of exact functors, is right exact.
 For any left $\C$\+comodule $\M$, right $\C$\+comodule $\N$, and
$k$\+vector space $V$, there is a natural isomorphism of $k$\+vector
spaces~\cite[Sections~0.2.4 and~3.2.2, and Proposition~3.2.3.1]{Psemi}
$$
 \Cohom_\C(\M,\Hom_k(\N,V))\simeq\Hom_k(\N\oc_\C\M\;V),
$$
where the $k$\+vector space $\Hom_k(\N,V)$ is endowed with
a left $\C$\+contramodule structure as explained in
Section~\ref{basic-properties}.

 For any three coalgebras $\C$, $\D$, and $\cE$, any
$\C$\+$\D$\+bicomodule $\N$, and any $\D$\+$\cE$\+bicomodule $\M$,
the cotensor product $\N\oc_\D\M$ has a natural
$\C$\+$\cE$\+bicomodule structure.
 Furthermore, for any right $\C$\+comodule $\N$, any
$\C$\+$\D$\+bicomodule $\K$, and any left $\D$\+comodule $\M$
there is a natural associativity isomorphism
$$
 (\N\oc_\C\K)\oc_\D\M\simeq\N\oc_\C(\K\oc_\D\M).
$$
 To put it simply, both the iterated cotensor products are
identified with one and the same subspace in the vector space
$\N\ot_k\K\ot_k\M$ (cf.\ the beginning of
Section~\ref{over-corings}).

 Similarly, for any $\C$\+$\D$\+bicomodule $\K$ and any left
$\C$\+contramodule $\P$, the space of cohomomorphisms
$\Cohom_\C(\K,\P)$ has a natural left $\D$\+contramodule structure.
 One can define it by noticing that $\Cohom_\C(\K,\P)$ is
a quotient contramodule of the left $\D$\+contramodule $\Hom_k(\K,\P)$,
whose contramodule structure is induced by the right
$\D$\+comodule structure on $\K$ via the construction described
in Section~\ref{basic-properties}. 
 For any $\C$\+$\D$\+bicomodule $\K$, any left $\D$\+comodule $\M$,
and any left $\C$\+contramodule $\P$, there is a natural
associativity isomorphism
$$
 \Cohom_\C(\K\oc_\D\M\;\P)\simeq\Cohom_\D(\M\;\Cohom_\C(\K,\P)).
$$
 Both the (iterated) Cohom spaces are identified with
the quotient space of the vector space $\Hom_k(\K\ot_k\M\;\P)\simeq
\Hom_k(\M,\Hom_k(\K,\P))$ by one and the same vector subspace
\cite[Sections~0.3.4 or~3.2.4]{Psemi}.

 In particular, it follows from these associativity isomorphisms
for a coalgebra $\C=\D$ that the category of $\C$\+$\C$\+bicomodules
$\C\bicomod\C$ is an associative tensor category with respect to
the cotensor product functor~$\oc_\C$, the category of left
$\C$\+comodules $\C\comodl$ is a left module category over
$\C\bicomod\C$, and the category $\C\contra^\sop$ opposite to
the category of left $\C$\+contramodules is a right module category
over $\C\bicomod\C$ with respect to the cohomomorphism functor
$\Cohom_\C$.

 A \emph{semialgebra} $\S$ over a coalgebra $\C$ over a field~$k$
is an associative ring object in the tensor category of
$\C$\+$\C$\+bicomodules.
 In other words, it is a $\C$\+$\C$\+bicomodule endowed with
a \emph{semimultiplication} map $\bm\:\S\oc_\C\S\rarrow\S$ and
a \emph{semiunit} map $\be\:\C\rarrow\S$ satisfying the following
\emph{colinearity}, \emph{semiassociativity} and~\emph{semiunitality}
equations.
 First of all, the maps~$\bm$ and~$\be$ must be $\C$\+$\C$\+bicomodule
morphisms.
 Secondly, the compositions of the two maps $\bm\oc\id_\S$ and
$\id_\S\oc\bm\:\S\oc_\C\S\oc_\C\S\birarrow\S\oc_\C\S$ induced by
the semimultiplication map~$\bm$ with the semimultiplication map
$$
 \S\oc_\C\S\oc_\C\S\birarrow\S\oc_\C\S\rarrow\S
$$
should be equal to each other, $\bm\circ(\bm\oc\id_\S) =
\bm\circ(\id_\S\oc\bm)$.
 Thirdly, both the compositions of the maps $\be\oc\id_\S$ and
$\id_\S\oc\be\:\S\birarrow\S\oc_\C\S$ induced by the semiunit
map~$\be$ with the semimultiplication map~$\bm$
$$
 \S\birarrow\S\oc_\C\S\rarrow\S
$$
should be equal to the identity map, $\bm\circ(\be\oc\id_\S)=\id_\S=
\bm\circ(\id_\S\oc\be)$.

 A \emph{left semimodule} $\bM$ over a semialgebra $\S$ over
a coalgebra $\C$ is a module object in the left module category
of left $\C$\+comodules over the ring object $\S$ in
the tensor category of $\C$\+$\C$\+bicomodules.
 In other words, it is a left $\C$\+comodule endowed with
a \emph{left semiaction} map $\bn\:\S\oc_\C\bM\rarrow\bM$
satisfying the following colinearity, \emph{semiassociativity}
and \emph{semiunitality} equations.
 First of all, the map~$\bn$ must be a left $\C$\+comodule morphism.
 Secondly, the compositions of the two maps $\bm\oc\id_\bM$ and
$\id_\S\oc\bn\:\S\oc_\C\S\oc_\C\bM\birarrow\S\oc_\C\bM$ induced
by the semimultiplication and semiaction maps with
the semiaction map
$$
 \S\oc_\C\S\oc_\C\bM\birarrow\S\oc_\C\bM\rarrow\bM
$$
should be equal to each other, $\bn\circ(\bm\oc\id_\bM)=
\bn\circ(\id_\S\oc\bn)$.
 Thirdly, the composition of the map $\be\oc_\C\id_\bM\:\bM\rarrow
\S\oc_\C\bM$ induced by the semiunit map~$\be$ with the semiaction
map~$\bn$
$$
 \bM\rarrow\S\oc_\C\bM\rarrow\bM
$$
should be equal to the identity map, $\bn\circ(\be\oc_\C\id_\bM)
= \id_\bM$.
 A \emph{right semimodule} $\bN$ over $\S$ is a right $\C$\+comodule
endowed with a right semiaction map $\bn\:\bN\oc_\C\S\rarrow\bN$
satisfying the similar equations
\begin{gather*}
 \bN\oc_\C\S\oc_\C\S\birarrow\bN\oc_\C\S\rarrow\bN, \\
 \bN\rarrow\bN\oc_\C\S\rarrow\bN.
\end{gather*}
 These definitions can be found in~\cite[Sections~2.3 and~6.1]{Agu},
\cite[Section~6]{Brz0}, \cite[Section~8]{Brz},
and~\cite[Sections~0.3.2 and~1.3.1]{Psemi};
see~\cite[Section~0.3.10]{Psemi} for some further references.

 Before defining semicontramodules, let us recall from the discussion
in Section~\ref{over-coalgebras} that there are two ways to define
the conventional modules over associative algebras over~$k$ in
tensor/polylinear algebra terms.
 In addition to the familiar definition of a left $A$\+module $M$
as a $k$\+vector space endowed with a $k$\+linear map
$n\:A\ot_kM\allowbreak\rarrow M$ satisfying the associativity and
unitality equations, one can also say that a left $A$\+module
structure on $M$ is defined by a linear map $p\:M\rarrow\Hom_k(A,M)$
satisfying the correspodingly rewritten equations.

 A \emph{left semicontramodule} $\bP$ over a semialgebra $\S$ over
a coalgebra $\C$ is an object of the opposite category to
the category of module objects in the right module category
$\C\contra^\sop$ over the ring object $\S$ in the tensor category
$\C\bicomod\C$.
 In other words, it is a left $\C$\+contramodule endowed with
a \emph{left semicontraaction} map $\bp\:\bP\rarrow\Cohom_\C(\S,\bP)$
satisfying the following \emph{contralinearity},
\emph{semicontraassociativity}, and \emph{semicontraunitality}
equations.
 First of all, the map~$\bp$ must be a left $\C$\+contramodule
morphism.
 Secondly, the compositions of the semicontraaction map~$\bp$
with the two maps $\Cohom(\bm,\bP)\:\Cohom_\C(\S,\bP)
\rarrow\Cohom_\C(\S\oc_\C\S\;\bP)$ and $\Cohom(\S,\bp)\:
\Cohom_\C(\S,\bP)\rarrow\Cohom_\C(\S,\Cohom_\C(\S,\bP))$
$$
 \bP\rarrow\Cohom_\C(\S,\bP)\birarrow
 \Cohom_\C(\S\oc_\C\S\;\bP)\simeq\Cohom_\C(\S,\Cohom_\C(\S,\bP))
$$
should be equal to each other, $\Cohom(\bm,\bP)\circ\bp=
\Cohom(\S,\bp)\circ\bp$, where the above identification
$\Cohom_\C(\S\oc_\C\S\;\bP)\simeq\Cohom_\C(\S,\Cohom_\C(\S,\bP))$
is presumed.
 Thirdly, the composition of the semicontraaction map with
the map $\Cohom_\C(\be,\bP)\:\Cohom_\C(\S,\bP)\rarrow\bP$
induced by the semiunit map~$\be$
$$
 \bP\rarrow\Cohom_\C(\S,\bP)\rarrow\bP
$$
should be equal to the identity map, $\Cohom_\C(\be,\bP)\circ\bp
= \id_\bP$.
 This definition can be found in~\cite[Sections~0.3.5
or~3.3.1]{Psemi}.

 For any right $\S$\+semimodule $\bN$ and any $k$\+vector space $V$,
the left $\C$\+contramodule $\Hom_k(\bN,V)$ has a natural left
$\S$\+semicontramodule structure.
 The left semicontraaction map $\bp\:\Hom_k(\bN,V)\rarrow
\Cohom_\C(\S,\Hom_k(\bN,V))$ is constructed by applying the functor
$\Hom_k({-},V)$ to the right semiaction map~$\bn$ of
the $\S$\+semimodule~$\bN$
$$
 \Hom_k(\bN,V)\rarrow\Hom_k(\bN\oc_\C\S\;V)\simeq
 \Cohom_\C(\S,\Hom_k(\bN,V)).
$$

 Generally speaking, the kernels of arbitrary morphisms exist in
the category of left $\S$\+semimodules $\S\simodl$ and are preserved
by the forgetful functors $\S\simodl\rarrow\C\comodl\rarrow k\vect$,
but the cokernels in $\S\simodl$ may be problematic when $\C$ is not
an injective right $\C$\+comodule.
 Similarly, the cokernels of arbitrary morphisms exist in
the category of left $\S$\+semicontramodules $\S\sicntr$ and are
preserved by the forgetful functors $\S\sicntr\rarrow\C\contra
\rarrow k\vect$, but the kernels in $\S\sicntr$ may be problematic
when $\C$ is not an injective left $\C$\+comodule.

 Now let us assume that the semialgebra $\S$ is an injective
right $\C$\+comodule.
 Then the cotensor product functor $\S\oc_\C{-}\:\C\comodl\rarrow
\C\comodl$ is exact, so the category $\S\simodl$ of left
$\S$\+semimodules is abelian and the forgetful functors
$\S\simodl\rarrow\C\comodl\rarrow k\vect$ are exact.
 Both the infinite direct sums and infinite products exist in
$\S\simodl$ and both are preserved by the forgetful functor
$\S\simodl\rarrow\C\comodl$, though only the infinite direct sums
are preserved by the full forgetful functor $\S\simodl\rarrow k\vect$.

 Indeed, let $\bM_\alpha$ be a family of left $\S$\+semimodules and
$\prod_\alpha\bM_\alpha$ be their infinite product in the category
of left $\C$\+comodules $\C\comodl$; then one can easily construct
a left semiaction map $\bm\:\S\oc_\C\prod_\alpha\bM_\alpha\rarrow
\prod_\alpha\bM_\alpha$ and show that the left $\S$\+semimodule so
obtained is the product of the family of objects $\bM_\alpha$ in
$\S\simodl$.
 So the category $\S\simodl$ satisfies the axioms Ab5 and Ab3*, but
not in general Ab4*.
 It also has a set of generators, for which one can take
the $\S$\+semimodules $\S\oc_\C\L$ induced from finite-dimensional
left $\C$\+comodules~$\L$.
 Hence there are enough injective objects in $\S\simodl$; we will see
in Section~\ref{underived-semi} below how one can construct them.

 Furthermore, the vector space $\gA=\Hom_{\S^\rop}(\S,\S)=
\Hom_{\C^\rop}(\C,\S)$ can be endowed with a topological $k$\+algebra
structure (with a base of neighborhoods of zero formed by open
left ideals) such that the category of left $\S$\+semimodules is
isomorphic to the category of discrete left $\gA$\+modules,
$\S\simodl\simeq\gA\discr$ \cite[Remark~10.9]{PS1}.
 This provides another point of view on the categorical properties
of the category $\S\simodl$.

 Assume that the semialgebra $\S$ is an injective left $\C$\+comodule.
 Then the cohomomorphism functor $\Cohom_\C(\S,{-})\:\C\contra\rarrow
\C\contra$ is exact, so the category $\S\sicntr$ of left
$\S$\+semicontramodules is abelian and the forgetful functors
$\S\sicntr\rarrow\C\contra\rarrow k\vect$ are exact.
 Both the infinite direct sums and infinite products exist in
the category $\S\sicntr$ and both are preserved by the forgetful functor
$\S\sicntr\rarrow\C\contra$, though only the infinite products are
preserved by the full forgetful functor $\S\sicntr\rarrow k\vect$.

 Indeed, let $\bP_\alpha$ be a family of left $\S$\+semicontramodules
and $\bigoplus_\alpha\bP_\alpha$ be their infinite direct sum in
the category $\C\contra$.
 Then one can easily construct a left semicontraaction map
$\bp:\bigoplus_\alpha\bP_\alpha\rarrow
\Cohom_\C(\S\;\bigoplus_\alpha\bP_\alpha)$
and show that the left $\S$\+semicontramodule so obtained is
the direct sum of the family of objects $\bP_\alpha$ in $\S\sicntr$.
 So the category $\S\sicntr$ satisfies the axioms Ab3 and Ab4*,
but not in general Ab4 or Ab5*.
 There are enough projective objects in $\S\sicntr$; we will see
in Section~\ref{underived-semi} how to construct them.

 Furthermore, the vector space $\R=\Hom_\S(\S,\S)=\Hom_\C(\C,\S)$
can be endowed with a topological $k$\+algebra structure (with a base of neighborhoods of zero formed by open right ideals) such that
the category of left $\S$\+semicontramodules is
isomorphic to the category of left $\R$\+contramodules,
$\S\sicntr\simeq\R\contra$ \cite[Proposition~10.8 and subsequent
discussion]{PS1} (see also Section~\ref{addm} below).
 This provides another point of view on the categorical properties
of the category $\S\contra$, including a description of its
projective objects.

\begin{ex}
 Let us explain the construction of the semialgebra $\S$ for which
the category of $\S$\+semicontramodules is equivalent to
the category of contramodules over a (locally compact totally
disconnected) topological group~$G$, as it was promised in
Section~\ref{over-top-groups}.
 In fact, we will see that for any given group $G$ there is a whole
family of such semialgebras $\S$ depending on the choice of
a compact (i.~e., profinite) open subgroup $H\subset G$.
 All of them are Morita equivalent to each other in the sense
of~\cite[Section~8.4.5]{Psemi}, i.~e., the categories of
(say, left) semimodules over all of them are equivalent, as are
the categories of semicontramodules.

 Given a commutative ring $k$, by a \emph{discrete $G$\+module
over~$k$} we mean a $k$\+module $\M$ endowed with a $k$\+linear
discrete $G$\+module structure $\M\rarrow\M\{G\}$; similarly,
a \emph{$G$\+contramodule over~$k$} is a $k$\+module endowed
with a $k$\+linear $G$\+contramodule structure $\P[[G]]\rarrow\P$.
 In other words, a discrete $G$\+module over~$k$ is a $k$\+linear
object in the additive category $G\discr$ and a $G$\+contramodule
over~$k$ is a $k$\+linear object in the additive category $G\contra$.

 For the beginning, let $k$~be a field.
 We will freely use the terminology and notation of
Section~\ref{over-top-groups}; in particular, $k(X)$ denotes
the vector space of locally constant compactly supported
$k$\+valued functions on a (locally compact totally disconnected)
topological space~$X$.
 Then for any topological spaces $X$ and $Y$ there is a natural
isomorphism $k(X\times Y)\simeq k(X)\ot_k k(Y)$.
 For any profinite group $H$, the inverse image map
$k(H)\rarrow k(H\times H)$ with respect to the multiplication
map $H\times H\rarrow H$, together with the map $k(H)\rarrow k$
of evaluation at the unit element $e\in H$, endow
the vector space $k(H)$ with the structure of a coassociative
coalgebra over~$k$.
 For any $k$\+vector space $A$ one has $A\{H\}=A(H)\simeq k(H)\ot_k A$
and $A[[H]]\simeq\Hom_k(k(H),A)$, so one can easily identify
discrete $H$\+modules over~$k$ with (left or right) $k(H)$\+comodules
and $H$\+contramodules over~$k$ with $k(H)$\+contramodules.

 Let $H$ be a compact open subgroup in a topological group~$G$;
then the left and right actions of $H$ in $G$ endow $k(G)$ with
a natural structure of bicomodule over~$k(H)$.
 Denote by $G\times^HG$ the quotient space of the Cartesian square
$G\times G$ by the equivalence relation $(g'h,g'')\sim(g',hg'')$
for all $g'$, $g''\in G$ and $h\in H$.
 Being a disjoint union of $G/H$ copies of $G$, this quotient
is also a locally compact and totally disconnected topological space.
 The inverse image of functions with respect to the factorization
map $G\times G\rarrow G\times^HG$ identifies the vector space
$k(G\times^HG)$ with the cotensor product $k(G)\oc_{k(H)}k(G)
\subset k(G)\ot_k k(G)=k(G\times G)$.

 For any \'etale map (local homeomorphism) of topological spaces
$p\:X\rarrow Y$ and any abelian group $A$, the push-forward map
$A(X)\rarrow A(Y)$, assigning to a function $f\:X\rarrow A$
the function $p_*(f)\:Y\rarrow A$,
$$
 (p_*f)(y)=\sum_{p(x)=y}f(x), \qquad y\in Y, \ \ x\in X,
$$
is defined~\cite[Section~E.1.1]{Psemi}.
 In particular, the push-forwards with respect to the multiplication
map $G\times^HG\rarrow G$ and the embedding map $H\rarrow G$ provide
natural left and right $H$\+equivariant $k$\+linear maps
$k(G)\oc_{k(H)}k(G)\rarrow k(G)$ and $k(H)\rarrow k(G)$ endowing
the vector space $\S_k(G,H)=k(G)$ with the structure of a semialgebra
over the coalgebra $\C_k(H)=k(H)$ \cite[Section~E.1.2]{Psemi}.

 It is claimed that the category of (left or right)
$\S_k(G,H)$\+semimodules is isomorphic to the category of
discrete $G$\+modules over~$k$; the datum of a $\S_k(G,H)$\+semimodule
structure on a $k$\+vector space $\M$ is equivalent to the datum of
a discrete $G$\+module structure on~$\M$.
 Indeed, according to~\cite[Sections~E.1.3 and~10.2.2]{Psemi},
the datum of a $\S_k(G,H)$\+semimodule structure on $\M$ is equivalent
to that of two structures of a $G$\+module and a $\C_k(H)$\+comodule
satisfying two compatibility equations; essentially, this is the same
as an action of $G$ in $\M$ whose restriction to $H$ comes from
a $\C_k(H)$\+coaction.
 It remains to notice that an action of $G$ whose restriction to $H$
is discrete is the same thing as a discrete action of~$G$.

 Similarly, the category of $\S_k(G,H)$\+semicontramodules is isomorphic
to the category of $G$\+contramodules over~$k$; the datum of
a $\S_k(G,H)$\+semicontramodule structure on a $k$\+vector space $\P$
is equivalent to the datum of a $G$\+contramodule structure.
 Indeed, according to \emph{loc.\ cit.}, the datum of
a $\S_k(G,H)$\+semicontramodule structure on $\M$ is equivalent to
that of two structures of a $G$\+module and a $\C_k(H)$\+contramodule
satisfying two compatibility equations.
 Essentially, it is claimed that a contraassociative $G$\+contraction
map $\P[[G]]\rarrow\P$ can be uniquely recovered from its restriction
to the point measures in $G$ and the measures supported inside $H$,
provided that the two compatibility equations are satisfied.
 One notices that there is an external product map $k[[G]]\ot_k
\P[[G]]\rarrow\P[[G\times G]]$ or $\P[[G]]\ot_k k[[G]]\rarrow
\P[[G\times G]]$ assiging to a $k$\+valued and a $\P$\+valued
measures on $G$ a $\P$\+valued measure on $G\times G$.
 The contraassociativity equation in the definition of
a $G$\+contramodule, restricted to the external products of
$k$\+valued point measures in $G$ and $\P$\+valued measures in~$H$,
taken in any fixed order, provides a prescription for the desired
recovering of the $G$\+contraaction map from its restrictions to
the two specific kinds of measures (cf.~\cite[proof of
Corollary~3.1]{Psm}).

 Notice that the underlying $k$\+vector space $k(G)$ of
the semialgebras $\S_k(G,H)$ does not depend on the choice of
a compact open subgroup $H\subset G$, but the semialgebra structure
does, $\S_k(G,H)$ being a semialgebra over the coalgebra
$\C_k(H)=k(H)$ depending on~$H$.
 Still, the abelian categories of semimodules and semicontramodules
over $\S_k(G,H)$ do not depend on~$H$; but their semiderived
categories and the functors of semi-infinite (co)homology and
the derived semimodule-semicontramodule correspondence, whose
construction is the aim of the book~\cite{Psemi}, \emph{do} depend
on $H$ in a quite essential way~\cite[Section~8.4.6
and Remark~E.3.2]{Psemi}.

 Now one would like to replace a field~$k$ with an arbitrary
commutative ring, including first of all $k=\Z$.
 This was one of the motivating examples for developing the theory
of semimodules and semicontramodules in the generality of
semialgebras over corings over (generally speaking, noncommutative)
rings rather than just over coalgebras over fields in the main
body of the book~\cite{Psemi}.
 One unpleasant technical complication that arises in this
connection is the possible nonassociativity of cotensor product
over a coring (see the discussion in the beginning of
Section~\ref{over-corings}).
 Hence the importance of various sufficient conditions guaranteeing
such associativity; see~\cite[Sections~11.6 and 22.5\+-6]{BW}
and~\cite[Section~1.2.5]{Psemi}.

 In particular, the results of~\cite[11.6(iv)]{BW}
or~\cite[Proposition~1.2.5(f)]{Psemi} ensure that the notions of
a semialgebra $\S_k(G,H)$ and arbitrary semimodules over it are
unproblematic for any commutative ring~$k$.
 To consider semicontramodules, one also needs associativity
of the cohomomorphism functor, which holds in this case by
the result of~\cite[Proposition~3.2.5(j)]{Psemi}.
 All the assertions mentioned above in this example still hold in
this setting.
 The more advanced homological constructions and results
of~\cite{Psemi} in the application to the semialgebras $\S_k(G,H)$
depend on the assumption of the ring~$k$ having finite
homological dimension, though.
\end{ex}

\subsection{The category $\sO^\ctr$}  \label{category-o-contra}
 The conventional concept of representations of an algebraic group
$G$ is that of comodules over the coalgebra of regular
functions $\C(G)$ on~$G$.
 Since every comodule over a coassociative coalgebra is the union
of its finite-dimensional subcomodules~\cite[Section~2.1]{Swe}
(cf.\ the discussion in
Sections~\ref{over-power-series}\+-\ref{over-l-adics}), it means
that infinite-dimensional representations of $G$ are simply
the unions of finite-dimensional representations, or
the \emph{ind-finite-dimensional} representations.

 In particular, while every finite-dimensional representation
of the Lie algebra~$\g$ of a simply connected semisimple complex
algebraic group $G$ can be integrated to a representation of $G$,
this is \emph{no} longer true for infinite-dimensional representations.
 Indeed, for any nonzero Lie algebra~$\g$ one can easily construct
a module that is not a union of its finite-dimensional submodules
(it suffices to consider the enveloping algebra $U(\g)$ with
the action of~$\g$ by left multiplications).

 The Lie correspondence takes a particularly simple form in the case
of unipotent algebraic groups and nilpotent Lie algebras over
a field of characteristic zero, which are two equivalent
categories~\cite[Corollaire~VI.2.4.5]{DeGa}
(see~\cite[Section~D.6.1]{Psemi} for further references and
a discussion including the finite characteristic case).
 A module over a finite-dimensional nilpotent Lie algebra~$\g$
comes from an (always unique) representation of the corresponding
unipotent algebraic group~$G$ if and only if it is a union of
finite-dimensional $\g$\+modules where all the vectors from~$\g$
act by nilpotent linear operators~\cite[Sections~3.3.5\+-7]{VO}.

 It is a classical idea to work with categories intermediate between
those of representations of a Lie (or algebraic) group $G$ and modules
over its Lie algebra~$\g$.
 For this purpose, one starts from a Lie algebra $\g$ with 
a chosen \emph{subgroup} $H$, i.~e., an algebraic group corresponding
to a Lie subalgebra $\h\subset\g$.
 Then one considers $\g$\+modules $\M$ for which the restriction to
$\h$ of the action of $\g$ in $\M$ can be/has been integrated to
an algebraic action of~$H$.
 As to the choice of the subgroup $H$, there are two basic
approaches: given a complex (or real) semisimple Lie group $G$ and
its Lie algebra~$\g$, one can take $H$ to be a maximal compact
subgroup of~$G$; or otherwise one can use a Borel (or maximal
unipotent) subgroup of $G$ in the role of~$H$.

 Modules over a semisimple Lie algebra $\g$ integrable to
representations of a maximal compact Lie subgroup $H$ are classically
known as \emph{Harish-Chandra modules}~\cite{Dix,Wal},
while $\g$\+modules which can be integrated to an algebraic action of
the Borel subgroup form what has been called
the \emph{BGG} (Bernstein--Gelfand--Gelfand) \emph{category\/~$\sO$}
\cite{BGG,Hum}.
 Both can be united under an umbrella notion of \emph{algebraic
Harish-Chandra modules}, which means simply ``modules over a pair
consisting of a Lie algebra and an algebraic subgroup''
(see a terminological discussion in~\cite[Remark~D.2.5]{Psemi}).

 An \emph{algebraic Harish-Chandra
pair}~\cite[Sections~1.8.2 and~3.3.2]{BB} is a set of data consisting
of a Lie algebra $\g$ over a field~$k$, a finite-dimensional Lie
subalgebra $\h\subset H$, an algebraic group $H$ over~$k$ whose Lie
algebra is identified with~$\h$, and an action of $H$ by automorphisms
of the Lie algebra~$\g$.
 The following two compatibility conditions have to be satisfied.
 Firstly, the subalgebra~$\h$ must be an $H$\+invariant
subspace in~$\g$, and the restriction of the action of $H$ in~$\g$
to~$\h$ should coincide with the adjoint action of $H$ in~$\h$.
 Secondly, the action of~$\h$ in~$\g$ obtained by taking the derivative
of the action of $H$ in~$\g$ should coincide with the adjoint action
of~$\h$ in~$\g$.

 A \emph{Harish-Chandra module} $\M$ over an algebraic Harish-Chandra
pair $(\g,H)$ is a $k$\+vector space endowed with a $\g$\+module
structure and an action of $H$ satisfying the following compatibility
conditions.
 Firstly, the restriction of the $\g$\+action in $\M$ to the Lie
subalgebra~$\h$ should coincide with the derivative of the action
of $H$ in~$\M$.
 Secondly, the $\g$\+action map $\g\ot_k\M\rarrow\M$ should be
$H$\+equivariant.

 In algebraic (rather than algebro-geometric) terms, an (affine)
algebraic group $G$ over a field~$k$ is described by
the $k$\+vector space $\C(G)$ of regular functions on $G$,
endowed with a noncommutative convolution comultiplication
$\mu\:\C(G)\rarrow\C(G)\ot_k\C(G)$ induced by the multiplication map
$G\times G\rarrow G$ and a commutative pointwise multiplication
$m\:\C(G)\ot_k\C(G)\rarrow\C(G)$.
 Together with the antipode map $s\:\C(G)\rarrow\C(G)$ induced
by the inverse element map $G\rarrow G$, these structures make
$\C(G)$ a commutative Hopf algebra~\cite{Swe}.
 By a representation of $G$ one conventionally means
a $\C(G)$\+comodule, while the multiplication on $\C(G)$ is being
used in order to define a $\C(G)$\+comodule structure on
the tensor product $\L\ot_k\M$ of any two $k$\+vector spaces $\L$
and $\M$ endowed with $\C(G)$\+comodule structures.
 The antipode map~$s$, being (always) an anti-automorphism for (both
the multiplication and) the comultiplication of $\C=\C(G)$ (or any
other Hopf algebra~$\C$), allows to identify the categories of left
and right $\C$\+comodules, so the difference between them is
inessential here.

 Over a field~$k$ of characteristic~$0$, the enveloping algebra $U(\g)$
of the Lie algebra~$\g$ of an algebraic group $G$ is interpreted as
the algebra of left or right invariant differential operators on $G$
or, simpler yet, the algebra of distributions (``delta functions'')
on $G$ supported at the unit element $e\in G$.
 The $k$\+vector space of distributions here is defined as the discrete
dual vector space to the linearly compact $k$\+vector space of functions
on the formal completion of $G$ at~$e$.
 The noncommutative convolution multiplication $m\:U(\g)\ot_kU(\g)\rarrow
U(\g)$ in the Hopf algebra $U(\g)$ is induced by the multiplication map
$G\times G\rarrow G$, while the commutative comultiplication
$\mu\:U(\g)\rarrow U(\g)\ot_kU(\g)$ is induced by the diagonal embedding
$G\rarrow G\times G$, and the antipode map $s\:U(\g)\rarrow U(\g)$
simply multiplies every vector from~$\g$ by~$-1$.

 Evaluating a $\{e\}$\+supported distribution at a regular function
on $G$ defines a natural pairing $\phi\:\C(G)\ot_k U(\g)\rarrow k$.
 For example, the pairing with an element of~$\g$ assigns to
a regular function on $G$ the value of its derivative along
the corresponding tangent vector at the origin $e\in G$.
 The pairing~$\phi$ is compatible with the Hopf algebra structures
on $\C(G)$ and $U(\g)$, transforming the comultiplication in
the former into the multiplication in the latter and vice versa.
 In our left-right conventions (see Section~\ref{over-l-adics} for
a discussion), this compatibility is expressed by the equations
\begin{align*}
 \phi(f,uv)&=\phi(f_{(2)},u)\phi(f_{(1)},v) \\
 \phi(fg,u)&=\phi(f,u_{(2)})\phi(g,u_{(1)}),
\end{align*}
for any $f$, $g\in\C(G)$ and $u$, $v\in U(\g)$,
where $\mu(f)=f_{(1)}\ot f_{(2)}$ and $\mu(u)=u_{(1)}\ot u_{(2)}$
is Sweedler's symbolic notation for the comultiplication maps.

 Given a $\C(G)$\+comodule $\M$, one defines the ``derivative''
$U(\g)$\+module structure $m\:U(\g)\ot_k\M\rarrow\M$ on $\M$
as given by the composition $U(\g)\ot_k\M\rarrow
U(\g)\ot_k\C(G)\ot_k\M\rarrow\M$ of the maps induced by
the coaction map and the pairing~$\phi$
(with the arguments' positions inverted).
 Alternatively, one can obtain a left action map in the form
$p\:\M\rarrow\Hom_k(U(g),\M)$ as the composition
$\M\rarrow\C\ot_k\M\rarrow\Hom_k(U(\g),\M)$ of the left coaction
map and the map induced by the pairing~$\phi$.
 Furthermore, the adjoint action of $G$ in itself preserves the origin,
so, \emph{unlike} the left and right actions of $\g$ in the enveloping
algebra $U(\g)$, the adjoint action of $\g$ can be integrated to
a representation of $G$ in $U(\g)$ as well as in~$\g$.
 Hence both $\g$ and $U(\g)$ are endowed with natural $\C(G)$\+comodule
structures.

 Now an algebraic Harish-Chandra pair $(\g,H)$ is described in purely
algebraic terms as a set of data consisting of a Lie algebra~$\g$,
a Hopf algebra $\C(H)$, a Lie subalgebra $\h\subset\g$, a pairing
$\phi\:\C(H)\ot_k U(\h)\rarrow k$ compatible with the Hopf algebra
structures, and a coaction of $\C(H)$ in $\g$ satisfying the following
compatibility conditions.
 Firstly, the coaction of $\C(H)$ in $\g$ should be compatible with
the Lie algebra structure on~$\g$; then there is also the induced
coaction of $\C(H)$ in $U(\g)$.
 Secondly, the Lie subalgebra $\h\subset\g$ should be preserved by
the $\C(H)$\+coaction and the pairing $\phi$ should be compatible
with the induced coaction of $\C(H)$ in $U(\h)$ and the adjoint
coaction of $\C(H)$ in itself (equivalently, the restriction
$\psi\:\C(H)\times\h\rarrow k$ of the pairing~$\phi$ should be
compatible with the adjoint $\C(H)$\+coaction in $\C(H)$ and
$\C(H)$\+coaction in~$\g$ restricted to~$\h$).
 Thirdly, the adjoint action of~$\h$ in~$\g$ should coincide with
the derivative $\h$\+action of the $\C(H)$\+coaction, which is
defined in terms of the given pairing~$\phi$.

 Two generalizations of this setting, in two different directions,
are discussed at length in the book~\cite{Psemi}.
 A ``quantum'' version, with two noncommutative, noncocommutative
Hopf algebras $\C$ and $K$ in place of $\C(H)$ and $U(\h)$,
an associative algebra $R$ in the role of $U(\g)$, and ``adjoint''
coactions of $\C$ in $K$ and $R$, is introduced
in~\cite[Section~C.1]{Psemi}.
 A ``Tate'' version, with the Hopf algebra $\C(H)$ of regular functions
on an infinite-dimensional pro-affine pro-algebraic group $H$ and
a linearly compact open subalgebra~$\h$ in a Tate (locally linearly
compact) Lie algebra~$\g$, can be found in~\cite[Section~D.2.1]{Psemi}
(see the overview in Section~\ref{tate-harish-chandra} below).

 To repeat a previously given definition in the slightly new language,
a Harish-Chandra module $\M$ over an algebraic Harish-Chandra pair
$(\g,H)$ is a $k$\+vector space endowed with a $\g$\+module and
a $\C(H)$\+comodule structures satisfying two compatibility conditions.
 The restriction of the $\g$\+action in $\M$ to~$\h$ should coincide
with the derivative $\h$\+action of the $\C(H)$\+coaction, and
the action map $\g\ot_k\M\rarrow\M$ should be a $\C(H)$\+comodule
morphism (where the coaction in the tensor product is defined
in terms of the multiplication in the Hopf algebra~$\C(H)$);
equivalently, the action map $U(\g)\ot_k\M\rarrow\M$ should be
a $\C(H)$\+comodule morphism.

 Before defining \emph{Harish-Chandra contramodules}, let us say
a few words about contramodules over the coalgebra $\C(G)$ of
regular functions on an algebraic group~$G$.
 Unfortunately, there does \emph{not} seem to be any particular way
to \emph{interpret} a $\C(G)$\+con\-tramodule structure on
a $k$\+vector space $\P$ in any terms more explicit than the general
definition of a contramodule over a coalgebra over a field~$k$.
 The only known exception is the case of a reductive algebraic
group~$H$ over a field~$k$ of characteristic~$0$, when
\cite[Lemma~A.2.2]{Psemi} or the last sentence of
Section~\ref{basic-properties} apply and the semisimple abelian
categories of $\C(H)$\+comodules and $\C(H)$\+contramodules
are equivalent.
 Otherwise there is only the general intuition of contramodules as
modules with infinite summation operations, supported by examples
such as that of comodules and contramodules over the coalgebra $\C(H)$
of regular functions on the one-dimensional unipotent algebraic
group $H=\mathbb G_a$ considered in Section~\ref{over-power-series}.

 However, the vector spaces $\P=\Hom_k(\M,V)$ for all
$\C(H)$\+comodules $\M$ (i.~e., vector spaces with the algebraic
group $H$ acting in them in the conventional sense) 
and all $k$\+vector spaces $V$ provide \emph{examples} of
$\C(H)$\+contramodules for \emph{any} algebraic group~$H$
(see Section~\ref{basic-properties}).
 Moreover, for any commutative (for simplicity) Hopf algebra $\C$,
a $\C$\+comodule $\M$, and a $\C$\+contramodule $\P$, the $k$\+vector
space $\Q=\Hom_k(\M,\P)$ has a natural $\C$\+contramodule structure.
 To construct the desired left $\C$\+contraaction map $\pi_\Q\:
\Hom_k(\C,\Hom_k(\M,\P))\rarrow\Hom_k(\M,\P)$, suppose that
we are given a $k$\+linear map $g\:\C\rarrow\Hom_k(\M,\P)$
and a vector $m\in\M$.
 Consider the $k$\+linear map $f\:\C\rarrow\P$ assigning to
any element $c\in\C$ the vector
$$
 f(c)=g\big(s(m_{(-1)})c\big)(m_{(0)})\in\P,
$$
and set
$$
\pi_\Q(g)=\pi_\P(f),
$$
where $m\longmapsto m_{(-1)}\ot m_{(0)}$ is the Sweedler notation for
the left coaction map $\nu_\M\:\M\rarrow\C\ot_k\M$, and
$\pi_\P\:\Hom_k(\C,\P)\rarrow\P$ is the left contraaction map of
the original $\C$\+contramodule~$\P$ \cite[Section~C.4.2]{Psemi}.

 Furthermore, given an algebraic group $G$ and a $\C(G)$\+contramodule
$\P$, one defines the ``contraderivative'' $U(\g)$\+module structure
$U(\g)\ot_k\P\rarrow\P$ on $\P$ as given by the composition
$U(\g)\ot_k\P\rarrow\Hom_k(\C(G),\P)\rarrow\P$ of the map
$$
 u\ot p\longmapsto (c\mapsto \phi(c,u)p), \qquad
 u\in U(\g), \ \,c\in\C(G), \ \,p\in\P
$$
induced by the pairing $\phi\:\C(G)\ot_kU(\g)\rarrow k$ with
the left contraaction map~$\pi$
(cf.~\cite[Sections~10.1.2]{Psemi}).

 A \emph{Harish-Chandra contramodule} $\P$ over an algebraic
Harish-Chandra pair $(\g,H)$ is a $k$\+vector space endowed with
a $\g$\+module and a $\C(H)$\+contramodule structures satisfying
the following two compatibility conditions.
 Firstly, the restriction of the $\g$\+module structure on $\P$
to the Lie subalgebra $\h\subset\g$ should coincide with
the contraderivative $\h$\+module structure of
the $\C(H)$\+contramodule structure on~$\P$.
 Secondly, the $U(\g)$\+action map in the form
$\P\rarrow\Hom_k(U(\g),\P)$ should be a $\C(H)$\+contramodule morphism,
where the $\C(H)$\+contramodule structure on $\Hom_k(U(\g),\P)$ is
obtained from the $\C(H)$\+comodule structure on $U(\g)$ and
the $\C(H)$\+contramodule structure on $\P$ as described above.
 
 The latter condition is equivalent to the $\g$\+action map
$\P\rarrow\Hom_k(\g,\P)$ being a $\C(H)$\+contramodule morphism.
 To convince oneself that this is so, one can present the space
$\Hom_k(U(\g),\P)$ as the projective limit of the Hom spaces
$\Hom_k(F_nU(\g),\P)$, where $F$ denotes the Poincar\'e--Birkhoff--Witt
filtration of $U(\g)$, and further present every space
$\Hom_k(F_nU(\g),\P)$ as a subspace of the Hom space
$\Hom_k(\bigoplus_{i=0}^n\g^{\ot n}\;\P)$.
 Then it remains to use the fact that the $\C(H)$\+comodule structure
on $U(\g)$ is compatible with the associative algebra structure, i.~e.,
the multiplication map $U(\g)\ot_k U(\g)\rarrow U(\g)$ is
a $\C(H)$\+comodule morphism, together with the assumption of
associativity of the $U(\g)$\+action in~$\P$.

 Viewing the case of a semisimple Lie algebra $\g$ with a Borel or
maximal unipotent subgroup $H$ as our main example, we denote by
$\sO(\g,H)$ the category of Harish-Chandra modules over
an algebraic Harish-Chandra pair $(\g,H)$ and by $\sO^\ctr(\g,H)$
the category of Harish-Chandra contramodules.
 These are abelian categories with exact forgetful functors
$\sO(\g,H)\rarrow k\vect$ and $\sO^\ctr(\g,H)\rarrow k\vect$,
the former of which preserves infinite direct sums, while
the latter preserves infinite products.
 The category $\sO(\g,H)$ satisfies the axioms Ab5 and Ab3* (but
not Ab4*), while the category $\sO^\ctr(\g,H)$ satisfies
the axioms Ab3 and Ab4* (but not Ab4 or Ab5*).

 Now it is claimed that there is a semialgebra $\S$ over
the coalgebra $\C=\C(H)$ such that the categories $\sO(\g,H)$
and $\sO^\ctr(\g,H)$ are identified with the categories of
semimodules and semicontramodules over~$\S$.
 More precisely, there are naturally \emph{two} such semialgebras
$\S^l(\g,H)$ and $\S^r(\g,H)$, differing by the left-right symmetry.
 The category of Harish-Chandra modules $\sO(\g,H)$ is isomorphic
to the categories of left semimodules over $\S^l(\g,H)$ and
right semimodules over $\S^r(\g,H)$, while the category of
Harish-Chandra contramodules $\sO^\ctr(\g,H)$ is isomorphic
to the category of left semicontramodules over $\S^r(\g,H)$.

 The semialgebra $\S^l(\g,H)$ is constructed as the tensor product
$U(\g)\ot_{U(\h)}\C(H)$, where the left $U(\h)$\+module structure
on $\C(H)$ is obtained by deriving the left coaction of $\C(H)$
in itself.
 The right $\C$\+comodule structure on $U(\g)\ot_{U(\h)}\C(H)$ is
induced by the right coaction of $\C$ in itself, while the left
$\C$\+comodule structure on this tensor product is defined in
terms of the multiplication in $\C$, as the tensor product of
the left coaction of $\C$ in itself and the adjoint coaction of
$\C$ in $U(\g)$.
 The semiunit map $\be\:\C\rarrow \S^l=U(\g)\ot_{U(\h)}\C$ is induced
by the embedding of algebras $U(\h)\rarrow U(\g)$.
 Finally, the semimultiplication map $\bm\:\S^l\oc_\C\S^l\rarrow\S^l$
is defined
as the composition
\begin{multline*}
 \big(U(\g)\ot_{U(\h)}\C\big)\oc_\C\big(U(\g)\ot_{U(\h)}\C\big)\.\simeq\.
 U(\g)\ot_{U(\h)}\big(\C\oc_\C(U(\g)\ot_{U(\h)}\C)\big)\\
 \.\simeq\. U(\g)\ot_{U(\h)}U(\g)\ot_{U(\h)}\C\lrarrow
 U(\g)\ot_{U(\h)}\C
\end{multline*}
of the mutual associativity isomorphism of the tensor and cotensor
products, whose natural existence in this case can be easily
established from the fact that $U(\g)$ is a projective right
$U(\h)$\+module (by the Poincar\'e--Birkhoff--Witt theorem)
\cite[Section~1.2.3]{Psemi}, and the map induced by
the multiplication map $U(\g)\ot_{U(\h)}U(\g)\rarrow U(\g)$.
 This construction can be found in~\cite[Section~10.2.1]{Psemi}.

 Notice that for any left $\C$\+comodule $\M$ there is a natural
isomorphism
$$
 \S^l\oc_\C\M=\big(U(\g)\ot_{U(\h)}\C\big)\oc_\C\M\simeq
 U(\g)\ot_{U(\h)}\M,
$$
where the $U(\h)$\+module structure on $\M$ is obtained by deriving
the $\C(H)$\+comodule structure.
 Hence it follows that the datum of a left $\S^l$\+semimodule
structure on a $k$\+vector space $\bM$ is equivalent to that of
a left $\C(H)$\+comodule and a left $U(\g)$\+module structures
on $\bM$ satisfying two compatibility
equations~\cite[Section~10.2.2]{Psemi}.
 These are easily seen to express the definition of a structure of
Harish-Chandra module over $(\g,H)$ on the vector space~$\bM$.

 Similarly, the semialgebra $\S^r=\S^r(\g,H)$ is constructed as
the tensor product $\C(H)\ot_{U(\h)}U(\g)$, where the right
$U(\h)$\+module structure on $\C(H)$ is obtained by deriving
the right coaction of $\C(H)$ in itself.
 The left $\C$\+comodule structure on this tensor product is
induced by the left coaction of $\C$ in itself, while the right
$\C$\+comodule structure is obtained by multiplying
the right coaction of $\C$ in itself and the adjoint
coaction of $\C$ in~$U(\g)$.
 The semimultiplication and semiunit maps of the semialgebra $\S^r$
are induced by the multiplication map $U(\g)\ot_{U(\h)}U(\g)\rarrow
U(\g)$ and the embedding $U(\h)\rarrow U(\g)$, as above.

 For any left $\C$\+contramodule $\P$ there is a natural isomorphism
$$
 \Cohom_\C(\S^r,\P)=\Cohom_\C(\C\ot_{U(\h)}U(\g)\;\P)
 \simeq\Hom_{U(\h)}(U(\g),\P),
$$
where the $U(\h)$\+module structure on $\P$ is obtained by
contraderiving the $\C(H)$\+contra\-module structure.
 Hence one concludes that the datum of a left $\S^r$\+semicontramodule
structure on a $k$\+vector spaced $\bP$ is equivalent to that of
a left $\C$\+contramodule and left $U(\g)$\+module structures on $\bP$
satisfying two compatibility equations (cf.\ Example in
Section~\ref{over-semialgebras}).
 These are easily seen to express the definition of a structure of
Harish-Chandra contramodule over $(\g,H)$ on the vector space~$\bP$.

\bigskip

 In addition to these descriptions of left $\S^l$\+semimodules and
left $\S^r$\+semicontra\-modules, one would like to have an explicit
interpretation of what it means to have a left $\S^r$\+semimodule
structure on a $k$\+vector space~$\bM$.
 Such a description of left $\S^r$\+semimodules is indeed obtained
in~\cite[Sections~C.2 and~C.4.3\+-4]{Psemi} under certain mild
assumptions, which we will now discuss.

 So far we used the notions of an ``algebraic group $H$''
and ``a commutative Hopf algebra $\C(H)$'' interchangeably, but
in fact there are several differences between the two
(cf.~\cite{DM}).
 First of all, it is only \emph{affine} algebraic groups that
can be described by the algebras of global functions on them.
 As our aim is to consider linear representations of our algebraic
groups, we can safely assume all of them to be affine.

 Secondly, it is only \emph{finitely generated} commutative
algebras over a field that correspond to algebraic varieties;
the spectrum of an arbitrary commutative algebra is, generally
speaking, a \emph{pro-affine pro-algebraic variety}.
 Considering Harish-Chandra pairs with pro-affine pro-algebraic
groups $H$ presumes also having an infinite-dimensional
linearly compact Lie subalgebra $\h\subset\g$.
 Postponing this discussion to Section~\ref{tate-harish-chandra},
we for now assume all our algebraic groups $H$ to be
finite-dimensional, or the commutative Hopf algebras $\C(H)$
to be finitely generated as algebras over~$k$.

 Thirdly and finally, \emph{algebraic varieties} are generally
assumed to be \emph{reduced schemes}, i.~e., to have no nilpotent
elements in their structure sheaves (or, if they are affine, in
the algebras of global functions on them).
 Now, over a field~$k$ of characteristic~$0$, every algebraic
\emph{group} scheme is a smooth variety, and over any field~$k$
every reduced algebraic group scheme is a smooth variety; but
over a field~$k$ of finite characteristic there exist
\emph{nonreduced} algebraic group schemes.
 It suffices to consider the spectrum of the finite-dimensional
algebra $\C=k[x]/(x^p)$ over a field~$k$ of characteristic~$p$
and notice that the rule $\mu(x)=1\ot x+x\ot 1$ describes
a well-defined coassociative, counital comultiplication making
$\C$ a Hopf algebra over~$k$.

 Now, assuming the algebraic group $H$ to be a smooth
finite-dimensional variety over a field~$k$, there is
the one-dimensional vector bundle of differential forms of
the top degree on~$H$.
 The group $H$ acts in itself by the left and right multiplications,
and there are the two induced actions in the vector space of
global differential top forms $\cE=\Omega^{top}(H)$.
 The subspace of left $H$\+invariant top forms in~$\cE$ is always
one-dimensional, as is the subspace of right $H$\+invariant top
forms, but these two subspaces may not coincide.
 A smooth algebraic group $H$ admitting a nonzero \emph{biinvariant}
top form is said to be \emph{unimodular}.
 An algebraic group $H$ is unimodular if and only if all
the operators of its adjoint representation $\Ad_H\:H\rarrow\GL(\h)$
belong to the subgroup $\SL(\h)\subset\GL(\h)$.
 All the reductive algebraic groups are unimodular, as are all
the nilpotent ones; but many solvable groups are not. 

 Let us first assume the smooth algebraic group $H$ to be unimodular.
 Then the semialgebras $\S^l(\g,H)$ and $\S^r(\g,H)$ are naturally
isomorphic to each other; the isomorphism is provided by the maps
given by the formulas
$$
 c\ot_{U(\h)}u\longmapsto u_{[0]}\ot_{U(\h)}cu_{[1]}
 \quad\text{and}\quad 
 u\ot_{U(\h)}c\longmapsto s(u_{[1]})c\ot_{U(\h)}u_{[0]},
$$
where $c\in\C(H)$, \,$u\in U(\g)$, the notation $c\ot_{U(\h)}u$
and $u\ot_{U(\h)}c$ stands for elements of $\S^r=\C(H)\ot_{U(\h)}U(\g)$
and $\S^l=U(\g)\ot_{U(\h)}\C(H)$, respectively, and
$u\longmapsto u_{[0]}\ot u_{[1]}$ with $u_{[0]}\in U(\g)$ and
$u_{[1]}\in\C(H)$ is the Sweedler notation for right coaction
map defining the adjoint coaction of $\C(H)$ in $U(\g)$
\cite[Section~C.2.6]{Psemi}.
 Accordingly, the categories of left $\S^r$\+semimodules and
left $\S^l$\+semimodules are naturally isomorphic.

 In the general case of a nonunimodular smooth algebraic group $H$,
the categories of left $\S^r$\+semimodules and left
$\S^l$\+semimodules are still naturally equivalent, but
the equivalence $\S^r\simodl\simeq\S^l\simodl$ does \emph{not} commute
with the forgetful functors $\S^r\simodl\rarrow\C\comodl$ and
$\S^l\simodl\rarrow\C\comodl$.
 Instead, the two forgetful functors differ by a twist with
the modular character $\det\Ad_H\:H\rarrow GL(\h)\rarrow k^*$
\cite[Sections~C.2.4\+-5]{Psemi}.

 When the ground field~$k$ has characteristic~$0$
and the Harish-Chandra pair $(\g,H)$ originates from a closed
embedding of algebraic groups $H\subset G$ over~$k$, choosing
a biinvariant top form on~$H$ allows to interpret elements of
the semialgebra $\S^l\simeq\S^r$ as distributions on the smooth
variety $G$, supported in the smooth closed subvariety $H$ and
regular along~$H$ \cite[Remark~C.4.4]{Psemi}.
 In the nonunimodular case, the $k$\+vector space of all
distributions on $G$, supported in $H$ and regular along $H$,
has a natural structure of an \emph{$\S^l$\+$\S^r$\+bisemimodule}
providing the Morita equivalence between the semialgebras
$\S^l$ and~$\S^r$.

\subsection{Tate Harish-Chandra pairs}
\label{tate-harish-chandra}
 It appears that one cannot integrate the Virasoro Lie algebra to
a Lie group, but \emph{a half} of it one easily can.
 Indeed, let $k$~be a field of characteristic~$0$.
 Denote by $H(k)$ the set of all formal Taylor power series
$a(z)=a_1z+a_2z^2+a_3z^3+\dotsb$ with a vanishing coefficient
in degree~$0$ and a nonvanishing coefficient $a_1\ne0$ in degree~$1$.
 Then the composition multiplication $(a*b)(z)=a(b(z))$ defines
a group structure on the set~$H(k)$.
 This group is naturally the group of $k$\+points of a certain
pro-affine pro-algebraic group, which we denote by~$H$.
 The Lie algebra $\h$ of the pro-algebraic group $H$ can be easily
identified with the algebra of vector fields on the formal disk
with a vanishing vector at the origin $zk[[z]]d/dz$, or with
the closed subalgebra in the Virasoro Lie algebra topologically
spanned by the basis vectors $L_0$, $L_1$, $L_2$,~\dots{}
(see Section~\ref{over-virasoro}).

 Let us say a few words about Lie theory in the pro-algebraic
group setting.
 For any subcoalgebra $\D$ in a commutative Hopf algebra $\C$,
the subalgebra generated by $\D+s(\D)$ in a Hopf subalgebra in~$\C$.
 Since $\C$ is the union of its finite-dimensional subcoalgebras
(see Section~\ref{over-power-series}), it is also the union of
its Hopf subalgebras that are finitely generated as commutative
algebras.
 Thus there is no difference between the notions of a pro-affine
pro-algebraic variety with a group structure and a pro-object in
the category of affine algebraic groups.
 The Lie functor on the category of (pro-afffine) pro-algebraic
groups can be simply obtained by passing to the pro-objects on
both sides of the functor assigning a Lie algebra to
an algebraic group; so the Lie algebra of a pro-algebraic group is
a filtered projective limit of finite-dimensional Lie algebras.
 In particular, it follows that the topological Lie algebra
$k[[z]]d/dz$, which has no closed ideals (see Section~\ref{over-Lie}),
\emph{cannot} correspond to any pro-algebraic group, though its
subalgebra $zk[[z]]d/dz$ does, as we have just seen.

 We are interested in considering Harish-Chandra pairs with
a Tate Lie algebra~$\g$ and a pro-algebraic subgroup $H$ corresponding
to an open linearly compact subalgebra $\h\subset\g$.
 A precise definition presents a small technical difficulty in that
one has to explain what it means for the coalgebra $\C=\C(H)$ to act
in a Tate vector space~$\g$.
 Neither the notion of a $\C$\+comodule \emph{nor} that of
a $\C$\+contramodule are suitable for the task; rather, the compact
vector space~$\h$, being dual to a $\C$\+comodule, can be viewed as
a $\C$\+contramodule, while the quotient space $\g/\h$ has
a $\C$\+comodule structure.

 However, an action of a algebraic group in a vector space is,
of course, \emph{not} determined by its restriction to an invariant
subspace and the induced action in the quotient space.
 The authors of the manuscript~\cite{BFM}, where (what we call)
Tate Harish-Chandra pairs seem to have first appeared, solve
the problem by working over the field of complex numbers~$\mathbb C$
and considering an action of the \emph{group of points}
$H(\mathbb C)$ of a pro-algebraic group $H$ in a Tate Lie
algebra~$\g$ \cite[Section~3.1]{BFM}.
 The approach taken in~\cite[Appendix~D]{Psemi}, which works over
an arbitrary field, is to give the definition of a \emph{continuous
coaction} of a discrete coalgebra in a topological vector space.

 Here we restrict ourselves to a brief sketch.
 Let $\C$ be a coassociative coalgebra and $V$ be a topological
vector space over~$k$ (see
Sections~\ref{over-top-algebras}\+-\ref{over-Lie}).
 A \emph{continuous right coaction} of $\C$ in $V$ is a continuous
linear map $V\rarrow V\ot^!\C$, where $\C$ is considered as
a discrete topological vector space, satisfying the coassociativity
and counitality equations.
 Equivalently, a continuous right coaction can be defined as
a continuous linear map $V\ot^*\C^*\rarrow V$, where $\C^*$ is
viewed as a linearly compact vector space, satisfying
the associativity and unitality equations.

 For any topological vector space $V$ with a continuous coaction of
a coassociative coalgebra~$\C$, open subspaces of $V$ invariant
under the continuous coaction form a base of neighborhoods of zero
in~$V$.
 Given a topological vector space $V$ endowed with a continuous
coaction of a coalgebra $\C$ and a topological vector space $W$
endowed with a continuous coaction of a coalgebra $\D$, all
the three topological tensor products $V\ot^!W$, \ $V\ot^*W$, and
$V\wot W$ are naturally endowed with a continuous coaction of
the coalgebra $\C\ot_k\D$.
 In particular, when two topological vector spaces $V$ and $W$ are
endowed with continuous coactions of a Hopf algebra $\C$,
the three topological tensor products acquire the tensor product
coactions of~$\C$ \cite[Sections~D.1.3\+-4]{Psemi}.

 A continuous coaction of a commutative Hopf algebra $\C$ in
a topological Lie algebra~$\g$ is said to be \emph{compatible} with
the Lie algebra structure if the bracket map $[\,,\,]\:\g\ot^*\g
\rarrow\g$ commutes with the continuous coactions of~$\C$.
 Similarly one can speak about compatibility of topological
associative algebra structures, continuous actions, pairings,
etc.\ with continuous coactions of~$\C$ \cite[Section~D.1.5]{Psemi}.
 A Tate Lie algebra~$\g$ with a continuous coaction of
a commutative Hopf algebra $\C$ compatible with the Lie algebra
structure has a base of neighborhoods of zero consisting of
$\C$\+invariant compact open Lie subalgebras $\h\subset\g$
\cite[Section~D.1.8]{Psemi}.

 A \emph{Tate Harish-Chandra pair} $(\g,\C)$ is a set of data consisting
of a Tate Lie algebra~$\g$, a commutative Hopf algebra~$\C$,
a continuous coaction of $\C$ in~$\g$ compatible with the Lie algebra
structure, a $\C$\+invariant compact open subalgebra $\h\subset\g$,
and a continuous pairing $\psi\:\C\times\h\rarrow k$, where $\C$ is
endowed with the discrete topology.
 These data should satisfy four compatibility equations:
the pairing~$\psi$ should be compatible with the comultiplication
in $\C$ and the Lie bracket in~$\g$, and also with the multiplication
in~$\C$; the pairing~$\psi$ should be compatible with the restriction
to~$\h$ of the continuous coaction of $\C$ in~$\g$ and the adjoint
coaction of $\C$ in itself; the action of~$\h$ in~$\g$ obtained by
deriving the continuous coaction of $\C$ in~$\g$ using
the paring~$\psi$ should coincide with the adjoint action of~$\h$
in~$\g$.
 We refer to~\cite[Section~D.2.1]{Psemi} for the details.

 In particular, let $\g$~be a Tate Lie algebra with a pro-nilpotent
compact open subalgebra $\h\subset\g$, i.~e., $\h$~is the projective
limit of a filtered projective system of finite-dimensional
nilpotent Lie algebras.
 Assume further that the discrete $\h$\+module $\g/\h$ is nilpotent
or ``ind-nilpotent'', i.~e., it is the union of finite-dimensional
nilpotent modules over finite-dimensional quotient Lie algebras
of~$\h$ by its open ideals.
 Then, over a field of characteristic~$0$, one can integrate
the pair of Lie algebras~$(\g,\h)$ to a Tate Harish-Chandra pair
with the Hopf algebra $\C=\C(H)$ of the pro-unipotent pro-algebraic 
group $H$ corresponding to the Lie algebra~$\h$.
 A version of this construction is applicable over fields
of arbitrary characteristic~\cite[Section~D.6]{Psemi}.
 In particular, let $\g=\bigoplus_{n<0}\g_n\oplus\prod_{n\ge0}\g_n$
be the Laurent totalization a $\Z$\+graded Lie algebra with
finite-dimensional components (see Section~\ref{over-Lie}); then
for any integer $m\ge1$ the Lie subalgebra $\h=\prod_{n\ge m}\g_m
\subset\g$ satisfies the above nilpotency conditions, so there is
the corresponding Tate Harish-Chandra pair $(\g,\C)$.
 
 A \emph{Harish-Chandra module} $\M$ over a Tate Harish-Chandra pair
$(\g,\C)$ is a $k$\+vector space endowed with a discrete $\g$\+module
and a $\C$\+comodule structures satisfying the following two
compatibility equations.
 Firstly, the derivative $\h$\+action of the $\C$\+coaction in $\M$,
which is always a discrete action due to the continuity/discreteness
condition imposed on the pairing~$\psi$, should coincide with
the restriction of the $\g$\+action in $\M$ to~$\h$.
 Secondly, the action map $\g\ot^*\M\rarrow\M$ should be compatible
with the continuous coactions of~$\C$; equivalently, the action map
$\g/U\ot_k\L\rarrow\M$ should be a morphism of $\C$\+comodules for
any finite-dimensional $\C$\+subcomodule $\L\subset\M$ and any
$\C$\+invariant compact open subspace $U\subset\g$ annihilating~$\L$
\cite[Section~D.2.5]{Psemi}.

 The pairing $\psi\:\C\times\h\rarrow k$ can be viewed as a Lie
algebra morphism $\h\rarrow\C^*$ (where the Lie algebra structure
on $\C^*$ is defined in terms of its associative algebra structure),
and as such, can be uniquely extended to an associative algebra
morphism $U(\h)\rarrow\C^*$, providing a pairing $\phi\:\C\ot_kU(\h)
\rarrow k$ compatible with the Hopf algebra structures on both sides.
 When the pairing~$\phi$ is nondegenerate in~$\C$
\cite[condition~D.2.2\,(iv)]{Psemi}, the derivative action
functor $\C\comodl\rarrow\h\discr$ is fully
faithful~\cite[Section~10.1.4]{Psemi}, which allows to simplify
the definition of a Harish-Chandra module over~$(\g,\C)$.
 Namely, the second one of the above two compatibility equations
holds automatically in this case and can be dropped, so
Harish-Chandra modules over $(\g,\C)$ can be simply defined as
discrete $\g$\+modules whose discrete $\h$\+module structure
comes from a $\C$\+comodule structure (cf.~\cite[Section~D.2.2]{Psemi}).
 In particular, this nondegeneracy condition holds in the above
example of a Tate Harish-Chandra pair associated with a Tate
Lie algebra~$\g$ with a pro-nilpotent compact open subalgebra~$\h$
acting ind-nilpotently in~$\g/\h$ (over any field~$k$).

 A \emph{Harish-Chandra contramodule} $\P$ over a Tate Harish-Chandra
pair $(\g,\C)$ is a $k$\+vector space endowed with a $\g$\+contramodule
and a $\C$\+contramodule structures satisfying the following two
compatibility equations.
 Firstly, the contraderivative $\h$\+contraaction of
the $\C$\+contraaction in $\P$, which is defined as the composition
$$
 \h\ot\comp\P\.\simeq\.\Hom_k(\h\dual,\P)\lrarrow\Hom_k(\C,\P)\lrarrow\P
$$
of the map induced by the pairing~$\psi$ and the $\C$\+contraaction map,
should coincide with the restriction of the $\g$\+contraaction in $\P$
to the subalgebra~$\h$ (see the definition of a contramodule over
a Tate Lie algebra~$\g$ in Section~\ref{over-Lie}).
 Secondly, for any $\C$\+invariant compact open subspace $U\subset\g$
the $\g$\+contraaction map $\Hom_k(U\dual,\P)\rarrow\P$ should be
a morphism of $\C$\+contramodules (where the $\C$\+contramodule
structure on the Hom space from a $\C$\+comodule $U\dual$ into
a $\C$\+contramodule $\P$ is provided by the construction from
Section~\ref{category-o-contra}).
 This definition can be found in~\cite[Sections~D.2.7\+-8]{Psemi}.

 Moreover, under some additional assumptions this definition of
the category of Harish-Chandra contramodules can be also simplified,
similarly to the above discussion of the simpler definition of
the category of Harish-Chandra modules.
 See Theorem~\ref{fully-faithful}.1, Corollary~\ref{fully-faithful},
and Example~\ref{fully-faithful}.2 below.

 For any Harish-Chandra module $\M$ over $(\g,\C)$ and any $k$\+vector
space $E$, the vector space $\P=\Hom_k(\M,E)$ has a natural
Harish-Chandra contramodule structure with the $\C$\+contraaction in
$\P$ provided by the construction from Section~\ref{basic-properties}
and the $\g$\+contraaction in $\P$ defined according to
the construction from Section~\ref{over-Lie}.

 The categories $\sO(\g,\C)$ and $\sO^\ctr(\g,\C)$ of Harish-Chandra
modules and contramodules over a Tate Harish-Chandra pair $(\g,\C)$
are abelian, and the forgetful functors $\sO(\g,\C)\rarrow k\vect$
and $\sO^\ctr(\g,\C)\rarrow k\vect$ are exact.
 Both the infinite direct sums and infinite products exist in
the cagories $\sO(\g,\C)$ and $\sO^\ctr(\g,\C)$, but only the direct
sums are preserved by the forgetful functor $\sO(\g,\C)\rarrow k\vect$
and only the products are preserved by the functor
$\sO^\ctr(\g,\C)\rarrow k\vect$.
 The category $\sO(\g,\C)$ satisfies the axioms Ab5 and Ab3*
(but not Ab4*), while the category $\sO(\g,\C)$ satisfies
the axioms Ab3 and Ab4* (but not Ab4 or Ab5*).

 There are enough injective objects in the category $\sO(\g,\C)$
and enough projective objects in the category $\sO^\ctr(\g,\C)$.
 The identification of these categories with categories of
semimodules and semicontramodules over certain semialgebras, which
we will now briefly discuss, will make the explicit constructions
of such injective and projective objects explained below in
Section~\ref{underived-semi} applicable in this case.

 As in Section~\ref{category-o-contra}, the semialgebras $\S^l(\g,\C)$
and $\S^r(\g,\C)$ are defined as the tensor products
$$
 \S^l=U(\g)\ot_{U(\h)}\C \quad\text{and}\quad \S^r=\C\ot_{U(\h)}U(\g),
$$
where, as above, $U(\g)$ and $U(\h)$ denote the universal enveloping
algebras of the Lie algebras $\g$ and~$\h$ considered as abstract
Lie algebras without any topologies.
 The left and right $U(\h)$\+module structures on $\C$ are obtained
by deriving the left and right coactions of $\C$ in itself using
the pairing~$\phi$.
 The right coaction of $\C$ in $\S^l$ and the left coaction in $\S^r$
are induced by the right and left coactions of $\C$ in itself.
 The construction of the left coaction of $\C$ in $\S^l$ and the right
coaction in $\S^r$ is rather delicate~\cite[Section~D.2.3]{Psemi}.
 Once the $\C$\+$\C$\+bicomodule structures on $\S^l$ and $\S^r$
has been defined, the constructions of the semimultiplication and
semiunit maps are similar to those in
Section~\ref{category-o-contra} (see~\cite[Section~10.2.1]{Psemi}),
though one still has to check that the maps $\S^l\oc_\C\S^l\rarrow
\S^l$ and $\S^r\oc_\C\S^r\rarrow\S^r$ so obtained are morphisms of
$\C$\+$\C$\+bicomodules (the nontrivial part is to show that
the former is a morphism of left $\C$\+comodules and
the latter a morphism of right ones).

 Now the category $\sO(\g,\C)$ of Harish-Chandra modules over
a Tate Harish-Chandra pair $(\g,\C)$ is isomorphic to the category
of left semimodules over the semialgebra $\S^l(\g,\C)$; the datum
of a left $\S^l$\+semimodule structure on a $k$\+vector space $\bM$
is equivalent to that of a Harish-Chandra module structure on~$\bM$.
 Similarly, the category $\sO^\ctr(\g,\C)$ of Harish-Chandra
contramodules over $(\g,\C)$ is isomorphic to the category of
left semicontramodules over the semialgebra $\S^r(\g,\C)$; the datum of
a left $\S^r$\+semicontramodule structure on a $k$\+vector space $\bP$
is equivalent to that a Harish-Chandra contramodule structure on~$\bP$
\cite[Sections~10.2.2, D.2.5, and~D.2.8]{Psemi}.

\bigskip

 As in Section~\ref{category-o-contra}, one would like to have also
an explicit description of what it means to have a left
$\S^r$\+semimodule structure on a $k$\+vector space~$\bM$.
 In the infinite-dimensional situation, one cannot hope to have
a Morita equivalence between the semialgebras $\S^l$ and~$\S^r$.
 Rather, the determinantal anomaly moves one step higher in
the cohomological degree when one passes to Tate vector spaces,
and what used to be a twist with the modular character in
the finite-dimensional case becomes \emph{a shift of the central
charge} in the Tate Harish-Chandra situation.
 Before formulating the precise assertion, let us give the definition
of a \emph{central extension} $\kap\:(\g',\C)\rarrow(\g,\C)$
of a Tate Harish-Chandra pair $(\g,\C)$.

 A \emph{morphism of Tate Harish-Chandra pairs} $(\g',\C')\rarrow
(\g,\C)$ can be defined as a set of data consisting of
a continuous morphism of Tate Lie algebras $\g'\rarrow\g$ and
a morphism of Hopf algebras $\C\rarrow\C'$ such that the map
$\g'\rarrow\g$ commutes with the continuous $\C'$\+coactions and
takes the subalgebra $\h'\subset\g'$ into the subalgebra
$\h\subset\g$, while the maps $\C\rarrow\C'$ and $\h'\rarrow\h$
are compatible with the pairings $\psi$ and~$\psi'$.
 A \emph{central extension of Tate Harish-Chandra pairs with
the kernel~$k$} is a morphism $(\g',\C')\rarrow(\g,\C)$
such that $\C'=\C$ is the same Hopf algebra and $\C\rarrow\C'$
the identity map, $\g'\rarrow\g$ is a quotient map of topological
vector spaces with a one-dimensional kernel $k\subset\g'$ in
which a fixed basis vector $1_{\g'}\in k\subset\g'$ is chosen,
the map of Lie subalgebras $\h'\rarrow\h$ is an isomorphism,
the kernel $k=k\cdot 1_{\g'}$ lies in the center of
the Lie algebra~$\g'$, and the coaction of the Hopf algebra $\C$
in the subspace $k\subset\g'$ is trivial~\cite[Section~D.2.2]{Psemi}.
 As it is usual for extensions with a fixed kernel, the set of
all (isomorphism classes of) central extensions of a given Tate
Harish-Chandra pair $(\g,\C)$ with the kernel~$k$ has a natural
structure of abelian group (and in fact, even of a $k$\+vector
space) with respect to the Baer sum
operation~\cite[Section~D.3.1]{Psemi}.

 The Lie algebra $\gl(V)$ of continuous endomorphisms of a Tate
vector space $V$ has a canonical central extension $\gl(V)\til$
with the kernel~$k$ defined in terms of the trace functional
on the space of all continuous linear operators $V\rarrow V$
of finite rank~\cite[Sections~2.7.8 and~3.8.17\+-18]{BD}
(see also~\cite[Sections~D.1.6\+-8]{Psemi}; a historical discussion
can be found in~\cite[the beginning of Section~2.7]{BD}).
 The central extension $\gamma_0\:\gl(V)\til\rarrow\gl(V)$ can be
characterized by the property that there is a well-defined linear action
of the Lie algebra $\gl(V)\til$ in the space $\bigwedge^{\infty/2}(V)$ of
semi-infinite exterior forms over $V$ lifting the projective action
of the Lie algebra $\gl(V)$ (see~\cite{Feig}, \cite[Lecture~4]{KaRa},
and~\cite[Sections~4.2.13 and~7.13.16]{BD2}).
 One chooses the canonical basis element $1_{\gl\til}\in k
\subset\gl(V)\til$ so that it acts by the identity operator in
the space of semi-infinite forms; abusing the terminology, we say
that $\bigwedge^{\infty/2}(V)$ is a ``$\gl(V)$\+module with
the central charge~$\gamma_0$''.

 Pulling back the central extension $\gl(\g)\til\rarrow\gl(\g)$ with
respect to the adjoint representation $\ad_\g\:\g\rarrow\gl(\g)$, one
obtains the canonical central extension $\kap_0\:\g\til\rarrow\g$ of
an arbitrary Tate Lie algebra~$\g$.
 The above convention for the choice of the canonical basis element
$1_{\gl\til}$ and the consequent choice of the basis element
$1_{\g\til}\in k\subset\g$ provide it that for any discrete module $M$
over the Lie algebra~$\g\til$ where the element~$1_{\g\til}$
acts by minus the identity operator (``a discrete $\g$\+module with
the central charge~$-\kap_0$'') there is a well-defined discrete
action of the original Lie algebra~$\g$ in the tensor product
$\bigwedge^{\infty/2}(V)\ot_kM$, allowing to define
the \emph{semi-infinite homology} of~$\g$ with coefficients in $M$
\cite[Sections~3.8.19\+-22]{BD} as the homology of a natural
differential on $\bigwedge^{\infty/2}(V)\ot_kM$ (see~\cite[the beginning
of Section~3.8]{BD} for a historical discussion).
 In addition, in~\cite[Sections~D.5.5\+-6]{Psemi}, the semi-infinite
\emph{cohomology} of a Tate Lie algebra $\g$ with coefficients in any
\emph{$\g\til$\+contramodule} $\P$ where $1_{\g\til}$~acts
by the identity (``a $\g$\+contramodule with the central
charge~$\kap_0$'') is also defined as the cohomology of a natural
differential on the space $\Hom_k(\bigwedge^{\infty/2}(V)\;\P)$.

 In particular, for the Lie algebra $\g=k((z))d/dz$ of vector fields
on the formal circle, the canonical central extension $\g\til$
is the Virasoro algebra $\Vir$ (see Section~\ref{over-virasoro}).
 The canonical basis element $1\in k\subset\Vir $ is $-C/26$, i.~e.,
the central element $C\in\Vir$ acts by the scalar~$-26$ in the space 
of semi-infinite forms $\bigwedge^{\infty/2}\big(k((z))d/dz\big)$.
 So the semi-infinite homology and cohomology is defined for any
discrete $\Vir$\+module (or, as we will say, ``$k((z))d/dz$\+module'')
with the central charge $C=26$ and any $\Vir$\+contramodule
(``$k((z))d/dz$\+contramodule'') with the central charge~$-26$.

 When a Tate vector space $V$ is endowed with a continuous coaction
of a commutative Hopf algebra $\C$, the topological Lie algebra
$\gl(V)\til$ acquires the induced continuous coaction of~$\C$
\cite[Sections~D.1.6\+-7]{Psemi}.
 Hence a continuous coaction of a commutative Hopf algebra $\C$
in a Tate Lie algebra~$\g$ always lifts naturally to a continuous
coaction of $\C$ in the canonical central extension~$\g\til$ of~$\g$.
 Furthermore, the canonical central extension $\g\til\rarrow\g$ splits
naturally over any compact open Lie subalgebra $\h\subset\g$
\cite[Section~3.8.19]{BD}, \cite[Section~D.1.8]{Psemi} (warning:
over a pair of embedded compact open Lie subalgebras $\h'\subset
\h''\subset\g$, these splittings do \emph{not} agree, but rather
differ by a relative adjoint trace character).
 For any Tate Harish-Chandra pair $(\g,\C)$, this allows to extend
the canonical central extension $\kap_0\:\g\til\rarrow\g$ of
the Tate Lie algebra~$\g$ to a canonical central extension of Tate
Harish-Chandra pairs $(\g\til,\C)\rarrow(\g,\C)$, which we will
denote also by~$\kap_0$.

 Let $\kap\:(\g',\C)\rarrow(\g,\C)$ be a central extension of
Tate Harish-Chandra pairs with the kernel $k=k\cdot 1_{\g'}$.
 By a \emph{Harish-Chandra module} $\M$ \emph{over\/ $(\g,\C)$ with
the central charge~$\kap$} we mean a Harish-Chandra module
over the Tate Harish-Chandra pair $(\g',\C)$ such that the central
element $1_{\g'}\in\g'$ acts by the identity operator in~$\M$.
 Similarly, a \emph{Harish-Chandra contramodule} $\P$ over
a Tate Harish-Chandra pair $(\g,\C)$ \emph{with the central
charge~$\kap$} is a Harish-Chandra contramodule over $(\g',\C)$
such that the central element~$1_{\g'}$ contraacts by
the identity operator in $\P$, that is the composition
$$
 \P\simeq k\ot\comp\P\rarrow\g'\ot\comp\P\rarrow\P
$$
of the map induced by the choice of the element~$1_{\g'}\in\g'$ with
the $\g'$\+contraaction map is equal to the identity map $\P\rarrow\P$.

 The category $\sO_\kap(\g,\C)$ of Harish-Chandra modules over $(\g,\C)$
with the central charge~$\kap$ is abelian; the fully faithful embedding
functor $\sO_\kap(\g,\C)\rarrow\sO(\g',\C)$ is exact and preserves
both the infinite direct sums and products.
 The category $\sO^\ctr_\kap(\g,\C)$ of Harish-Chandra contramodules over
$(\g,\C)$ with the central charge~$\kap$ is also abelian; the fully
faithful embedding functor $\sO_\kap(\g,\C)\rarrow\sO(\g',\C)$ is exact
and preserves the infinite direct sums and products.
 There are enough injective objects in the category $\sO_\kap(\g',\C)$
and enough projective objects in the category $\sO^\ctr_\kap(\g',\C)$;
we will see below in Section~\ref{underived-semi} how one can
construct them.

 Our next goal is to define semialgebras $\S^l=\S^l_\kap(\g,\C)$
and $\S^r=\S^r_\kap(\g,\C)$ such that the categories of Harish-Chandra
modules and contramodules over a Tate Harish-Chandra pair $(\g,\C)$
with the central charge~$\kap$ could be identified with
the categories of semimodules and semicontramodules over $\S^l$
and~$\S^r$.
 Let us start with the related elementary construction of
the modified universal enveloping algebra corresponding to
a central extension of (nontopological) Lie algebras.

 Given a central extension $\kap\:\g'\rarrow\g$ with the kernel
$k=k\cdot 1_{\g'}$, the algebra $U_\kap(\g)$ is constructed
as the quotient algebra $U(\g')/(1_{U(\g')}-1_{\g'})$ of the enveloping
algebra $U(\g')$ by the ideal generated by the difference between
the unit element $1_{U(\g')}$ of the associative algebra $U(\g')$ and
the fixed basis vector $1_{\g'}$ in the kernel of the central extension.
 Then the category of left $U_\kap(\g)$\+modules is isomorphic to
the category of ``$\g$\+modules with the central charge~$\kap$'',
i.~e., $\g'$\+modules where the element $1_{\g'}$ acts by the identity
operator, while the category of right $U_\kap(\g)$\+modules is
isomorphic to the category of $\g$\+modules with the central
charge~$-\kap$.

 Now the semialgebras $\S^l_\kap(\g,\C)$ and $\S^r_\kap(\g,\C)$ are
constructed as the tensor products
$$
 \S^l=U_\kap(\g)\ot_{U(\h)}\C \quad\text{and}\quad
 \S^r=\C\ot_{U(\h)}U_\kap(\g).
$$
 The datum of a left $\S^l$\+semimodule structure on a given $k$\+vector
space $\bM$ is equivalent to that of a Harish-Chandra module structure
over $(\g,\C)$ with the central charge~$\kap$, while the datum of a left
$\S^r$\+semicontramodule structure on a given $k$\+vector space
$\bP$ is equivalent to that of a Harish-Chandra contramodule structure
over $(\g,\C)$ with the central charge~$\kap$
\cite[Sections~D.2.2, D.2.5 and~D.2.8]{Psemi}.
 The following theorem, when its condition is satisfied, allows to
describe left $\S^r$\+semimodules.

\begin{thm}
 Assume that the pairing $\phi\:\C\ot_kU(\h)\rarrow k$ is nondegenerate
in~$\C$.
 Then there is a natural isomorphism of semialgebras
$$
 \S^r_{\kap+\kap_0}(\g,\C)\simeq\S^l_\kap(\g,\C)
$$
over the coalgebra\/~$\C$.
\end{thm}

\begin{proof}
 This is one of the most difficult, and at the same time a singularly
least well-understood result of the book~\cite{Psemi}.
 The precise formulation, where the desired isomorphism is uniquely
characterized by a certain list of conditions, can be found
in~\cite[Section~D.3.1]{Psemi}.
 The lengthy proof, which is based on the relative nonhomogeneous
quadratic duality theory developed in~\cite[Chapter~11]{Psemi}
(see~\cite[Section~0.4]{Psemi} for an introduction, cf.~\cite{Prel}),
occupies the rest of~\cite[Section~D.3]{Psemi}.
\end{proof}

\Section{Tensor Operations and Adjusted Objects}
\label{tensor-and-adjusted}

\subsection{Comodules and contramodules over coalgebras over fields}
\label{tensor-adjusted-over-coalgebra}
 Let $\C$ be a coassociative coalgebra over a field~$k$.
 The constructions of the \emph{cotensor product} $\N\oc_\C\M$ of a right
$\C$\+comodule $\N$ and a left $\C$\+comodule $\M$ and the vector space
of \emph{cohomomorphisms} $\Cohom_\C(\M,\P)$ from a left $\C$\+comodule
$\M$ to a left $\C$\+contramodule $\P$ were already presented at the end
of Section~\ref{over-corings} and repeated in the beginning of
Section~\ref{over-semialgebras}.
 The main function of these two tensor operations on comodules and
contramodules in the theories developed in the book~\cite{Psemi} is
to provide the tensor and module category structures in whose terms
the notion of a semialgebra and the categories of semimodules
and semicontramodules are subsequently defined.

 Let us now introduce the definition of the \emph{contratensor
product} of a right $\C$\+comodule $\N$ and a left $\C$\+contramodule
$\P$, which plays a key role in the comodule-contramodule
correspondence constructions.
 The contratensor product $\N\ocn_\C\P$ is a $k$\+vector space defined
as the cokernel of (the difference of) the pair of maps
$$
 (\id\ot\ev_\C)\circ(\nu_N\ot\id), \ \,\id\ot\pi_\P\.\:
 \N\ot_k\Hom_k(\C,\P)\birarrow\N\ot_k\P,
$$
one of which is the composition $\N\ot_k\Hom_k(\C,\P)\rarrow
\N\ot_k\C\ot_k\Hom_k(\C,\P)\rarrow\N\ot_k\P$ of the map induced
by the $\C$\+coaction in $\N$ and the map induced by the evaluation
map $\ev_\C\:\C\ot_k\Hom_k(\C,\P)\rarrow\P$, while the other one is
induced by the $\C$\+contraaction in~$\P$.
 The functor of contratensor product of comodules and contramodules
over a coalgebra $\C$ is right exact.
 For any right $\C$\+comodule $\N$ and $k$\+vector space $V$
there is a natural isomorphism of $k$\+vector spaces
$$
 \N\ocn_\C\Hom_k(\C,V)\simeq\N\ot_kV.
$$
 Furthermore, for any right $\C$\+comodule $\N$, left
$\C$\+contramodule $\P$, and $k$\+vector space $V$ there is a natural
isomorphism of $k$\+vector spaces
$$
 \Hom_k(\N\ocn_\C\P\;V)\simeq\Hom^\C(\P,\Hom_k(\N,V)),
$$
where the vector space $\Hom_k(\N,V)$ is endowed with a left
$\C$\+contramodule structure as explained in
Section~\ref{basic-properties}
\,\cite[Sections~0.2.6 and~5.1.1]{Psemi}.

 For any two coalgebras $\C$ and $\D$ over~$k$, any
$\C$\+$\D$\+bicomodule $\K$, and any left $\C$\+comodule $\M$,
the vector space of $\C$\+comodule homomorphisms
$\Hom_\C(\K,\M)$ has a natural left $\D$\+contramodule structure.
 One can define it by noticing that $\Hom_\C(\K,\M)$ is
a subcontramodule in the left $\D$\+contramodule $\Hom_k(\K,\M)$,
whose contramodule structure is induced by the right $\D$\+comodule
structure on~$\K$.
 Similarly, for any $\C$\+$\D$\+bicomodule $\K$ and left
$\D$\+contramodule $\P$ the vector space $\K\ocn_\D\P$ has
a natural left $\C$\+comodule structure.

 For any $\C$\+$\D$\+bicomodule $\K$, left $\C$\+comodule $\M$,
and left $\D$\+contramodule $\P$ there is a natural isomorphism
of $k$\+vector spaces
$$
 \Hom_\C(\K\ocn_\D\P\;\M)\simeq\Hom^\D(\P,\Hom_\C(\K,\M)).
$$
 In other words, the functor $\K\ocn_\D{-}\:\D\contra\rarrow
\C\comodl$ is left adjoint to the functor $\Hom_\C(\K,{-})\:
\C\comodl\rarrow\D\contra$ \cite[Section~5.1.2]{Psemi}.
 The adjoint functors $\Hom_\C(\C,{-})$ and $\C\ocn_\C{-}$ of
comodule homomorphisms from and contratensor product with
the $\C$\+$\C$\+bicomodule $\K=\C$ restricted to the additive
subcategories of injective $\C$\+comodules and projective
$\C$\+contramodules provide the equivalence of additive categories
$\C\comodl_\inj\simeq\C\contra_\proj$ described at the end of
Section~\ref{basic-properties}.

\bigskip

 Relations between (or rather, in the case of a coalgebra $\C$
over a field~$k$, the coincidences of) the following classes of
adjusted objects, together with the classes of injective and
projective objects, in the comodule and contramodule categories play
an important role in the co/contramodule and semico/semicontramodule
theory.
 A discussion of these coincidences is the aim of the remaining part
of this section.
 We will see in the proof of Proposition~\ref{underived-semi} below
how these results are being used.

 A left $\C$\+comodule $\M$ is called \emph{coflat} if the functor
${-}\oc_\C\M\:\comodr \C\rarrow k\vect$ of cotensor product with $\M$
is exact on the abelian category of right $\C$\+comodules.
 A left $\C$\+comodule $\M$ is called \emph{coprojective} if
the functor $\Cohom_\C(\M,{-})\:\C\contra\rarrow k\vect$ is
exact on the abelian category of left $\C$\+contramodules.

 Similarly, a left $\C$\+contramodule $\P$ is called \emph{contraflat}
if the functor ${-}\ocn_\C\nobreak\P\:\allowbreak\comodr \C\rarrow
k\vect$ of contratensor product with $\P$ is exact on the abelian
category of right $\C$\+comodules.
 A left $\C$\+contramodule $\P$ is called \emph{coinjective} if
the functor $\Cohom_\C({-},\P)\:\C\comodl^\sop\rarrow k\vect$ is exact
on the abelian category of left
$\C$\+comodules~\cite[Section~0.2.9]{Psemi}.

\begin{lem}
 Let\/ $\C$ be a coassociative coalgebra over a field~$k$.  Then \par
 \textup{(a)} a\/ $\C$\+comodule is coflat if and only if it is
coprojective and if and only if it is injective; \par
 \textup{(b)} a\/ $\C$\+contramodule is contraflat if and only if it is
coinjective and if and only if it is projective.
\end{lem}

\begin{proof}
 Part~(a): it is clear from the natural isomorphism
$\Hom_k(\N\oc_\C\M\;V)\simeq\Cohom_\C(\M,\Hom_k(\N,V))$ for
any right $\C$\+comodule $\N$, left $\C$\+comodule $\M$, and
$k$\+vector space $V$ (see Section~\ref{over-semialgebras}) that
any coprojective left $\C$\+comodule $\M$ is coflat, and from
the natural isomorphism $\Cohom_\C(\C\ot_kV\;\P)\simeq\Hom_k(V,\P)$
for any $k$\+vector space $V$ and left $\C$\+contramodule $\P$
(see Section~\ref{over-corings}) that any injective left
$\C$\+comodule $\M$ is coprojective.

 Conversely, by a comodule version of Baer's criterion,
a left $\C$\+comodule $\M$ is injective whenever the functor
$\Hom_\C({-},\M)$ is exact on the category of \emph{finite-dimensional}
left $\C$\+comodules.
 Indeed, a $\C$\+comodule morphism into $\M$ from a subcomodule of
any $\C$\+comodule can be successively extended to larger and larger
subcomodules in the way of a transfinite induction or Zorn's lemma,
and one only has to deal with subcomodules of finite-dimensional
$\C$\+comodules in the process.
 It remains to notice the natural right $\C$\+comodule structure
on the dual vector space $\L^*$ to any finite-dimensional left
$\C$\+comodule $\L$, and the natural isomorphism
$$
 \Hom_\C(\L,\M)\simeq\L^*\oc_\C\M
$$
for any finite-dimensional left $\C$\+comodule $\L$ and any left
$\C$\+comodule $\M$, in order to conclude that any coflat left
$\C$\+comodule $\M$ is injective.

 Part~(b): since any $\C$\+comodule is a union of its
finite-dimensional subcomodules, and the functor of contratensor
product ${-}\ocn_\C\P$ preserves inductive limits (in its comodule
argument), a left $\C$\+contramodule $\P$ is contraflat whenever
the functor ${-}\ocn_\C\P$ is exact on the category of
finite-dimensional right $\C$\+comodules.
 It remains to notice the natural isomorphism
$$
 \Cohom_\C(\L,\P)\simeq\L^*\ocn_\C\P
$$
for any finite-dimensional left $\C$\+comodule $\L$ and any left
$\C$\+contramodule $\P$ in order to conclude that any coinjective
left $\C$\+contramodule is contraflat.
 The assertion that every projective $\C$\+contramodule is
coinjective follows immediately from the natural isomorphism
$\Cohom_\C(\M,\Hom_k(\C,V))\simeq\Hom_k(\M,V)$ for any left
$\C$\+comodule $\M$ and $k$\+vector space $V$ (see
Section~\ref{over-corings}).

 Showing that every contraflat $\C$\+contramodule is projective
is much more difficult.
 This assertion was formulated as a conjecture in~\cite{Plet}
and eventually proven in~\cite[Section~A.3]{Psemi}.
 The result in question is a generalization of the classical theorem
that flat modules over a finite-dimensional associative algebra are
projective~\cite[Theorem~P and Examples~I.3(1)]{Bas}.
 The argument is~\cite{Bas} is based on the structure theory of
Artinian rings.
 Similarly, the proof of projectivity of contraflat contramodules
in~\cite{Psemi} is based on the structure theory of coalgebras over
fields~\cite[Sections~9.0\+-1]{Swe} and the structure theory of
contramodules over them developed in~\cite[Section~A.2]{Psemi}, and
first of all, on the contramodule Nakayama lemma
(see Section~\ref{over-topol-rings}).

 In the exposition below, we restrict ourselves to proving that
any coinjective $\C$\+contramodule is projective.
 This is the result that has an important application to
semicontramodules that will be considered in
Section~\ref{underived-semi}.
 The argument that we describe here also has an advantage of being
generalizable to comodules and contramodules over corings over
arbitrary noncommutative rings~\cite[Lemma~5.2]{Psemi}.

 The assertion comes out as an unexpected consequence of one of
the results establishing mutual associativity of the cotensor
product/cohomomorphism operations $\oc_\C$ or $\Cohom_\C$ with
the operations of contratensor product, comodule homomorphisms,
or contramodule homomorphisms $\ocn_\C$, \,$\Hom_\C$, or $\Hom^\C$.
 The operations from the first and the second list are exact on
different sides, so they are only mutually associative under certain
adjustedness conditions on the objects
involved~\cite[Section~5.2]{Psemi} (cf.~\cite{Plet}, where these
mutual associativity assertions were formulated in a less general
form insufficient for deducing the corollary under discussion).

 Let $\C$ and $\D$ be two coalgebras over a field~$k$.

\begin{prop1}
 Let\/ $\N$ be a right\/ $\D$\+comodule, $\K$ be
a\/ $\D$\+$\C$\+bicomodule, and\/ $\P$ be a left\/ $\C$\+contramodule.
 Then there is a natural map of $k$\+vector spaces
$$
 (\N\oc_\D\K)\ocn_\C\P\lrarrow\N\oc_\D(\K\ocn_\C\P),
$$
which is an isomorphism whenever\/ $\P$ is a contraflat left\/
$\C$\+contramodule or\/ $\N$ is a coflat right\/ $\D$\+comodule. 
\end{prop1}

\begin{prop2}
 Let\/ $\L$ be a left\/ $\D$\+comodule, $\K$ be
a\/ $\C$\+$\D$\+bicomodule, and\/ $\M$ be a left\/ $\C$\+comodule.
 Then there is a natural map of $k$\+vector spaces
$$
 \Cohom_\D(\L,\Hom_\C(\K,\M))\lrarrow
 \Hom_\C(\K\oc_\D\L\;\M),
$$
which is an isomorphism whenever\/ $\M$ is an injective left\/
$\C$\+comodule or\/ $\L$ is a coprojective left\/ $\D$\+comodule.
\end{prop2}

\begin{prop3}
 Let\/ $\P$ be a left\/ $\C$\+contramodule, $\K$ be
a\/ $\D$\+$\C$\+bicomodule, and\/ $\Q$ be a left\/ $\D$\+contramodule.
 Then there is a natural map of $k$\+vector spaces
$$
 \Cohom_\D(\K\ocn_\C\P\;\Q)\lrarrow
 \Hom^\C(\P,\Cohom_\D(\K,\Q)),
$$
which is an isomorphism whenever\/ $\P$ is a projective left\/
$\C$\+contramodule or\/ $\Q$ is a coinjective left\/
$\D$\+contramodule.
\end{prop3}

\begin{proof}
 To construct the natural map in Proposition~1, one considers
the composition
$$
 (\N\oc_\D\K)\ot_k\P\lrarrow\N\ot_k\K\ot_k\P\lrarrow
 \N\ot_k(\K\ocn_\C\P)
$$
and observes that it is has equal compositions with the two maps
$(\N\oc_\D\K)\ot_k\Hom_k(\C,\P)\birarrow(\N\oc_\D\K)\ot_k\P$,
as well as with the two maps
$\N\ot_k(\K\ocn_\C\P)\birarrow\N\ot_k\D\ot_k(\K\ocn_\C\P)$.
 This map is an isomorphism when the $\C$\+contramodule $\P$
is contraflat, since the exact sequence
$$
 0\lrarrow\N\oc_\D\K\lrarrow\N\ot_k\K\lrarrow\N\ot_k\D\ot_k\K
$$
remains exact after applying the functor ${-}\ocn_\C\P$,
and when the $\D$\+comodule $\N$ is coflat, since the exact
sequence
$$
 \K\ot_k\Hom_k(\C,\P)\lrarrow\K\ot_k\P\lrarrow\K\ocn_\C\P\lrarrow0
$$
remains exact after applying the functor $\N\oc_\D{-}$.

 To construct the natural map in Proposition~3, one considers
the composition
\begin{multline*}
 \Hom_k(\K\ocn_\C\P\;\Q)\lrarrow\Hom_k(\K\ot_k\P\;\Q) \\ \.\simeq\.
 \Hom_k(\P,\Hom_k(\K,\Q))\lrarrow\Hom_k(\P,\Cohom_\D(\K,\Q))
\end{multline*}
and observes that it has equal compositions with the two maps
$\Hom_k(\D\ot_k(\K\ocn_\C\P)\;\allowbreak\Q)\birarrow
\Hom_k(\K\ocn_\C\P\;\Q)$,
as well as with the two maps $\Hom_k(\P,\Cohom_\D(\K,\Q))\allowbreak
\birarrow\Hom_k(\Hom_k(\C,\P),\Cohom_\D(\K,\Q))$.
 This map is an isomorphism when the $\C$\+contramodule $\P$ is
projective, since the exact sequence
$$
 \Hom_k(\D\ot_k\K\;\Q)\lrarrow\Hom_k(\K,\Q)\lrarrow\Cohom_\D(\K,\Q)
 \lrarrow0
$$
remains exact after applying the functor $\Hom^\C(\P,{-})$, and
when the $\D$\+contramodule $\Q$ is coinjective, since the exact
sequence
$$
 \K\ot_k\Hom_k(\C,\P)\lrarrow\K\ot_k\P\lrarrow\K\ocn_\C\P\lrarrow0
$$
remains exact after applying the functor $\Cohom_\D({-},\Q)$.
 The proof of Proposition~2 is similar.
\end{proof}

 Now we can prove that any coinjective left $\C$\+contramodule $\P$
is projective.
 Let $l\:\E\rarrow\P$ be a surjective $\C$\+contramodule morphism.
 According to the second assertion of Proposition~3 applied to
the coalgebras $\C=\D$, the bicomodule $\K=\C$, and the contramodules
$\Q=\P$, there is a commutative diagram of maps of $k$\+vector spaces
with a lower horizontal isomorphism
$$\dgARROWLENGTH=4em
\begin{diagram} 
\node{\Cohom_\C(\C\ocn_\C\P\;\E)}\arrow[2]{e}
\arrow{s,l,A}{\Cohom_\C(\C\ocn_\C\P\;l)}
\node[2]{\Hom^\C(\P,\E)}\arrow{s,r}{\Hom^\C(\P\;l)} \\
\node{\Cohom_\C(\C\ocn_\C\P\;\P)} \arrow[2]{e,=}
\node[2]{\Hom^\C(\P,\P)} 
\end{diagram}
$$
 The leftmost vertical map, being a quotient of the surjective map
$$
 \Hom_k(\C\ocn_\C\P\;l)\:\Hom_k(\C\ocn_\C\P\;\E)
 \lrarrow\Hom_k(\C\ocn_\C\P\;\P),
$$
is surjective; so the rightmost vertical map is surjective, too.
 It follows that the morphism $\id\:\P\rarrow\P$ can be lifted to
a $\C$\+contramodule morphism $\P\rarrow\E$, i.~e., our surjective
morphism of $\C$\+contramodules $l\:\E\rarrow\P$ splits.
\end{proof}

\subsection{Contramodules over pro-Artinian local rings}
\label{pro-artinian}
 A \emph{pro-Artinian commutative ring} $\R$ is the projective limit of
a filtered projective system of Artinian commutative rings and
surjective morphisms between them.
 Equivalently, $\R$ is a complete and separated topological
commutative ring where open ideals form a base of neighborhoods of
zero and all the discrete quotient rings are
Artinian~\cite[Section~A.2]{Pweak}.

 A \emph{pro-Artinian local ring} is the projective limit of a filtered
projective system of Artinian commutative local rings and surjective
morphisms between them.
 E.~g., any complete Noetherian local ring can be viewed as
a pro-Artinian local ring.

 Contramodules over pro-Artinian local rings are suggested
in~\cite{Pweak} for use in the role of coefficients in homological
theories involving tensor operations, infinite direct sums and products,
and reductions to the residue field.
 The point is that the use of contramodules with their well-behaved
reductions (satisfying Nakayama's lemma irrespectively of any finite
generatedness assumptions) allows to extend to pro-Artinian local rings
many results originally provable over a field only.

 Let $\R$ be a pro-Artinian local ring with the maximal ideal~$\m$ and
the residue field~$k$.
 The category of $\R$\+contramodules $\R\contra$ and the reduction 
functor $\P\longmapsto\P/\m\P$ acting from it to the category of 
$k$\+vector spaces have the following formal properties:
\begin{enumerate}
\renewcommand{\theenumi}{\roman{enumi}}
\item $\R\contra$ is an abelian category with infinite direct sums
and products; the infinite product functors in $\R\contra$ are exact
and preserved by the forgetful functor $\R\contra\rarrow\R\modl$;
\item $\R\contra$ is a tensor category with a right exact tensor
product functor $\ot^\R$ and an internal Hom functor $\Hom^\R$;
the forgetful functor $\R\contra\rarrow\R\modl$ takes the unit
object of the tensor structure $\R\in\R\contra$ to the unit object
$\R\in\R\modl$ and commutes with the internal Hom functors;
\item there are enough projective objects in $\R\contra$; these are
called the \emph{free $\R$\+contramodules}, as they are precisely
the direct sums of copies of the object~$\R$; the class of free
$\R$\+contramodules is preserved by infinite direct sums, infinite
products, and the operations $\ot^\R$ and $\Hom^\R$;
\item the reduction functor $\R\contra\rarrow k\vect$ preserves
infinite direct sums, infinite products, commutes with the tensor
products, and commutes with the internal Hom from free
$\R$\+contramodules;
\item the reduction functor does not annihilate any objects:
$\P/\m\P=0$ implies $\P=\nobreak0$ for any $\P\in\R\contra$;
in other words, a morphism of free $\R$\+contra\-modules $\F\rarrow\gG$
is an isomorphism whenever its reduction $\F/\m\F\rarrow\gG/\m\gG$ is
an isomorphism of $k$\+vector spaces.
\end{enumerate}

 The constructions of these structures and the proofs of the listed
assertions can be found in~\cite[Sections~1.1\+-1.6]{Pweak}.
 Let us point out two caveats, which are in fact two ways to formulate
one and the same observation.
 Firstly, the infinite direct sums of $\R$\+contramodules are \emph{not}
always exact functors.
 Secondly, the functors of tensor product $\F\ot^\R{-}$ with a free
$\R$\+contramodule $\F$ is \emph{not} always exact in $\R\contra$.

 Both problems do not occur in the homological dimension~$1$ case.
 E.~g., in the case of the ring of $l$\+adic integers $\R=\Z_l$
or the ring of formal power series in one variable $\R=k[[z]]$,
the infinite direct sums in $\R\contra$ are still exact, as are
the tensor products with free $\R$\+contramodules.
 Of course, even in these cases the infinite direct sums and
tensor products of $\R$\+contramodules are not preserved
by the forgetful functor $\R\contra\rarrow\R\modl$ (and one would
not expect them to be).

 The definition of a left contramodule over a complete and separated
topological ring $\R$ with open right ideals forming a base of
neighborhoods of zero was explained in Section~\ref{over-topol-rings}.
 Let us now define the operation of tensor product $\ot^\R$ of
contramodules over a commutative topological ring $\R$ with open
ideals forming a base of neighborhoods of zero.

 The following definition of a contrabilinear map for contramodules
over a commutative topological ring was suggested to the author
by Deligne.
 Let $\P$, $\Q$, and $\gK$ be three contramodules over~$\R$.
 A map $b\:\P\times\Q\rarrow\gK$ is called \emph{contrabilinear} over
$\R$ if for any two families of coefficients $r_\alpha$ and $s_\beta\in\R$
converging to zero in the topology of $\R$ and any two families of
elements $p_\alpha\in\P$ and $q_\beta\in\Q$ the equation
$$
 b\left(\sum_\alpha r_\alpha p_\alpha\;\sum_\beta s_\beta q_\beta\right)
 = \sum_{\alpha,\beta}(r_\alpha s_\beta)\,b(p_\alpha,q_\beta)
$$
holds in~$\gK$.
 An $\R$\+contramodule $\gL$ endowed with an $\R$\+contrabilinear map
$\P\times\Q\rarrow\gL$ is called the \emph{contramodule tensor product}
of the $\R$\+contramodules $\P$ and $\Q$ if for any $\R$\+contramodule
$\gK$ and any $\R$\+contrabilinear map $\P\times\Q\rarrow\gK$ there
exists a unique $\R$\+contramodule morphism $\gL\rarrow\gK$ making
the triangle diagram $\P\times\Q\rarrow\gL\rarrow\gK$
commutative~\cite[Section~1.6]{Pweak}.

 One has to check that, for any $\R$\+subcontramodules $\P'\subset\P$
and $\Q'\subset\Q$, any $\R$\+contrabilinear map $\P\times\Q\rarrow\gK$
annihilating $\P'\times\Q$ and $\P\times\Q'$ factorizes uniquely
through a contrabilinear map $\P/\P'\times\Q/\Q'\rarrow\gK$.
 Then it follows that the contramodule tensor product $\P/\P'\ot^\R
\Q/\Q'$ can be obtained as the cokernel of the natural morphism
of contramodule tensor products
$$
 \P'\ot^\R\Q\,\oplus\,\P\ot^\R\Q'\lrarrow\P\ot^\R\Q,
$$
so in order to produce the contramodule tensor products of arbitrary
pairs of $\R$\+con\-tramodules it remains to explain what
the tensor products of free $\R$\+contramodules are.
 Here one notices that setting $\R[[X]]\ot^\R\R[[Y]]\simeq
\R[[X\times Y]]$ for any two sets $X$ and $Y$ does the job.
 Moreover, one has $\R[[X]]\ot^\R\P\simeq\bigoplus_X\P$ for any set
$X$ and any $\R$\+contramodule $\P$, and generally the contramodule
tensor product is a right exact functor preserving infinite direct
sums in the category $\R\contra$.

 The set $\Hom^\R(\P,\Q)$ of all $\R$\+contramodule morphisms between
two $\R$\+contra\-modules $\P$ and $\Q$ is endowed with
the $\R$\+contramodule structure provided by the pointwise infinite
summation operations
$$
 \left(\sum_\alpha r_\alpha f_\alpha\right)(p)=
 \sum_\alpha r_\alpha f_\alpha(p)
$$
for any family of morphisms $f_\alpha\:\P\rarrow\Q$, any element
$p\in\P$, and any family of coefficients $r_\alpha\in\R$ converging
to zero in the topology of~$\R$ \cite[Section~1.5]{Pweak}.
 One easily checks from the definitions that
$$
 \Hom^\R(\P\ot^\R\Q\;\gK)\simeq\Hom^\R(\P,\Hom^\R(\Q,\gK))
$$
for any $\R$\+contramodules $\P$, $\Q$, and $\gK$, as is required
of the internal Hom functor in a tensor category.

 Given an $\R$\+contramodule $\P$, one denotes by $\m\P\subset\P$
the image of the contraaction map $\R[[\P]]\supset\m[[\P]]\rarrow\P$.
 The passage to the quotient $k$\+vector space $\P/\m\P$ provides
the construction of the reduction functor $\P\longmapsto\P/\m\P$.

 We have explained the constructions of all the structures mentioned
in~(i\+v).
 The proof of the assertion that the class of projective
$\R$\+contramodules is closed under the operations of infinite product
and the internal Hom functor depends on the assumption that
the ring $\R$ is pro-Artinian~\cite[Lemma~1.3.7]{Pweak}, as does
the proof of the fact that the reduction functor preserves
infinite products~\cite[Lemma~1.3.6]{Pweak}.
 The proof of the Nakayama lemma~(v) is based on the assumption of
topological nilpotence of the ideal~$\m$ implied by the definition
of a pro-Artinial local ring (see Section~\ref{over-topol-rings}),
and the assertion that all projective $\R$\+contramodules are free
follows from the Nakayama lemma~\cite[Lemma~1.3.2]{Pweak}.

 In addition to $\R$\+contramodules, the coefficient formalism developed
in~\cite[Section~1]{Pweak} also includes \emph{discrete\/ $\R$\+modules}
or \emph{$\R$\+comodules}; see Section~\ref{underived-top-rings} for
a short discussion and~\cite[Sections~1.4 and~1.9]{Pweak} for
the details.

\subsection{Flat contramodules over topological rings}
\label{flat-contra}
 The aim of this section is to describe a certain class of adjusted
contramodules over topological rings which plays a crucial role in
the arguments in~\cite[Section~5]{PSl1} and~\cite[Section~7]{PSl2},
and will probably prove to be important and useful in other contexts
as well (cf.~\cite{Pproperf,PS3,BP3,BPS,PT}).
 To begin with, let us briefly return to the discussion of contramodules
over the adic completions of Noetherian rings by centrally generated
ideals from Section~\ref{over-adic-completions}.

 Let $I$ be an ideal generated by central elements in a right
Noetherian ring $R$; denote by $\R=\varprojlim_n R/I^n$
the $I$\+adic completion of the ring $R$, viewed as a complete
and separated topological ring in its natural topology.
 Let $\I=\varprojlim_n I/I^n$ denote the ideal generated by
the image of $I\subset R$ in the ring~$\R$.

 The following result is~\cite[Proposition~C.5.4]{Pcosh}
or~\cite[Lemma~10.2 and proof of Proposition~5.5 in Section~10]{PT}.
 In the commutative case, its proof can be found
in~\cite[Lemma~B.9.2]{Pweak} or~\cite[Corollary~10.3]{Pcta}.

\begin{lem1} \label{noetherian-adic-flat}
 A left\/ $\R$\+contramodule\/ $\P$ is a flat $R$\+module if and only
if the $R/I^n$\+module\/ $\P/\I^n\P$ is flat for every $n\ge1$.
 The natural map\/ $\P\rarrow\varprojlim_n\P/\I^n\P$ is an isomorphism
if this is the case. \qed
\end{lem1}

 In other words, the nonseparatedness phenomenon demonstrated by
the counterexamples from~\cite{Sim,Yek1,Psemi} described
in Section~\ref{counterexamples} above does not occur for
$\R$\+contramodules satisfying any one of the two equivalent flatness
conditions from the first sentence of Lemma.
 The situation turns out to be similar in the much more general
contexts discussed below.

 Let us now pass to the setting of~\cite[Section~D.1]{Pcosh}.
 Let $R_0\larrow R_1\larrow R_2\larrow R_3\larrow\dotsb$ be
a projective system of associative rings and surjective morphisms
between them.
 Consider the projective limit $\R=\varprojlim_n R_n$ and endow it
with the topology of projective limit of discrete rings~$R_n$.
 Let $\I_n\subset\R$ denote the kernels of the natural surjective
morphisms of rings $\R\rarrow R_n$; then the open ideals $\I_n$ form
a base of neighborhoods of zero in the topological ring~$\R$.
 To make our notation closer to that of~\cite[Appendix~D]{Pcosh}
and~\cite{PR}, we switch a bit away from our previous notation
pattern and denote by $\J\tim\P\subset\P$ the image of the contraaction
map $\J[[\P]]\rarrow\P$ for any closed ideal $\J\subset\R$ and any
$\R$\+contramodule~$\P$.

 We recall from Section~\ref{over-topol-rings} that the category
of left $\R$\+contramodules $\R\contra$ is an abelian category with
infinite direct sums, exact infinite products, and enough projective
objects.
 The next result is~\cite[Lemma~D.1.1]{Pcosh}; see
also~\cite[Lemma~A.2.3]{Psemi}.

\begin{lem2}
 For any left\/ $\R$\+contramodule\/ $\P$, the natural map\/
$\P\rarrow\varprojlim_n\P/\I_n\tim\P$ is surjective. \qed
\end{lem2}

 Here is the promised definition in this context.
 A left $\R$\+contramodule $\F$ is called \emph{flat} if
the left $R_n$\+module $\F/\I_n\tim\F$ is flat for every~$n$.
 The following proposition is the main related result.

\begin{prop1}
\textup{(a)} For any flat left\/ $\R$\+contramodule\/ $\F$, the map\/
$\F\rarrow\varprojlim_n\F/\I_n\tim\F$ is an isomorphism. \par
\textup{(b)} The class of flat left\/ $\R$\+contramodules is closed
under extensions and the passages to the kernels of surjective
morphisms.
 For any short exact sequence of flat\/ $\R$\+contramodules\/
$0\rarrow\gH\rarrow\gG\rarrow\F\rarrow0$, the short sequences
of $R_n$\+modules\/ $0\rarrow\gH/\I_n\tim\gH\rarrow\gG/\I_n\tim\gG
\rarrow\F/\I_n\tim\F\rarrow0$ are exact.
\end{prop1}

\begin{proof}
 Part~(a) is~\cite[Corollary~D.1.7]{Pcosh}.
 Part~(b) is~\cite[Lemmas~D.1.4 and~D.1.5]{Pcosh}.
\end{proof}

 The next result provides a characterization of projective
$\R$\+contramodules, generalizing the results
of~\cite[Corollary~4.5]{Yek1} and~\cite[Corollary~B.8.2]{Pweak}
(see also~\cite[Theorem~1.10]{PSY2}
and~\cite[Corollary~C.5.6(a)]{Pcosh}).

\begin{prop2}
 A left\/ $\R$\+contramodule\/ $\F$ is a projective object in\/
$\R\contra$ if and only if the left $\R_n$\+modules\/ $\F/\I_n\tim\F$
are projective for all~$n\ge0$.
\end{prop2}

\begin{proof}
 This is~\cite[Corollary~D.1.10(a)]{Pcosh}.
\end{proof}

 Finally, we consider the quite general setting of a complete, separated
topological ring\/ $\R$ with a countable base of neighborhoods of zero
consisting of open right ideals~\cite[Sections~5\+-7]{PR}.
 The next lemma is~\cite[Lemma~6.3(b)]{PR}.

\begin{lem3}
 For any left\/ $\R$\+contramodule\/ $\P$, the natural map\/
$\P\rarrow\varprojlim_\I\P/\I\tim\P$, where the projective limit is
taken over the poset of open right ideals\/ $\I\subset\R$ ordered by
the inverse inclusion, is surjective. \qed
\end{lem3}

 Let $\N$ be a discrete right $\R$\+module and $\P$ be a left
$\R$\+contramodule.
 Then the \emph{contratensor product} $\N\ocn_\R\P$ is an abelian
group constructed as the quotient group of $\N\ot_\Z\P$ by
the subgroup of all elements of the form
$$
 \sum_\alpha xr_\alpha\ot p_\alpha - x\ot \sum_\alpha r_\alpha p_\alpha.
$$
 Here $r_\alpha\in\R$ is family of elements of converging to zero in
the topology of $\R$, indexed by some family of indices~$\alpha$
(which can be assumed to be at most countable when $\R$ has a countable
base of neighborhoods of zero), $p_\alpha\in\P$ is an arbitrary family
of elements, and $x\in\N$ is an element.
 The sum in the left-hand side is well-defined since $xr_\alpha=0$ for
all but a finite subset of indices~$\alpha$, while the sum in
the right-hand side denotes the contramodule infinite summation
operation.

 For any discrete right $\R$\+module $\N$, any left $\R$\+contramodule
$\P$, and any abelian group~$V$, there is a natural isomorphism of
abelian groups
$$
 \Hom_\Ab(\N\ocn_\R\P\;V)\simeq\Hom^\R(\P,\Hom_\Ab(\N,V)),
$$
where the left $\R$\+contramodule structure on the abelian group
$\Hom_\Ab(\N,V)$ is constructed as explained in
Section~\ref{over-topol-rings}.
 For any set $X$ and any discrete right $\R$\+module $\N$, there is
a natural isomorphism of abelian groups
$$
 \N\ocn_\R\R[[X]]\simeq\N[X].
$$

 Here is our key definition.
 A left $\R$\+contramodule $\F$ is called \emph{flat} if the functor
of contratensor product $\N\longmapsto\N\ocn_\R\F$ is exact on
the abelian category of discrete right $\R$\+modules~$\N$.
 For a (complete and separated) topological ring $\R$ with a countable
base of neighborhoods of zero consisting of open \emph{two-sided}
ideals, this definition is equivalent to the previous one.
 The following proposition lists some of the main results of
the theory.

\begin{prop3}
\textup{(a)} For any flat left\/ $\R$\+contramodule\/ $\F$, the map\/
$\F\rarrow\varprojlim_\I\F/\I\tim\F$ is an isomorphism. \par
\textup{(b)} All projective left\/ $\R$\+contramodules are flat.
 The class of flat left\/ $\R$\+contra\-modules is closed
under extensions and the passages to the kernels of surjective
morphisms.
 For any short exact sequence of flat left\/ $\R$\+contramodules\/
$0\rarrow\gH\rarrow\gG\rarrow\F\rarrow0$ and any discrete right\/
$\R$\+module\/ $\N$, the short sequence of abelian groups\/
$0\rarrow\N\ocn_\R\gH\rarrow\N\ocn_\R\gG\rarrow\N\ocn_\R\F\rarrow0$
is exact. \par
\textup{(c)} The class of flat left\/ $\R$\+contramodules is closed
under filtered inductive limits in\/ $\R\contra$.
 The functor of filtered inductive limit of diagrams of flat left\/
$\R$\+contramodules is exact.
\end{prop3}

\begin{proof}
 Part~(a) is~\cite[Corollary~6.15]{PR}.
 Part~(b) is~\cite[Lemma~6.7, Corollary~6.8, Lemma~6.9, and
Corollary~6.13]{PR} (see also~\cite[Lemma~6.10]{PR}).
 Part~(c) is~\cite[Lemmas~5.6 and~6.16]{PR} (see
also~\cite[Proposition~6.17]{PR}).
\end{proof}

 An even more general setting of contramodules over a complete,
separated topological ring with a (not necessarily countable) base
of neighborhoods of zero consisting of open right ideals is
discussed in the papers~\cite{Pproperf,PS3,BPS,PT}.
 In this context, the definition of a flat contramodule is the same
as in the previous one.
 The class of flat contramodules is still closed under filtered
inductive limits, but it \emph{no} longer needs to be closed under
the kernels of surjective morphisms~\cite[Example~12.4]{PT}.
 An example of a flat contramodule which cannot be obtained as
a filtered inductive limit of projective ones can be found
in~\cite[Example~9.2]{PT}.
 Topological rings over which the classes of projective and flat
left contramodules coincide are characterized, under the name of
\emph{topologically left perfect} topological rings, in
the paper~\cite[Section~14]{PS3}.

\subsection{Underived co-contra correspondence over corings}
\label{underived-over-corings}
 It was already mentioned in the end of Section~\ref{basic-properties}
that the categories of \emph{injective} left \emph{comodules} and
\emph{projective} left \emph{contramodules} over a coalgebra $\C$ over
a field~$k$ are naturally equivalent.
 Similarly, at the end of Section~\ref{over-l-adics} we pointed out
the equivalence between the categories of injective discrete modules
and projective contramodules over the ring of $l$\+adic integers~$\Z_l$.
 These are the simplest instances of a very general homological
phenomenon called the \emph{comodule-contramodule correspondence},
which has many manifestations in
algebra~\cite{Mat2,DG,PSY,Jor,Kra,IK,Pkoszul,Pweak,Pmgm,Pfp},
algebraic geometry~\cite{Neem,Murf,Psing,Pcosh}, and representation
theory~\cite{FF0,FF,RCW,Psemi,PS1,PS2}.
 In the following three sections we restrict ourselves to an overview of
those versions of the co-contra correspondence that can be readily
formulated on the level of \emph{additive} or \emph{exact} categories,
while referring to the presentation~\cite{Psli} and the introduction to
the paper~\cite{Pmgm} for a discussion of the derived
comodule-contramodule correspondence.

 Let $\C$ be a coassociative coring over an associative ring~$A$
(see Section~\ref{over-corings}).
 Then the assignment of the left $\C$\+contramodule $\Hom_A(\C,V)$ to
the left $\C$\+comodule $\C\ot_AV$ and vice versa establishes
an equivalence between the full additive subcategories of
coinduced $\C$\+comodules in $\C\comodl$ and induced
$\C$\+contramodules in $\C\contra$.
 Indeed, one has
\begin{multline*}
 \Hom_\C(\C\ot_AU\;\C\ot_AV)\simeq\Hom_A(\C\ot_AU\;V)\\ \simeq
 \Hom_A(U,\Hom_A(\C,V))\simeq\Hom^\C(\Hom_A(\C,U),\Hom_A(\C,V))
\end{multline*}
for any left $A$\+modules $U$ and~$V$ \cite[Section~0.2.6]{Psemi}.
 This is a particular case of the isomorphism of \emph{Kleisli
categories} for a pair (left adjoint comonad, right adjoint monad)
in any base category~\cite{Kle,BBW}.
 
 Adjoining the direct summands to the full subcategories of
coinduced $\C$\+comodules and induced $\C$\+contramodules, one
obtains what are called the full subcategories of relatively
injective (or $\C/A$\+injective) $\C$\+comodules and relatively
projective $\C$\+contra\-modules in
the book~\cite[Sections~18.17\+-18]{BW} and
the paper~\cite[Sections~2.7\+-8]{BBW}.
 This may be the standard terminology; still in
the monograph~\cite[Section~5.1.3]{Psemi} we chose to call these
\emph{quite relatively injective} (quite $\C/A$\+injective)
$\C$\+comodules and \emph{quite relatively projective}
(quite $\C/A$\+projective) $\C$\+contramodules, while preserving
the shorter terms with the single word ``relatively'' for
the wider and more important classes of comodules and
contramodules discussed below.

 A left $\C$\+comodule $\cJ$ is said to be \emph{quite
$\C/A$\+injective} if the short sequence of abelian groups
\begin{equation} \label{comodule-hom-sequence}
 0\rarrow\Hom_\C(\M,\cJ)\rarrow\Hom_\C(\L,\cJ)
 \rarrow\Hom_\C(\K,\cJ)\rarrow0
\end{equation}
is exact for every short exact sequence of left $\C$\+comodules
$0\rarrow\K\rarrow\L\rarrow\M\rarrow0$ that \emph{splits} as
a short exact sequence of left \emph{$A$\+modules}.
 We recall that the category of left $\C$\+comodules is not abelian
in general, there being a problem with the kernels of morphisms
(see Section~\ref{over-corings}).
 However, any $A$\+split surjection of $\C$\+comodules has
a kernel preserved by the forgetful functor $\C\comodl\rarrow
A\modl$.

 Moreover, the category of left $\C$\+comodules with the class of
all $A$\+split short exact sequences is an exact category
(see~\cite{Bueh} for the definition, discussion, and references);
the quite $\C/A$\+injective $\C$\+comodules are simply
the injective objects of this exact category structure.
 In particular, the coaction morphism $\M\rarrow\C\ot_A\M$ of
any left $\C$\+comodule $\M$ is split by the $A$\+module map
$\C\ot_A\M\allowbreak\rarrow\M$ induced by the counit~$\eps$
of the coring~$\C$.
 Considering the corresponding $A$\+split short exact sequence of
$\C$\+comodules, one easily concludes that a $\C$\+comodule is
quite $\C/A$\+injective if and only if it is a direct summand of
a coinduced $\C$\+comodule.

 Similarly, a left $\C$\+contramodule $\F$ is said to be \emph{quite
$\C/A$\+projective} if the short sequence of abelian groups
\begin{equation} \label{contramodule-hom-sequence}
 0\rarrow\Hom^\C(\F,\P)\rarrow\Hom^\C(\F,\Q)
 \rarrow\Hom^\C(\F,\R)\rarrow0
\end{equation}
is exact for every short exact sequence of left $\C$\+contramodules
$0\rarrow\P\rarrow\Q\rarrow\R\rarrow0$ that \emph{splits} as
a short exact sequence of left $A$\+modules.
 Recall that the category of left $\C$\+contramodules is not abelian
in general, there being a problem with the cokernels of morphisms.
 However, any $A$\+split embedding of $\C$\+contramodules has
a cokernel preserved by the forgetful functor $\C\contra\rarrow A\modl$.

 Moreover, the category of left $\C$\+contramodules with the class of
all $A$\+split short exact sequences is an exact category;
the quite $\C/A$\+projective $\C$\+contramodules are simply
the projective objects of this exact category structure.
 In particular, the contraaction morphism $\Hom_A(\C,\P)\rarrow\P$
of any $\C$\+contramodule $\P$ is split by the $A$\+module map
$\P\rarrow\Hom_A(\C,\P)$ induced by the counit of the coring~$\C$.
 Considering the corresponding $A$\+split short exact sequence of
$\C$\+contramodules, one easily concludes that a $\C$\+contramodule
is quite $\C/A$\+projective if and only if it is a direct summand of
an induced $\C$\+contramodule.

 So the full additive subcategories of quite $\C/A$\+injective
$\C$\+comodules in $\C\comodl$ and quite $\C/A$\+projective
$\C$\+contramodules in $\C\contra$ are equivalent for any
coassociative coring~$\C$.
 When the coring $\C$ coincides with the ring~$A$ (i.~e.,
the counit map $\C\rarrow A$ is bijective), this reduces to
the identity equivalence of the category of left $A$\+modules
with itself.
 However, the category of $A$\+modules is abelian and not only
additive; viewing it just as an additive category is rather
unsatisfactory from our point of view.
 Still the categories of quite relatively injective comodules and
quite relatively projective contramodules do not seem to carry
any homological structures beyond those of additive categories;
in particular, they do \emph{not} have any nontrivial exact category
structures.
 The following definitions~\cite[Sections~5.1.4 and~5.3]{Psemi}
are purported to overcome this drawback (cf.~\cite[Sections~4.1
and~4.3]{BP}, where similarly defined relatively adjusted
modules are called, more in line with the traditional
terminology, ``weakly relatively projective'' and ``weakly 
relatively injective'').

 Assume that the coring $\C$ is a projective left and a flat right
$A$\+module; then the categories of left $\C$\+comodules and
left $\C$\+contramodules are abelian.
 A left $\C$\+comodule $\cJ$ is called \emph{injective relative
to~$A$} ($\C/A$\+injective) if the short sequence of Hom
groups~\eqref{comodule-hom-sequence} is exact for any short
exact sequence of left $\C$\+comodules $0\rarrow\K\rarrow\L\rarrow\M
\rarrow0$ that are \emph{projective as left $A$\+modules}.
 Similarly, a left $\C$\+contramodule $\F$ is called \emph{projective
relative to~$A$} ($\C/A$\+projective) if the short sequence of Hom
groups~\eqref{contramodule-hom-sequence} is exact for any short
exact sequence of left $\C$\+contramodules $0\rarrow\P\rarrow\Q
\rarrow\R\rarrow0$ that are \emph{injective as left $A$\+modules}.

 Assume further that the ring $A$ has finite left homological dimension
(i.~e., the category of left $A$\+modules has finite homological
dimension; cf.\ the last sentence of Section~\ref{over-semialgebras}).
 Then the full subcategory $\C\comodl_{\C/A\dinj}$ of
$\C/A$\+injective $\C$\+comodules is closed under extensions and
the passages to the cokernels of injective morphisms in $\C\comodl$.
 Similarly, the full subcategory $\C\contra_{\C/A\dproj}$ of
$\C/A$\+projective $\C$\+contra\-modules is closed under extensions
and the passages to the kernels of subjective morphisms in $\C\contra$
\cite[Lemma~5.3.1]{Psemi}.
 Being closed under extensions, the full subcategories
$\C\comodl_{\C/A\dinj}$ and $\C\contra_{\C/A\dproj}$ inherit the exact
category structures of the abelian categories $\C\comodl$ and
$\C\contra$.
 Moreover, the following strong converse assertions hold.

\begin{lem} \hbadness=2250
\textup{(a)} The subcategory of\/ $\C/A$\+injective\/ $\C$\+comodules\/
$\C\comodl_{\C/A\dinj}\subset\C\comodl$ is the minimal full subcategory
of the abelian category\/ $\C\comodl$ containing the coinduced\/
$\C$\+comodules and closed under extensions and direct summands. \par
\textup{(b)} The subcategory of\/ $\C/A$\+projective\/
$\C$\+contramodules\/ $\C\contra_{\C/A\dproj}\subset\C\contra$ is
the minimal full subcategory of the abelian category\/ $\C\contra$
containing the induced $\C$\+contramodules and closed under
extensions and direct summands.
\end{lem}

\begin{proof}
 This is the strengthening of the result of~\cite[Remark~9.1]{Psemi}
that one obtains by replacing the resolution technique
of~\cite[Lemma~9.1.2]{Psemi} with that of~\cite[second half of
the proof of Theorem~10]{ET}; see also~\cite[Corollary~B.2.4]{Pcosh}.

 To be more specific, let us prove part~(b).
 Let $\P$ be a left $\C$\+contramodule.
 The $\C$\+contraaction map $\gG=\Hom_A(\C,\P)\rarrow\P$ is a surjective
morphism of $\C$\+contra\-modules; let us denote its kernel by~$\gK$.
 According to~\cite[Lemma~3.1.3(b)]{Psemi}, there exists an injective
$\C$\+contramodule morphism $\gK\rarrow\E$ from $\gK$ into
an $A$\+injective left $\C$\+contramodule $\E$ such that the quotient
contramodule $\E/\gK$ is a finitely iterated extension of induced
$\C$\+contramodules.
 Denote by $\F$ the fibered coproduct $(\E\oplus\gG)/\gK$ of
the $\C$\+contramodules $\E$ and $\gG$ over~$\gK$; then $\F$ is
an extension of the $\C$\+contramodules $\E/\gK$ and $\gG$, and
there is a surjective morphism $\F\rarrow\P$ with the kernel~$\E$.

 Now suppose that the $\C$\+contramodule $\P$ is $\C/A$\+projective.
 Then the Ext group $\Ext^{\C,\.1}(\P,\E)$ in the abelian category
$\C\contra$ vanishes by~\cite[Lemma~5.3.1(b)]{Psemi}, hence
the $\C$\+contramodule $\P$ is a direct summand of a finitely
iterated extension of induced $\C$\+contramodules~$\F$.
 Notice that the length of the iterated extension in this construction
is bounded by the left homological dimension of the ring~$A$.
\end{proof}

 According to~\cite[Theorem~5.3]{Psemi}, the exact categories of
$\C/A$\+injective left $\C$\+comod\-ules and $\C/A$\+projective left 
$\C$\+contramodules are naturally equivalent
$$
 \C\comodl_{\C/A\dinj}\simeq\C\contra_{\C/A\dproj}.
$$
 The equivalence is provided by the functor $\Psi_\C\:\M\longmapsto
\Hom_\C(\C,\M)$ of $\C$\+comodule homomorphisms from the left
$\C$\+comodule $\C$ and the functor $\Phi_\C\:\P\longmapsto\C\ocn_\C\P$
of \emph{contratensor product} of left $\C$\+contramodules with
the right $\C$\+comodule~$\C$ \cite[Sections~0.2.6\+-7 and~5.1.1]{Psemi}
(cf.\ Section~\ref{tensor-adjusted-over-coalgebra};
see also~\cite[Section~5]{BBW}).
 This equivalence of exact categories can be viewed as an instance of
the \emph{$\infty$\+tilting-cotilting correspondence}
phenomenon~\cite[Example~6.7]{PS2}.

\subsection{Underived semico-semicontra correspondence}
\label{underived-semi}
 Let $\S$ be a semialgebra over a coalgebra $\C$ over a field~$k$
(see Section~\ref{over-semialgebras}).
 Assume that $\S$ is an injective left $\C$\+comodule and an injective
right $\C$\+comodule, so the categories of left $\S$\+semimodules and
left $\S$\+semicontramodules $\S\simodl$ and $\S\sicntr$ are abelian.
 The full subcategory $\S\simodl_{\C\dinj}$ of left $\S$\+semimodules
that are \emph{injective as left\/ $\C$\+comodules} is obviously
closed under extensions (and cokernels of injective morphisms) in
$\S\simodl$, while the full subcategory $\S\sicntr_{\C\dproj}$ of
left $\S$\+semicontra\-modules that are \emph{projective as left\/
$\C$\+contramodules} is closed under extensions (and kernels of
surjective morphisms) in $\S\sicntr$.
 Hence the full subcategories $\S\simodl_{\C\dinj}$ and
$\S\sicntr_{\C\dproj}$ inherit the exact category structures of
the abelian categories $\S\simodl$ and $\S\sicntr$.

 According to~\cite[Sections~0.3.7 and~6.2]{Psemi} (see
also~\cite{Plet}), the exact categories of $\C$\+injective left
$\S$\+semimodules and $\C$\+projective left $\S$\+semicontramodules
are naturally equivalent.
 The equivalence is provided by the functor $\Psi_\S\:\bM\longmapsto
\Hom_\S(\S,\bM)$ of $\S$\+semimodule homomorphisms from the left
$\S$\+semimodule $\S$ and the functor $\Phi_\S\:\bP\longmapsto
\S\Ocn_\S\bP$ of (\emph{semi})\emph{contratensor product} of
left $\S$\+semicontramodules with the right $\S$\+semimodule~$\S$
(cf.\ the definition of the contratensor product over a coalgebra
in Section~\ref{tensor-adjusted-over-coalgebra}).
 Furthermore, the equivalence of exact categories $\S\simodl_{\C\dinj}
\simeq\S\sicntr_{\C\dproj}$ forms a commutative diagram of functors with
the equivalence $\C\comodl_\inj\simeq\C\contra_\proj$ between
the additive categories of injective left $\C$\+comodules and
projective left $\C$\+contramodules, and the forgetful functors
$\S\simodl_{\C\dinj}\rarrow\C\comodl_\inj$ and
$\S\sicntr_{\C\dproj}\rarrow\C\contra_\proj$,
\begin{equation} \label{semi-co-contra} \dgARROWLENGTH=4em
\begin{diagram} 
\node{\S\simodl_{\C\dinj}}
\arrow[2]{e,tb,=}{\Psi_\S\ \ \rightarrow}{\leftarrow\ \ \Phi_\S}
\arrow{s}
\node[2]{\S\sicntr_{\C\dproj}} \arrow{s} \\
\node{\C\comodl_\inj}
\arrow[2]{e,tb,=}{\Psi_\C\ \ \rightarrow}{\leftarrow\ \ \Phi_\C}
\node[2]{\C\contra_\proj.}
\end{diagram}
\end{equation}
 The equivalence of exact categories in the upper line of
the diagram~\eqref{semi-co-contra} can be also viewed as an instance of
the $\infty$\+tilting-cotilting correspondence~\cite[Example~6.9]{PS2}.

 The particular case considered in Section~\ref{tate-harish-chandra},
with the semialgebra
$$
 \S^r_{\kap+\kap_0}(\g,\C)\simeq\S=\S^l_\kap(\g,\C)
$$
corresponding to a central extension $\kap\:(\g',\C)\rarrow(\g,\C)$
of Tate Harish-Chandra pairs satisfying the condition of
Theorem~\ref{tate-harish-chandra}, is especially notable.
 In this situation we obtain the commutative diagram
of equivalences of exact/additive categories and forgetful functors
\begin{equation} \label{o-co-contra} \dgARROWLENGTH=4em
\begin{diagram} 
\node{\sO_\kap(\g,\C)_{\C\dinj}}
\arrow[2]{e,tb,=}{\Psi_\S\ \ \rightarrow}{\leftarrow\ \ \Phi_\S}
\arrow{s}
\node[2]{\sO^\ctr_{\kap+\kap_0}(\g,\C)_{\C\dproj}} \arrow{s} \\
\node{\C\comodl_\inj}
\arrow[2]{e,tb,=}{\Psi_\C\ \ \rightarrow}{\leftarrow\ \ \Phi_\C}
\node[2]{\C\contra_\proj,}
\end{diagram}
\end{equation}
where $\sO_\kap(\g,\C)_{\C\dinj}\subset\sO_\kap(\g,\C)$ and
$\sO^\ctr_\kap(\g,\C)_{\C\dproj}\subset\sO^\ctr_\kap(\g,\C)$ denote
the full exact subcategories of $\C$\+injective Harish-Chandra
modules and $\C$\+projective Harish-Chandra contramodules with
the central charge~$\kap$.

\bigskip

 It was pointed out by Feigin--Fuchs~\cite{FF0}, \cite[Remark~2.4]{FF}
and Meurman--Frenkel--Rocha\+Caridi--Wallach~\cite{RCW} back in
the first half of 1980's that the categories of Verma modules
over the Virasoro algebra on any pair of complementary levels
$C=c$ and $C=26-c$ are anti-isomorphic.
 The above result extends this classical observation to the whole
exact subcategories of $\C$\+adjusted objects in the abelian
categories $\sO$ and~$\sO^\ctr$ over any Tate Harish-Chandra
pair satisfying the nondegeneracy condition.

 Indeed, consider the Tate Harish-Chandra pair
$(\g,\C)=\big(k((z))d/dz,\C(H)\big)$ over a field~$k$ of
characteristic~$0$ with the pro-algebraic subgroup $H$
corresponding to the compact open Lie subalgebra $\h=zk[[z]]d/dz
\subset k((z))d/dz$ as described in the beginning of
Section~\ref{tate-harish-chandra}.
 The group~$H$ acts in the Lie algebra~$\g$ by changing
the independent variable in the vector fields,
$a^{-1}\big(f(z)d/dz\big) = f(a(z))\,d/da(z)=f(a(z))/a'(z)\,d/dz$
for all $a\in H(k)$, \,$f\in k((z))$.

 A \emph{Verma module} over $\Vir$ is an $U(\Vir)$\+module induced
from a one-dimensional module $kv_0$ over the compact open subalgebra
$\h\oplus kC\subset\Vir$.
 The subalgebra $z^2k[[z]]d/dz\subset\h$ topologically spanned
by the basis vectors $L_n$ with $n\ge1$ acts by zero in~$kv_0$,
while the generators $C$ and $L_0$ act by certain scalars $c$
and~$h_0\in k$.
 These modules belong to the categories $\sO_\kap(\g,\C)$ with
the respective central charges $\kap=c$, but, as such, do \emph{not}
play any noticeable role in our theory.
 Indeed, they do \emph{not} belong to the subcategories
$\sO_\kap(\g,\C)_{\C\dinj}\subset\sO_\kap(\g,\C)$, being freely
generated as modules over a Lie subalgebra \emph{complementary}
to~$\h$ in~$\Vir$ and having no particular adjustedness properties
as comodules over $\C=\C(H)$.

 The \emph{contragredient} Verma modules are relevant for us instead.
 Notice first of all that the pro-algebraic group $H$ contains
a subgroup whose group of points consists of the coordinate
changes $z\longmapsto a_1z$ multiplying the coordinate~$z$ with
a scalar factor $a_1\in k$.
 The category of comodules (as well as contramodules) over
(the coalgebra of) this algebraic group, which is isomorphic to
the multiplicative group $\mathbb G_m$, is equivalent to
the category of graded $k$\+vector spaces.
 In particular, the Verma modules $\M(c,h_0)$ over $\Vir$ carry
the grading by the weights of the semisimple operator~$L_0$.

 Furthermore, the discrete Lie subalgebra $\bigoplus_n kL_n\oplus C
\subset\Vir$ spanned \emph{nontopologically} by the generators
$C$ and $L_n\in\Vir$ has an involutive automorphism~$\sigma$ given
by the rules $\sigma(L_n)=-L_{-n}$ and $\sigma(C)=-C$.
 The \emph{contragredient Verma module} $\M(c,h_0)\dual$ is
the graded dual vector space to $\M(c,h_0)$ endowed by the induced
action of the Lie subalgebra $\bigoplus_n kL_n\oplus C$, twisted
by the automorphism~$\sigma$ and then extended to the whole 
Lie algebra~$\Vir$ by continuity.
 Both the passage to the dual module and the involution~$\sigma$
change the sign of the central charge; hence $\M(c,h_0)\dual$
is again a $\Vir$\+module with the central charge~$c$.
 So the full subcategory of Verma modules on the level $\kap=c$
(with a varying parameter~$h_0\in k$) in $\sO_\kap(\g,\C)$ is
anti-isomorphic to the full subcategory of contragredient Verma
modules in the same category $\sO_\kap(\g,\C)$.

 Denoting by $H_+\subset H$ the pro-unipotent pro-algebraic subgroup
whose points are the power series $a(z)=1+a_2z^2+a_3z^3+\dotsb$ with
$a_n\in k$ for $n\ge2$, the contragredient Verma modules can be
described as precisely those objects of the category $\sO_\kap(\g,\C)$
whose structures of $\C(H_+)$\+comodules are those of cofree
comodules with one cogenerator. 
 Alternatively, the contragredient Verma modules are distinguished
among all the objects of $\sO_\kap(\g,\C)$ by the property that
their underlying $\C$\+comodules are the ``minimal possible'',
i.~e., \emph{indecomposable} injective comodules.

 Finally, consider the category of Harish-Chandra \emph{contramodules}
over $(\g,\C)$ with the central charge $\kap+\kap_0=-26+c$ whose
$\C(H_+)$\+contramodule structures are those of free contramodules
with one generator, or equivalently, whose underlying
$\C$\+contramodules are indecomposable projective contramodules.
 This full subcategory in $\sO_{\kap+\kap_0}^\ctr(\g,\C)$ is
anti-equivalent to the category of contragredient Verma modules
with the central charge $26-c$ via the linear duality functor
$\P=\N^*=\Hom_k(\N,k)$.
 It is also \emph{equivalent} to the category of Verma modules with
the central charge $26-c$ via the functor assigning to a Verma
module $\M$ the \emph{infinite product} $\P=\prod_{n\in\Z}\M_n$ of
its grading components, endowed with the $\sigma$\+twisted
action of the Virasoro Lie algebra and a natural
$\C$\+contramodule structure (cf.\ the discussion of graded
contramodules in~\cite[Section~11.1.1]{Psemi}
and~\cite[Remark~2.2]{Pkoszul}).

 So the classical duality between the categories of Verma modules
on the complementary levels appears in our setting, after
the na\"\i ve twisting and linear duality are taken into account,
as the restriction of the equivalence of exact
categories~\eqref{o-co-contra} to the full subcategories
of objects ``of the minimal possible size''.
 (Cf.~\cite[Corollary and Remark~D.3.1]{Psemi}, where a discussion
of these results in the \emph{derived} comodule-contramodule
correspondence context can be found; see
also~\cite[Sections~0.2.6\+-7]{Psemi} for relevant counterexamples 
demonstrating how exotic derived categories appear
in the derived co-contra correspondence.)

\bigskip

 Now let us explain, as it was promised in
Sections~\ref{over-semialgebras} and~\ref{tate-harish-chandra},
how to use the equivalence of exact categories~\eqref{semi-co-contra}
in order to construct injective objects in the category
$\S\simodl$ and projective objects in the category $\S\sicntr$.
 As above, we assume that the semialgebra $\S$ is an injective left
and right $\C$\+comodule.

\begin{prop}
\textup{(a)} There are enough injective objects in the abelian
category of left\/ $\S$\+semimodules.
 A left\/ $\S$\+semimodule is injective if and only if it is
a direct summand of an\/ $\S$\+semimodule of the form\/
$\Phi_\S(\Hom_k(\S,V))$, where $V$ is a $k$\+vector space. \par
\textup{(b)} There are enough projective objects in the abelian
category of left\/ $\S$\+semicontra\-modules.
 A left\/ $\S$\+semicontramodule is projective if and only if it is
a direct summand of an\/ $\S$\+semicontramodule of the form\/
$\Psi_\S(\S\ot_kV)$, where $V$ is a $k$\+vector space.
\end{prop}

\begin{proof}
 A left semimodule $\bM$ over a semialgebra $\S$ over a coalgebra
$\C$ over a field~$k$ is called \emph{semiprojective} if it is
a direct summand of the $\S$\+semimodule $\S\ot_k V$ for some
$k$\+vector space~$V$.
 Similarly, a left semicontramodule $\bP$ over $\S$ is called
\emph{semiinjective} if it is a direct summand of
the $\S$\+semicontramodule $\Hom_k(\S,V)$ for some vector space~$V$
\cite[Sections~3.4.3 and~9.2]{Psemi}.

 The semiprojective semimodules are projective objects in
the exact category of $\C$\+injective left $\S$\+semimodules,
as the functor
\begin{multline*}
 \bM\longmapsto\Hom_\S(\S\ot_kV\;\bM)\simeq\Hom_k(V,\Psi_\S(\bM)) \\
 \simeq\Hom_k(V,\Psi_\C(\bM))=\Hom_k(V,\Hom_\C(\C,\bM))\simeq
 \Hom_\C(\C\ot_kV\;\bM)
\end{multline*}
is exact on $\S\simodl_{\C\dinj}$.
 For any $\C$\+injective left $\S$\+semimodule $\bM$,
the semiaction morphism $\S\oc_\C\bM\rarrow\bM$ is an admissible
epimorphism in $\S\simodl_{\C\dinj}$ from a semiprojective left
$\S$\+semimodule $\S\oc_\C\bM$ onto~$\bM$; so the projective
objects of the category $\S\simodl_{\C\dinj}$ are precisely
the semiprojective semimodules.

 Similarly, the semiinjective semicontramodules are injective objects
in the exact category of $\C$\+projective left $\S$\+semicontramodules,
as the functor
\begin{multline*}
 \bP\longmapsto\Hom^\S(\bP\;\Hom_k(\S,V))
 \simeq\Hom_k(\Phi_\S(\bP),V) \\ \simeq
 \Hom_k(\Phi_\C(\bP),V)\simeq\Hom^\C(\bP\;\Hom_k(\C,V))
\end{multline*}
is exact on $\S\sicntr_{\C\dproj}$.
 For any $\C$\+projective left $\S$\+semicontramodule $\bP$,
the semicontraaction morphism $\bP\rarrow\Cohom_\C(\S,\bP)$ is
an admissible monomorphism in $\S\sicntr_{\C\dproj}$ from $\bP$ into
a semiinjective left $\S$\+semicontramodule $\Cohom_\C(\S,\bP)$.
 So the injective objects of the category $\S\sicntr_{\C\dproj}$ are
precisely the semiinjective semicontramodules.

 The functors $\Psi_\S$ and $\Phi_\S$ being mutually inverse
equivalences between the exact categories $\S\simodl_{\C\dinj}$ and
$\S\sicntr_{\C\dproj}$, it follows that the $\S$\+semimodules
$\Phi_\S(\Hom_k(\S,V))$ and their direct summands are the injective
objects of the exact category $\S\simodl_{\C\dinj}$, while
the $\S$\+semicontramodules $\Psi_\S(\S\ot_kV)$ and their direct
summands are the projective objects of the exact category
$\S\sicntr_{\C\dproj}$.
 Furthermore, any left $\S$\+semimodule can be embedded into
a $\C$\+injective $\S$\+semimodule.
 This assertion is provided by the combination of the construction
of~\cite[Lemma~1.3.3]{Psemi} (see also~\cite{Plet}) with the result
of Lemma~\ref{tensor-adjusted-over-coalgebra}(a) above.

 Similarly, any left $\S$\+semicontramodule is the quotient
contramodule of a $\C$\+pro\-jective $\S$\+semicontramodule.
 To prove this fact, one has to combine the construction
of~\cite[Lemma~3.3.3]{Psemi} (which was present already
in~\cite{Plet}) with the assertion of~\cite[Lemma~5.2 or~5.3.2]{Psemi}
whose proof we reproduced, in our generality of coalgebras over fields,
in Lemma~\ref{tensor-adjusted-over-coalgebra}(b) above.
 Therefore, any left $\S$\+semimodule can be embedded into
an $\S$\+semimodule of the form $\Phi_\S(\Hom_k(\S,V))$, and
any left $\S$\+semicontramodule is the quotient contramodule of
an $\S$\+semicontramodule of the form $\Psi_\S(\S\ot_kV)$.

 Finally, in order to show that any injective object $\bcJ$ of
the exact category $\S\simodl_{\C\dinj}$ is also an injective
object in the abelian category $\S\simodl$, suppose that we are
given an injective morphism $\bcJ\rarrow\bL$ from $\bcJ$ into
a left $\S$\+semimodule~$\bL$.
 Let $\bL\rarrow\bM$ be an injective morphism from $\bL$ into
a $\C$\+injective left $\S$\+semimodule~$\bM$.
 Then $\bcJ$ is a direct summand in $\bM$, so there is a projection
$\bM\rarrow\bcJ$ splitting the embedding $\bcJ\rarrow\bM$.
 Restricting this projection to the subsemimodule $\bL\subset\bM$,
we see that the embedding $\bcJ\rarrow\bL$ also splits.

 Similarly, in order to prove that any projective object $\bF$ of
the exact category $\S\sicntr_{\C\dproj}$ is also a projective
object in the abelian category $\S\sicntr$, suppose that we
are given a surjective morphism $\bQ\rarrow\bF$ onto $\bF$ from
a left $\S$\+semi\-contramodule~$\bQ$.
 Let $\bP\rarrow\bQ$ be a surjective morphism onto $\bQ$ from
a $\C$\+projective left $\S$\+semicontramodule~$\bP$.
 Then $\bF$ is a direct summand in $\bP$, so there is a section
$\bF\rarrow\bP$ splitting the surjection $\bP\rarrow\bF$.
 Composing this section with the projection $\bP\rarrow\bQ$,
we obtain a section $\bF\rarrow\bQ$ showing that the surjection
$\bQ\rarrow\bF$ also splits (cf.~\cite[proof of Lemma~9.2.1]{Psemi}).
\end{proof}

\subsection{Co-contra correspondence over topological rings}
\label{underived-top-rings}
 In this section we discuss generalizations of the equivalence
between the additive categories of injective comodules and projective
contramodules over a coalgebra over a field to topological rings $\R$
more complicated than the linearly compact topological algebras
(which are dual to coalgebras over fields).
 For examples of \emph{derived} co-contra correspondence over
topological rings the reader is referred to~\cite[Sections~C.1,
C.5, and~D.2]{Pcosh} (see also~\cite[Sections~8 and~10.3]{PS1}).

 First let us suppose that $\R$ is a pro-Artinian commutative ring
(see Section~\ref{pro-artinian}).
 By the definition, an \emph{$\R$\+comodule} is an ind-object in
the abelian category opposite to the category of discrete
$\R$\+modules of finite length~\cite[Section~1.4]{Pweak}.
 There is a natural contravariant functor $\M\longmapsto\M^\rop$
assigning to every $\R$\+comodule a pro-object in the category
of discrete $\R$\+modules of finite length.
 Furthermore, there is a distinguished object $\C=\C(\R)$ in
the category $\R\comodl$ of $\R$\+comodules such that
$\C^\rop=\R$; the functor $\M\longmapsto\M^\rop$, viewed as
a contravariant functor from $\R\comodl$ to the category of
abelian groups, is represented by $\C(\R)$.
 A \emph{cofree} $\R$\+comodule is a direct sum of copies of
the $\R$\+comodule~$\C$; the cofree $\R$\+comodules are injective,
and any $\R$\+comodule can be embedded into a cofree one.

 According to Matlis' duality
(see~\cite[Corollary~4.3]{Mat1} or~\cite[Theorem~18.6]{Mats}),
choosing an injective hull of the irreducible module over
an Artinian commutative local ring $R$ fixes an anti-equivalence
of the category of $R$\+modules of finite length with itself.
 Passing to the inductive limit of such auto-anti-equivalences over
all the discrete quotient rings of a pro-Artinian commutative
ring~$\R$, one obtains an auto-anti-equivalence of the category
of discrete $\R$\+modules of finite length depending on
the choice of a ``minimal injective cogenerator'' of the abelian
category $\R\discr$, i.~e., an injective hull of the direct sum
of the irreducible discrete $\R$\+modules.
 Hence choosing such an injective object $\cE\in\R\discr$
identifies the category of discrete $\R$\+modules with the category
of $\R$\+comodules; the equivalence of categories $\R\comodl\simeq
\R\discr$ takes the object $\C\in\R\comodl$ to the object~$\cE$
\cite[Section~1.9]{Pweak}.

 For any pro-Artinian commutative ring $\R$, the categories of
injective $\R$\+comodules and projective $\R$\+contramodules are
naturally equivalent, $\R\comodl_\inj\simeq\R\contra_\proj$.
 The equivalence is provided by the functor $\Psi_\R\:\M\longmapsto
\Hom_\R(\C,\M)$ of $\R$\+comodule homomorphisms from
the $\R$\+comodule $\C(\R)$ and the functor $\Phi_\R\:\P\longmapsto
\C\ocn_\R\P$ of \emph{contratensor product} of $\R$\+contramodules
with the $\R$\+comodule $\C(\R)$.
 It assigns the free $\R$\+contramodule $\R[[X]]$ to the cofree
$\R$\+comodule $\bigoplus_X\C(\R)$ for any set~$X$.
 This result can be found in~\cite[Section~1.5]{Pweak}.

\bigskip

 More generally, let $\R$ be a \emph{right pseudo-compact}
topological ring, i.~e., a complete, separated topological ring
with a base of neighborhoods of zero formed by open right
ideals~$\J$ for which the right $\R$\+modules $\R/\J$ have
finite length~\cite[\SS\,IV.3]{Gab}.
 We define \emph{left\/ $\R$\+comodules} as the ind-objects in
the abelian category opposite to the category of discrete right
$\R$\+modules of finite length.
 The category of left $\R$\+comodules $\R\comodl$ is anti-equivalent
to the category of \emph{pseudo-compact right\/ $\R$\+modules},
i.~e., pro-objects in the category of discrete right $\R$\+modules
of finite length or, which is the same, complete and separated
topological right $\R$\+modules with a base of neighborhoods of
zero formed by open submodules with discrete quotient modules of
finite length.
 As above, we denote this anti-equivalence by $\M\longmapsto\M^\rop$.

 There is a distinguished left $\R$\+comodule $\C=\C(\R)$ for which
$\C^\rop=\R$; the functor $\M\longmapsto\M^\rop$, viewed as
a contravariant functor from $\R\comodl$ to the category of abelian
groups, is represented by $\C(\R)$.
 A \emph{cofree} left $\R$\+comodule is a direct sum of copies of
the $\R$\+comodule~$\C$; the cofree $\R$\+comodules are injective,
and any left $\R$\+comodule can be embedded into a cofree one.
 The abelian category $\R\comodl$ is locally
finite~\cite[\SS\,II.4]{Gab}; and the choice of an injective
cogenerator $\cE$ in any locally finite abelian category $\sA$ fixes
an equivalence between $\sA$ and the category of left comodules
over the topological ring $\R=\End_\sA(\cE)^\rop$ opposite to the ring
of endomorphisms of the object $\cE\in\sA$.
 The topology on the ring $\R$ is defined to have a base of
neighborhoods of zero consisting of (the right ideals opposite to)
the annihilators of subobjects of finite length $\L\subset\cE$.
 The equivalence of categories $\sA\simeq\R\comodl$ assigns
the object $\C\in\R\comodl$ to the object~$\cE$
\cite[Corollaire~VI.4.1]{Gab}.

 For any right pseudo-compact topological ring~$\R$, the categories
of injective left $\R$\+comodules and projective left
$\R$\+contramodules are naturally equivalent,
$$
 \R\comodl_\inj\simeq\R\contra_\proj.
$$
 Indeed, the injective left $\R$\+comodules are the direct summands
of the cofree $\R$\+co\-modules $\bigoplus_X\C(R)$, and the projective
left $\R$\+contramodules are the direct summands of the free
$\R$\+contramodules $\R[[X]]$, where $X$ denotes arbitrary sets.
 It remains to compute the groups of morphisms
between the cofree $\R$\+comodules and the free $\R$\+contramodules
in terms of projective limits over the subcomodules of finite
length $\L\subset\C$ and the open right ideals $\J\subset\R$,
\begin{multline*}\textstyle
 \Hom_\R\big(\bigoplus_X\C\;\bigoplus_Y\C\big)\simeq
 \prod_X\Hom_\R\big(\C\;\bigoplus_Y\C\big)\simeq
 \prod_X\varprojlim_{\L\subset\C}\Hom_\R\big(\L\;\bigoplus_Y\C\big) \\
 \textstyle\simeq
 \prod_X\varprojlim_{\L\subset\C}\bigoplus_Y\Hom_\R(\L,\C)\simeq
 \prod_X\varprojlim_{\J\subset\R}\bigoplus_Y\R/\J 
 =\prod_X\varprojlim_{\J\subset\R}\R/\J[Y] \\
 \textstyle\simeq
 \prod_X\R[[Y]]\simeq\prod_X\Hom^\R\big(\R,\R[[Y]]\big)\simeq
 \Hom^\R\big(\R[[X]],\R[[Y]]\big),
\end{multline*}
in order obtain an isomorphism of the categories they form.
 One also has to check that these isomorphisms agree with
the compositions of morphisms in the two categories.

 In other words, we can conclude that the additive categories of
projective left $\R$\+contramodules and projective pseudo-compact right
$\R$\+modules are \emph{anti-equivalent} (cf.~\cite[the end
of Remark~A.3]{Psemi}).

\bigskip

 One would like to generalize this equivalence from locally finite to
locally Noetherian abelian categories, i.~e., abelian categories
satisfying the axiom Ab5 and admitting a set of generators consisting
of Noetherian objects, or equivalently, Ab5-categories where every
object is the union of its Noetherian subobjects and isomorphism
classes of Noetherian objects form
a set~\cite[\SS\,II.4]{Gab} (cf.~\cite{Kra}).

\begin{rem}
 A remark at the end of~\cite[\SS\,IV.3]{Gab} suggests considering
topological rings $\R$ with a base of neighborhoods of zero formed
by (say, right) ideals $\J$ such that the quotient modules $\R/\J$
are Artinian, and topological right $\R$\+modules with a base of
neighborhoods of zero formed by open $\R$\+submodules with Artinian
quotient modules.
 Then the opposite abelian category $\sE(\R)$ to such category of
$\R$\+modules is locally Noetherian, and the object $\C$ opposite to
the right $\R$\+module $\R$ is injective in it.

 However, the direct summands of direct sums of copies of the object
$\C$ do \emph{not} exhaust the class of injective objects in
$\sE(\R)$, as one can see already in the example of the topological
ring $\R=\Z_l$ with the Artinian discrete right $\R$\+module
$\mathbb Q_l/\Z_l$, which admits no surjective continuous
morphisms from topological products of copies of the right
$\R$\+module~$\R$.
 Furthermore, given a locally Noetherian abelian category $\sA$,
choosing an injective object $\cE\in\sA$ such that all
the injectives in $\sA$ are direct summands of the direct sums
of copies of $\cE$ leads to a topological ring
$\R=\End_\sA(\cE)^\rop$ which does \emph{not} satisfy the above
Artinianness condition in general.
 
 Indeed, it suffices to take $\sA=\Ab$ and $\cE=\mathbb R/\Z$;
then the Noetherian object $\Z\in\sA$ is embeddable into $\cE$
by means of any irrational number in $\mathbb R/\Z$, hence
the discrete right $\R$\+module $\Hom_\Ab(\Z,\mathbb R/\Z)=
\mathbb R/\Z$ is the quotient module of
$\R=\End_\Ab(\mathbb R/\Z)^\rop$ by a certain open
right ideal, and it is \emph{not} an Artinian module.
 \emph{Nor} does the functor assigning the $\R$\+module
$\Hom_\Ab(\Z,\mathbb R/\Z)$ to the abelian group $\Z$ appear
anywhere close to being an anti-equivalence of abelian categories.
 That is why we choose a different path below.
\end{rem}

 Let us start from an arbitrary locally Noetherian abelian category~$\sA$.
 Recall that, being a Grothendieck abelian category (an Ab5-category
with a set of generators), any locally Noetherian category has
enough injectives~\cite[N$^{\mathrm o}$~1.10]{GrToh}.

\begin{thm}
 For any locally Noetherian abelian category\/ $\sA$ there exists 
a unique abelian category\/ $\sB$ with enough projectives such
that the full additive subcategories of injective objects in\/ $\sA$
and projective objects in\/ $\sB$ are (covariantly) equivalent,
$$
 \sA_\inj\simeq\sB_\proj.
$$
 All the infinite direct sums and products exist in the abelian
categories\/ $\sA$ and\/ $\sB$, and both the subcategories of injective
objects in\/ $\sA$ and projective objects in\/ $\sB$ are closed under
both the infinite direct sums and the infinite products.
\end{thm}

\begin{proof}
 To prove uniqueness, we notice that an abelian category with enough
projectives is determined by its full additive subcategory of
projective objects.
 Indeed, given an abelian category $\sB'$ with enough projectives
and an abelian category $\sB''$, any additive functor $\sB'_\proj
\rarrow\sB''$ can be uniquely extended to a right exact functor
$\sB'\rarrow\sB''$.
 In particular, let $\sB'$ and $\sB''$ be two abelian categories with
equivalent full subcategories of projectives $\sB'_\proj\simeq
\sB''_\proj$.
 Then the embedding functor $\sB'_\proj\rarrow\sB''$ extends uniquely
to a right exact functor $\sB'\rarrow\sB''$, while the embedding
functor $\sB''_\proj\rarrow\sB'$ extends uniquely to a right exact
functor $\sB''\rarrow\sB'$.
 The compositions $\sB'\rarrow\sB''\rarrow\sB'$ and $\sB''\rarrow
\sB'\rarrow\sB''$ are right exact functors isomorphic to
the identity functors on the full subcategories of projective
objects, so they are also naturally isomorphic to identity functors
on the whole abelian categories $\sB'$ and~$\sB''$.

 Now, given a locally Noetherian abelian category $\sA$, choose
an injective object $\cE\in\sA$ such that all the injectives in
$\sA$ are directs summands of the direct sums of copies of~$\cE$.
 E.~g., one can take $\cE$ to be the direct sum of injective
envelopes of all the quotient objects of Noetherian generators
of~$\sA$.
 Consider the topological ring $\R=\End_\sA(\cE)^\rop$ with a base
of neighborhoods of zero formed by the right ideals opposite to
the annihilators of Noetherian submodules $\L\subset\cE$
in $\End_\sA(\cE)$.
 Set $\sB=\R\contra$ to be the abelian category of left
$\R$\+contramodules.

 To identify the full subcategory of direct sums of copies of
the object $\cE$ in $\sA$ with the full subcategory of free
$\R$\+contramodules in $\sB$, one computes the Hom groups
\begin{multline*}\textstyle
 \Hom_\sA\big(\bigoplus_X\cE\;\bigoplus_Y\cE\big)\simeq
 \prod_X\Hom_\sA\big(\cE\;\bigoplus_Y\cE\big)\simeq
 \prod_X\varprojlim_{\L\subset\cE}
 \Hom_\sA\big(\L\;\bigoplus_Y\cE\big) \\
 \textstyle\simeq
 \prod_X\varprojlim_{\L\subset\cE}\bigoplus_Y\Hom_\sA(\L,\cE)\simeq
 \prod_X\varprojlim_{\J\subset\R}\bigoplus_Y\R/\J \\
 \textstyle\simeq
 \prod_X\R[[Y]]\simeq\prod_X\Hom^\R\big(\R,\R[[Y]]\big)\simeq
 \Hom^\R\big(\R[[X]],\R[[Y]]\big)
\end{multline*}
in both subcategories in terms of projective limits over
the Noetherian subobjects $\L\subset\cE$ and the open right
ideals $\J\subset\R$.
 Once again, one has to check that these isomorphisms agree with
the compositions of morphisms in the two categories.
 Adjoining the direct summands (the images of idempotent endomorphisms)
to both the full subcategories, one obtains the desired equivalence
between the full subcategories of injective objects in $\sA$ and
projective objects in~$\sB$.

 Any abelian category with infinite direct sums and a set of generators
has infinite products by Freyd's Special Adjoint Functor existence
Theorem~\cite[Corollary~V.8]{McL}; and the category of left contramodules
over a topological ring $\R$ with a base of neighborhoods of zero
consisting of open right ideals always has both the infinite direct
sums and products, as we have seen in Section~\ref{over-topol-rings}.
 In any abelian category the infinite direct sums of projective
objects are projective and the infinite products of injective objects
are injective.

 The infinite direct sums of injective objects in a locally Noetherian
abelian category $\sA$ are injective~\cite[Corollaire~II.4.1 and
Proposition~IV.2.6]{Gab}.
 Finally, to prove that the infinite products of projective objects in
our category $\sB$ are projective, we notice that any
family of objects of the full subcategory $\sB_\proj\subset\sB$
has an infinite product in $\sB_\proj$ (since any family of
objects of $\sA_\inj$ has an infinite product in $\sA_\inj$).
 It is claimed that whenever an abelian category $\sB$ has enough
projectives and the full subcategory of projectives $\sB_\proj\subset\sB$
has infinite products, these are also the infinite products of
objects of $\sB_\proj$ in the whole category $\sB$, i.~e., the embedding
functor $\sB_\proj\rarrow\sB$ preserves infinite products.

 Indeed, let an object $X\in\sB$ be presented as the cokernel of
a morphism of projective objects $Q\rarrow P$, so that there is
an initial fragment of projective resolution $Q\rarrow P\rarrow X
\rarrow0$ in~$\sB$.
 Let $F_\alpha$ be a family of objects in $\sB_\proj$ and
$F=\prod^{\sB_\proj}_\alpha F_\alpha$ be their product in $\sB_\proj$.
 Then one computes the group $\Hom_\sB(X,F)$ as the kernel of
the map of abelian groups $\Hom_\sB(P,F)\rarrow\Hom_\sB(Q,F)$, which
is isomorphic to the product of the kernels of the maps
$\Hom_\sB(P,F_\alpha)\rarrow\Hom_\sB(Q,F_\alpha)$, that is the group
$\prod_\alpha\Hom_\sB(X,F_\alpha)$.
\end{proof}

 A further discussion of the correspondence between the abelian
categories $\sA$ and $\sB$ described in the theorem can be found
in~\cite[Examples~6.3\+-6.5]{PS2} and~\cite[Section~10.2]{PS1}.

\subsection{$\Add(M)$ and projective contramodules} \label{addm}
 The following results and constructions from
the papers~\cite{PS1,PS2,Pper} provide a series of far-reaching
generalizations of the theorem and proof from the previous section.

 Let $\sA$ be an additive category with infinite direct sums
and $M\in\sA$ be an object.
 Then we denote by $\Add(M)\subset\sA$ the full subcategory in $\sA$
consisting of all the direct summands of (infinite) direct sums
of copies of the object~$M$.
 
 We start with the case of a module over an associative ring.

\begin{prop1}
 Let $R$ be an associative ring and $M$ be a left $R$\+module.
 Then there is a complete and separated topological ring\/ $\gS$
with a base of neighborhoods of zero formed by open right ideals
such that the full subcategory\/ $\Add(M)\subset R\modl$ is equivalent
to the category of projective left\/ $\gS$\+contramodules,
$\Add(M)\simeq\gS\contra_\proj$.
\end{prop1}

\begin{proof}
 When the $R$\+module $M$ is finitely generated, the category
$\Add(M)$ is equivalent to the category of projective left
$S$\+modules, $\Add(M)\simeq S\modl$, for the discrete ring $S$ of
endomorphisms of the $R$\+module~$M$.
 This is a classical result (see~\cite[Remark~7.2]{PS1} for
references and a further discussion).

 To be more precise, in the general case $\gS=\Hom_R(M,M)^\rop$ is
the opposite ring to the ring of endomorphisms of the $R$\+module $M$;
so the ring $\gS$ acts in $M$ on the right, making $M$
an $R$\+$\gS$\+bimodule.
 The topology on $\gS$ is defined by the condition that $M$ should be
a discrete right $\gS$\+module; specifically, the annihilators of
finitely generated $R$\+submodules $L\subset M$ form a base of
neighborhoods of zero in~$\gS$.
 We refer to~\cite[Theorem~7.1]{PS1} for the details.
\end{proof}

 The next proposition provides a generalization to the categorical
context.

\begin{prop2}
 Let\/ $\sA$ be a locally finitely presentable, or more generally,
a locally finitely generated abelian category in the sense of
the book~\cite[Sections~1.A and~1.E]{AR}.
 Let $M\in\sA$ be an object.
 Then there is a complete and separated topological ring\/ $\R$
with a base of neighborhoods of zero formed by open right ideals
such that the full subcategory\/ $\Add(M)\subset\sA$ is equivalent
to the category of projective left\/ $\R$\+contramodules,
$\Add(M)\simeq\R\contra_\proj$.
\end{prop2}

\begin{proof}
 As above, put $\R=\Hom_\sA(M,M)^\rop$; so $\R$ is opposite ring to
the ring of endomorphisms of $M$, or in other words, the universal ring
acting in the object $M\in\sA$ on the right.
 The annihilators of finitely generated subobjects $L\subset M$ form
a base of the topology on~$\R$.
 The proof is the same as in Proposition~1.
\end{proof}

 Even more generally, let $\sA$ be a \emph{locally weakly finitely
generated} abelian category with a generator, in the sense of
the paper~\cite[Section~9.2]{PS1} or which is the same,
a \emph{nearly locally finitely presentable} abelian category in
the sense of the paper~\cite{PR2}.
 Then, for any object $M\in\sA$, the full subcategory
$\Add(M)\subset\sA$ is equivalent to the category of projective
left contramodules $\R\contra_\proj$ over the topological ring
$\R=\Hom_\sA(M,M)^\rop$ with a base of the topology formed by
the annihilators of \emph{weakly finitely generated}
(\,$=$~nearly finitely presentable) subobjects $L\subset M$.
 This is the result of~\cite[Theorem~9.9]{PS1}.

 A further generalization to \emph{additive categories with closed
functors} discussed in~\cite[Section~9.3]{PS1} makes the above
assertions applicable to the abelian categories $\sA=\C\comodl$ and
$\sA=\S\simodl$ of comodules over corings and semimodules over
semialgebras~\cite[Section~10.3]{PS1}.
 These results allow to interpret the abelian categories of
left $\C$\+contramodules and left $\S$\+semicontramodules
$\C\contra$ and $\S\sicntr$ as the categories of contramodules
over appropriately constructed topological rings $\R$, as mentioned
above in Sections~\ref{over-corings} and~\ref{over-semialgebras}.

 The following proposition is even more general.

\begin{prop3}
 Let\/ $\sA$ be an idempotent-complete additive category with infinite
direct sums, and let $M\in\sA$ be an object.
 Then there exists a unique abelian category\/ $\sB$ with enough
projective objects such that the full subcategory\/ $\Add(M)\subset\sA$
is equivalent to the full subcategory of projective objects in $\sB$,
that is\/ $\Add(M)\simeq\sB_\proj$.
\end{prop3}

\begin{proof}
 The uniqueness was explained in the proof of
Theorem~\ref{underived-top-rings}.
 Concerning existence, there are two constructions of the category
$\sB$ suggested in~\cite[proof of Theorem~1.1(a)]{PS2}.
 Fistly, one can construct $\sB$ as the category of
\emph{finitely presented} (\emph{coherent}) \emph{functors}
$\Add(M)^\sop\rarrow\Ab$.
 Since the category $\Add(M)$ has weak kernels, the category of
coherent functors on it is abelian.

 Secondly, $\sB$ is the category of algebras/modules over the additive
monad $X\longmapsto\Hom_\sA(M,M^{(X)})$ on the category of
sets (see also~\cite[Section~6]{PS1} and~\cite[Section~1]{Pper}).
 The advantage of this point of view is that it specializes to
the above descriptions of $\sB$ as the abelian category of contramodules
over a topological ring, under appropriate assumptions.
\end{proof}

 An approach to the \emph{Enochs conjecture on covers and direct limits}
based on the results described above in this section was suggested in
the papers~\cite{Pproperf,BP3,BPS}.
 An application of these results to direct limits of classes of modules
was developed in~\cite{PT}.

\subsection{Fully faithful contramodule forgetful functors}
\label{fully-faithful}
 Let $\R$ be a complete and separated topological ring with a base of
neighborhoods of zero formed by open right ideals.
 Let $R$ be an associative ring, and let $\rho\:R\rarrow\R$ be a ring
homomorphism.
 Consider the composition of the natural forgetful functor
$\R\contra\rarrow\R\modl$ with the functor of restriction of
scalars $\R\modl\rarrow R\modl$.
 When is the forgetful functor $\R\contra\rarrow R\modl$ fully faithful?

 We have seen in Section~\ref{recovering} that the forgetful functors
$k[[z]]\contra\rarrow k[z]\modl$ and $\Z_l\contra\rarrow\Z\modl$ are
fully faithful.
 Moreover, for a right Noetherian associative ring $R$ and its
completion $\R=\varprojlim_n R/I^n$ in the adic topology of
a centrally generated ideal $I\subset R$, 
Theorem~\ref{over-adic-completions} stated that the forgetful functor
$\R\contra\rarrow R\modl$ is fully faithful and explicitly described
its essential image.
 Far-reaching generalizations of these relatively basic observations
were obtained in the papers~\cite[Theorem~2.1]{Psm},
\cite[Section~3]{Pper}, and~\cite[Section~6]{Pcoun}.

 Let us start with the case of a coalgebra $\C$ \cite[Section~2.2]{Psm}.
 A coalgebra $\C$ over a field~$k$ is said to be \emph{coaugmented}
if it is endowed with a coalgebra morphism (coagmentation)
$\gamma\:k\rarrow\C$.
 Over a coaugmented coalgebra $\C$, the one-dimensional vector
space~$k$ (as well as any other $k$\+vector space) carries
the so-called \emph{trivial} left $\C$\+comodule structure, which is
defined in terms of the coagmentation.
 The \emph{cohomology} of a coaugmented coalgebra $\C$ is
the graded vector space (or the graded algebra) of Yoneda Ext in
the category of left $\C$\+comodules, $H^*(\C)=\Ext_\C^*(k,k)$.
 In particular, one has $H^0(\C)=k$, and the space $H^1(\C)=
\Ext^1_\C(k,k)$ can be computed as the kernel of the comultiplication
map $\mu\:\C_+\rarrow\C_+\ot\C_+$, where $\C_+=\C/\gamma(k)$.

 A coaugmented coalgebra $\C$ is said to be \emph{conilpotent}
(see, e.~g., \cite[Sections~5.5 and~6.4]{Pkoszul}) if for every element
$c\in\C$ there exists an integer $m\ge0$ such that the element~$c$ is annihilated by the composition $\C\rarrow\C^{\ot m+1}\rarrow
\C_+^{\ot m+1}$ of the iterated comultiplication map
$\mu^{(m)}\:\C\rarrow\C^{\ot m+1}$ with the natural surjection
$\C^{\ot m+1}\rarrow\C_+^{\ot m+1}$.
 The vector space $H^1(\C)=\ker(\C_+\to C_+^{\ot2})$ is interpreted as
the vector space of \emph{cogenerators} of a conilpotent coalgebra $\C$
(see~\cite[Section~5]{Pqf} for a discussion of cogenerators and
corelations of conilpotent coalgebras).
 A conilpotent coalgebra $\C$ is said to be \emph{finitely cogenerated}
if the $k$\+vector space $H^1(\C)$ is finite-dimensional.

 Following the discussion in Section~\ref{over-top-algebras}, for any
coalgebra $\C$ over a field~$k$, a left $\C$\+contramodule is the same
thing as a left contramodule over the pro-finite-dimensional
topological algebra $\R=\C^*$.
 Hence we have a natural forgetful functor $\C\contra\rarrow\C^*\modl$
\cite[Section~A.1.2]{Psemi}, \cite[Section~2.1]{Psm}.
 Given a left $\C$\+contramodule $\P$, the left action of
the $k$\+algebra $\C^*$ in $\P$ is constructed as the composition
$\C^*\ot_k\P\rarrow\Hom_k(\C,\P)\rarrow\P$ of the natural embedding
$\C^*\ot_k\P\rarrow\Hom_k(\C,\P)$ with the $\C$\+contraaction map
$\pi_\P\:\Hom_k(\C,\P)\rarrow\P$.
 As above, given an associative ring homomorphism $\rho\:R\rarrow\C^*$,
we obtain a forgetful functor $\C\contra\rarrow R\modl$.

\begin{thm1}
 Let\/ $\C$ be a finitely cogenerated conilpotent coalgebra over~$k$.
 Then the forgetful functor\/ $\C\contra\rarrow\C^*\modl$ is fully
faithful.
 Moreover, for any dense subring $R\subset\C^*$ in
the pro-finite-dimensional topology of the $k$\+algebra\/ $\C^*$,
the forgetful functor\/ $\C\contra\rarrow R\modl$ is fully faithful.
\end{thm1}

\begin{proof}
 This is~\cite[Theorem~2.1]{Psm}.
 The argument is based on the contramodule Nakayama lemma for
coalgebras over fields~\cite[Lemma~A.2.1]{Psemi}.
\end{proof}

\begin{ex1}
 The following example is a particular case
of~\cite[Examples~3.3]{Pper}.
 Let $\R=k\{\{z_1,\dotsc,z_m\}\}$ be the algebra of noncommutative
formal Taylor power series in a finite set of variables $z_1$,~\dots,
$z_m$ over a field~$k$.
 We consider $\R$ as a topological ring in the formal power series
topology (or, in other words, the adic topology for the ideal
generated by the variables).
 Then $\R\simeq\C^*$ is the dual pro-finite-dimensional algebra to
the \emph{cofree conilpotent coalgebra} $\C$ with $m$~cogenerators
$z_1^*$,~\dots,~$z_m^*$.
 Let $R=k\{z_1,\dotsc,z_m\}$ be the $k$\+algebra of noncommutative
polynomials in $z_1$,~\dots, $z_m$, and let $\rho\:R\rarrow\R$ be
the natural embedding.
 According to Theorem~1, the forgetful functor $\R\contra\rarrow
R\modl$ is fully faithful.
\end{ex1}

 Now we turn to a discussion of topological rings $\R$ with
a \emph{countable} base of neighborhoods of zero.
 Let $\R$ be a topological ring and $\J\subset\R$ be a right ideal.
 One says that a finite set of elements $s_1$, \dots, $s_m\in\J$
\emph{strongly generates} $\J$ if for every family of elements
$r_x\in\J$, indexed by a set $X$ and converging to zero in
the topology of $\R$, there exist $m$~families of elements
$t_{j,x}\in\R$, \ $j=1$,~\dots,~$m$, each of them indexed by
the set $X$ and converging to zero in the topology of $\R$, such that
$r_x=\sum_{j=1}^m s_jt_{j,x}$ for every $x\in X$.
 Since any element $r\in\J$ can be viewed as a family of elements
indexed by a one-point set $X$, and any finite family of elements in
$\R$ converges to zero in $\R$ by definition, any finite family of
elements~$s_j$ strongly generating a right ideal $\J$ also generates
$\J$ in the conventional sense~\cite[Section~3]{Pper}.

 The following proposition is a generalization of Theorem~1 to
topological rings of much more general nature than
the pro-finite-dimensional algebras over fields.

\begin{prop}
 Let\/ $\R$ be a complete and separated topological associative ring
and $R\subset\R$ be a dense subring.
 Assume that\/ $\R$ has a countable base of neighborhoods of zero
consisting of open two-sided ideals\/ $\J$, each of which, viewed as
a right ideal, is strongly generated by a finite set of elements
belonging to $R\cap\J$.
 Then the forgetful functor\/ $\R\contra\rarrow R\modl$ is fully
faithful.
\end{prop}

\begin{proof}
 This is~\cite[Theorem~3.1]{Pper}.
 The argument is based on a suitable version of the contramodule
Nakayama lemma for topological rings (\cite[Lemma~D.1.2]{Pcosh}
or~\cite[Lemma~6.14]{PR}; see Lemma~1 below).
\end{proof}

 An even more general result can be found in~\cite[Section~6]{Pcoun}.
 Let $\R$ be a complete and separated topological ring with a base
of neighborhoods of zero formed by open right ideals, and let
$\rho\:R\rarrow\R$ be an associative ring homomorphism with a dense
image.
 Then the full preimages of open right ideals in $\R$ under~$\rho$
form a base of a topology on $R$, making $R$ a topological ring.
 The ring $\R$ is the completion of the ring $R$ in this topology.
 Furthermore, the assignment $\I\longmapsto I=\rho^{-1}(I)$ defines
a bijection between open right ideals $\I\subset\R$ and open right
ideals $I\subset R$ \cite[Section~4]{Pcoun}.

 Generalizing the previous definition, we say that an open right
ideal $I\subset R$ is \emph{strongly generated} (or, in other words,
the corresponding open right ideal $\I\subset\R$ is \emph{strongly
generated by elements coming from~$R$}) if, for any set $X$ and any
$X$\+indexed family of elements $r_x\in\I$ converging to zero in
the topology of $\R$, there exists a finite set of elements
$s_1$,~\dots, $s_m\in I$ and $m$~families of elements
$t_{j,x}\in\R$, \ $j=1$,~\dots,~$m$, each of them indexed by
the set $X$ and converging to zero in the topology of $\R$,
such that $r_x=\sum_{j=1}^m\rho(s_j)t_{j,x}$ for every $x\in X$.
 When the topological ring $\R$ has a countable base of neighborhoods
of zero, it suffices to check this condition for families of elements
$r_x$ belonging to the dense subring $\rho(R)\subset\R$
\cite[Lemma~6.4]{Pcoun}.

 Moreover, as explained in~\cite[Section~6]{Pcoun}, the property of
an open right ideal $\I\subset\R$ to be strongly generated by elements
coming from $R$ can be expressed consizely by the equation
$\I[[X]]=I\R[[X]]$.
 Here $\I[[X]]\subset\R[[X]]$ is the subgroup of all $X$\+indexed
zero-convergent families of elements from $\I$ in the left
$\R$\+module of all $X$\+indexed zero-convergent families of
elements from~$\R$.
 According to~\cite[Remark~6.5]{Pcoun}, an open right ideal $I\subset R$
is strongly generated by a finite set of its elements (in the sense of
our previous definition) whenever it is strongly generated \emph{and}
is finitely generated as a right ideal in an abstract associative
ring~$R$.

\begin{thm2}
 Let\/ $\R$ be a complete and separated topological ring with
a countable base of neighborhoods of zero consisting of open right
ideals, and let $\rho\:R\rarrow\R$ be a ring homomorphism with
a dense image.
 Then the forgetful functor\/ $\R\contra\rarrow R\modl$ is fully
faithful \emph{if and only if} all the open right ideals\/ $\I\subset\R$
are strongly generated by elements coming from~$R$.
 It suffices to check the latter condition for any chosen base of
neighborhoods of zero in\/ $\R$ consisting of open right ideals.
\end{thm2}

\begin{proof}
 This is~\cite[Theorem~6.2]{Pcoun}.
 Once again, the proof of the ``if'' part is based on the contramodule
Nakayama lemma, which is formulated immediately below.
\end{proof}

\begin{lem1}
 Let\/ $\R$ be a complete and separated topological ring with
a countable base of neighborhoods of zero consisting of open right
ideals, and let\/ $\P$ be a left\/ $\R$\+contramodule.
 Assume that one has\/ $\I\tim\P=\P$ for every open right ideal\/
$\I\subset\R$ (see Section~\ref{flat-contra} for the notation).
 Then\/ $\P=0$.
\end{lem1}

\begin{proof}
 This is~\cite[Lemma~6.14]{PR} (cf.\ Lemma~\ref{over-topol-rings}).
\end{proof}

 In the rest of this section we discuss the full-and-faithfulness of
certain forgetful functors originating from the categories of
semicontramodules $\S\sicntr$ over semialgebras $\S$ over finitely
cogenerated conilpotent coalgebras~$\C$ (see
Section~\ref{over-semialgebras}).

 Let us briefly recall a construction of semialgebras
from~\cite[Section~10.2]{Psemi} that was already used in
Sections~\ref{category-o-contra}\+-\ref{tate-harish-chandra} above.
 Let $k$ be a field, $\C$ be a coalgebra over $k$, and $K$ be
an associative algebra over~$k$.
 Let $K\rarrow\C^*$ be a homomorphism of $k$\+algebras whose image
is dense in the pro-finite-dimensional topology of~$\C^*$.
 Then the related pairing $\phi\:\C\ot_kK\rarrow k$ is nondegenerate
in~$\C$.
 Following~\cite[Section~10.1.4]{Psemi}, the composition $\comodr\C
\rarrow\modr\C^*\rarrow\modr K$ of the natural functor $\comodr\C\rarrow
\modr\C^*$ (see Section~\ref{over-l-adics}) with the functor of
restriction of scalars $\modr\C^*\rarrow\modr K$ is fully faithful.
 The same applies to the similar composition of functors $\C\comodl
\rarrow\C^*\modl\rarrow K\modl$ between the categories of
left (co)modules.

 Let $R$ be another associative algebra over~$k$ and $f\:K\rarrow R$
be a $k$\+algebra homomorphism such that $R$ is a flat left
$K$\+module in the (bi)module structure induced by~$f$.
 Set $\S=\C\ot_KR$, where the right $K$\+module structure on $\C$
is provided by the above functor $\comodr\C\rarrow\modr K$.
 Then $\S$ is naturally an injective left $\C$\+comodule and
a right $R$\+module.
 Assume that the underlying right $K$\+module structure on $\S$
originates from a right $\C$\+comodule structure (i.~e., the right
$K$\+module $\S$ belongs to the essential image of the fully
faithful functor $\comodr\C\rarrow\modr K$).
 Then the $\C$\+$\C$\+bicomodule $\S$ has a natural structure of
a semialgebra over $\C$ with the semiunit map $\be\:\C\rarrow\S$ induced
by the morphism $f\:K\rarrow R$ and the semimultiplication map
$\bm\:\S\oc_\C\S\rarrow\S$ induced by the multiplication map
$R\ot_KR\rarrow R$ \cite[Section~10.2.1]{Psemi}.

\begin{lem2}
 Assume that $R$ is a projective left $K$\+module.
 Then the abelian category of left\/ $\S$\+semicontramodules\/
$\S\sicntr$ is isomorphic to the category of $k$\+vector spaces\/ $\bP$
endowed with the structures of a left\/ $\C$\+contramodule and a left
$R$\+module satisfying the following two compatibility conditions:
firstly, the two underlying left $K$\+module structures should coincide,
and secondly, the $R$\+action map\/ $\bP\rarrow\Hom_K(R,\bP)$ should be
a morphism of left\/ $\C$\+contramodules.

 Here the left\/ $\C$\+contramodule structure on the $k$\+vector space\/
$\Hom_K(R,\bP)$ is provided by the natural isomorphism\/
$\Hom_K(R,\P)\simeq\Cohom_\C(\S,\P)$, which holds for any left\/
$\C$\+contramodule\/~$\P$.

 In particular, there is a natural exact, faithful forgetful functor\/
$\S\sicntr\rarrow R\modl$ forming a commutative square diagram with
the forgetful functors\/ $\S\sicntr\rarrow\C\contra\rarrow K\modl$
and $R\modl\rarrow K\modl$.
\end{lem2}

\begin{proof}
 This is explained in~\cite[Section~10.2.2]{Psemi}.
\end{proof}

\begin{cor}
 Let\/ $\C$ be a coalgebra over a field~$k$, let $K\rarrow R$ be
a morphism of associative algebras over~$k$ making $R$ a projective left
$K$\+module, and let $K\rarrow\C^*$ be a morphism of algebras with
a dense image.
 Assume that the right $K$\+module structure on\/ $\S=\C\ot_KR$ comes
from a right\/ $\C$\+comodule structure, so\/ $\S$ is a semialgebra
over\/ $\C$; and assume further that\/ $\C$ is a finitely cogenerated
conilpotent coalgebra.
 Then the forgetful functor\/ $\S\sicntr\rarrow R\modl$ is fully
faithful, and its essential image consists of all the left
$R$\+modules whose underlying left $K$\+module structure comes from
a left\/ $\C$\+contramodule structure.
\end{cor}

\begin{proof}
 By Theorem~1, the forgetful functor $\C\contra\rarrow K\modl$ is
fully faithful.
 Hence the first assertion of the corollary follows immediately from
the first assertion of Lemma~2.
 To deduce the second assertion of the corollary, it remains to observe
that, for any left $R$\+module $P$, the action map $P\rarrow\Hom_k(R,P)$
is a left $R$\+module morphism, hence also a left $K$\+module morphism,
and use Theorem~1 again.
\end{proof}

\begin{ex2}
 Let $\g$ be a Tate (locally linearly compact) Lie algebra
(see Section~\ref{over-Lie}), and let $\h\subset\g$ be a compact open
subalgebra.
 Assume that the topological Lie algebra~$\h$ is pro-nilpotent and
the discrete $\h$\+module $\g/\h$ is nilpotent.
 Let $\C$ be the coassociative coalgebra related to~$\h$ and
$(\g,\C)$ be the Tate Harish-Chandra pair related to $\g$ and $\h$
(see Section~\ref{tate-harish-chandra} and~\cite[Section~D.6]{Psemi}).
 Assume that the coalgebra $\C$ is finitely cogenerated, or
equivalently, the Lie algebra $\h$ is topologically finitely generated
(see~\cite[Section~D.6.1]{Psemi} for the isomorphism of cohomology
of the Lie coalgebra $\L=\h\dual$ and the coassociative coalgebra~$\C$).

 Let $\bar\g\subset\g$ be a dense Lie subalgebra; then $\bar\h=
\h\cap\bar\g$ is a dense Lie subalgebra in~$\h$.
 Consider the two enveloping algebras $K=U(\bar\h)$ and $R=U(\bar\g)$.
 Then the associative algebra morphism $K\rarrow R$ induced by
the embedding $\bar\h\rarrow\bar\g$ and the associative algebra
morphism $K\rarrow\C^*$ obtained as the composition
$U(\bar\h)\rarrow U(\h)\rarrow\C^*$ satisfy the assumptions of
Corollary.
 The semialgebra $\S=\C\ot_KR$ is naturally isomorphic to
the semialgebra $\S^r=\C\ot_{U(\h)}U(\g)$ from
Section~\ref{tate-harish-chandra}.

 Hence the following description of the category of (Tate)
Harish-Chandra contramodules $\sO^\ctr(\g,\C)=\S^r\sicntr$ is
provided by Corollary.
 The category of Harish-Chandra contramodules $\P$ over $(\g,\C)$ is
isomorphic to the category of $\bar\g$\+modules whose underlying
$\bar\h$\+module structure comes from a $\C$\+contramodule structure.
 In particular, the forgetful functor $\sO^\ctr(\g,\C)\rarrow
\bar\g\modl$ is fully faithful.

 For example, let $\g=\Vir$ be the Virasoro Lie algebra,
$\h\subset\Vir$ be a compact open subalgebra contained in
the topological span of the generators $L_i$, \,$i\ge1$, and
$\bar\g=\mathrm{Vir}\subset\Vir$ be a dense subalgebra.
 Then $\h$ is a topologically finitely generated pro-nilpotent
Lie algebra, so the above considerations apply.
 Thus the forgetful functor $\sO^\ctr(\Vir,\C)\rarrow
\mathrm{Vir}\modl$ is fully faithful.
\end{ex2}

\bigskip\addtocontents{toc}{\smallskip}

\end{document}